\documentclass{amsart}
\usepackage{pinlabel}
\usepackage{color}
\usepackage{amsfonts,amsmath,amssymb,bbm,amsthm,amsbsy,amscd}
\usepackage{mathtools}
\usepackage{tkz-euclide}
\usepackage{caption}

\usepackage{thmtools} 
\usepackage{thm-restate}

\usepackage{color}
\usepackage{pdfcolmk}

\usepackage[margin=1in, heightrounded,
marginparwidth=23mm, marginparsep=2mm]{geometry}
\newcounter{jcomments}

\newcounter{acomments}

\newtheorem{thm}{Theorem}[section]
\newtheorem{prop}[thm]{Proposition}
\newtheorem{lemma}[thm]{Lemma}
\newtheorem{cor}[thm]{Corollary}
\newtheorem{claim}[thm]{Claim}

\theoremstyle{remark}
\newtheorem{rmk}[thm]{Remark}

\newtheorem{quest}[thm]{Question}

\theoremstyle{definition}
\newtheorem{defn}[thm]{Definition}
\newtheorem{exmp}[thm]{Example}

\newcommand{\Z}{\mathbb{Z}}
\newcommand{\TP}{\mathcal{TP}}
\newcommand{\B}{\mathcal{B}}
\newcommand{\R}{\mathbb{R}}
\newcommand{\C}{\mathcal{C}}
\newcommand{\Hbb}{\mathbb{H}}
\newcommand{\MCG}{\operatorname{MCG}}
\newcommand{\defeq}{\vcentcolon=}
\newcommand{\gammaup}[1]{\gamma^{(#1)}}
\newcommand{\alphaup}[1]{\alpha^{(#1)}}
\newcommand{\uga}{\underline{\gamma}}
\newcommand{\Ray}{\mathcal{R}}

\numberwithin{equation}{section}

\title{Laminations and 2-filling rays on infinite type surfaces}
\author{Lvzhou Chen}
\address{Department of Mathematics\\ The University of Texas at Austin \\ Austin, TX, USA}
\email[L.~Chen]{lvzhou.chen@math.utexas.edu}

\author{Alexander J. Rasmussen}
\address{Department of Mathematics \\ University of Utah \\ Salt Lake City, UT, USA}
\email[A.J.~Rasmussen]{rasmussen@math.utah.edu}

\begin{document}
	\maketitle
	
	\begin{abstract}
	The \emph{loop graph} of an infinite type surface is an infinite diameter hyperbolic graph first studied in detail by Juliette Bavard. An important open problem in the study of infinite type surfaces is to describe the boundary of the loop graph as a space of geodesic laminations. We approach this problem by constructing the first examples of \emph{$2$-filling rays} on infinite type surfaces. Such rays have strong filling properties while failing to correspond to points on the boundary of the loop graph. As such they may be thought of as ``fake boundary points.'' We give multiple constructions using both a hands-on combinatorial approach and an approach using train tracks and automorphisms of flat surfaces. In addition, our approaches are sufficiently robust to describe \emph{all} $2$-filling rays with certain other basic properties as well as to produce uncountably many distinct mapping class group orbits.
	\end{abstract}

\section{Introduction}
Mapping class groups of infinite type surfaces (so-called \emph{big mapping class groups}) have recently become an object of intense study, in part owing to their connections to dynamics and foliations of 3-manifolds. A key tool for studying big mapping class groups has been the \emph{loop graphs}. Introduced by Danny Calegari on his blog \cite{calblog} and first studied in detail by Juliette Bavard in \cite{ray}, the loop graph $L(S;p)$ of a surface $S$ with an isolated puncture $p$ is an infinite diameter hyperbolic graph (see \cite{ray} and \cite{simultaneous}). See below for the definition. The graph $L(S;p)$ is acted on by the subgroup $\MCG(S;p)$ of the mapping class group consisting of the mapping classes stabilizing $p$. It is in many ways analogous to the \emph{curve graph} $\C(S)$ of a finite-type surface $S$. Bavard in \cite{ray} and Bavard--Walker in \cite{simultaneous} have successfully applied the action $\MCG(S;p)\curvearrowright L(S;p)$ to study the \emph{second bounded cohomology} of $\MCG(S;p)$. Recently, Schaffer-Cohen has shown in \cite{sc} that $L(S;p)$ is an optimally strong model for the geometry of $\MCG(S;p)$ in the case that $S$ is the plane minus a Cantor set --- the two are quasi-isometric.

An obstacle to fully harnessing the power of the action $\MCG(S;p)\curvearrowright L(S;p)$ has been a non-trivial amount of mystery surrounding the Gromov boundary $\partial L(S;p)$. In the finite-type case, the Gromov boundary $\partial \C(S)$ may be identified with the space of \emph{ending laminations} on $S$ with the coarse Hausdorff topology (\cite{kla}). Ideally, in the infinite-type case one would like to have an analogous description of $\partial L(S;p)$ as a space of geodesic laminations. In this paper we shed some light on the problem of understanding $\partial L(S;p)$, while at the same time pointing to even more mystery than was previously known.

Our main goal in this paper is to prove the \textit{existence of 2-filling rays on infinite type surfaces}. We recall the definition. Let $S$ be an infinite type surface with an isolated puncture $p$ and fix a complete hyperbolic metric on $S$. A simple geodesic ray on $S$ is a \textit{loop} if it is asymptotic to $p$ at both ends. The graph $L(S;p)$ has as vertices the loops on $S$ and edges joining disjoint loops. It is a subgraph of a larger graph $\Ray(S;p)$, called the \textit{completed ray-and-loop graph}. We call a simple geodesic ray which is proper and asymptotic to $p$ at exactly one end a \textit{short} ray. We call a simple geodesic ray \textit{long} if it is asymptotic to $p$ and is neither short nor a loop. The vertex set of $\Ray(S;p)$ consists of all of the long rays, short rays, and loops on $S$ with edges joining disjoint pairs. From the definition, $L(S;p)$ is naturally a subgraph of $\Ray(S;p)$. It is shown in \cite{Bavard_Walker} and \cite{simultaneous} that the graph $\Ray(S;p)$ consists of uncountably many components. One of these components contains $L(S;p)$ and is quasi-isometric to it. The other components are \textit{cliques} of rays and we call the members of these cliques \textit{high-filling}. The Gromov boundary $\partial L(S;p)$ is naturally identified with the set of cliques of high-filling rays on $S$. From the definitions, a high-filling ray in particular is \textit{filling} in the sense that it intersects every short ray and loop.

We define a ray to be \textit{2-filling} if it is filling but not high-filling. Bavard--Walker showed in \cite{simultaneous} that this is equivalent to the following: the ray $\gamma$ is 2-filling if it intersects every loop, but is disjoint from a long ray $\tau$ which is in turn disjoint from a loop. Thus, 2-filling rays \emph{just slightly fail} to be high-filling. Wondering if all filling rays are high-filling, Bavard--Walker asked the following in \cite{Bavard_Walker}:

\begin{quest}{\cite[Question 2.7.7]{Bavard_Walker}}
Does there exist an example of a surface $S$ with an isolated puncture and a 2-filling ray on $S$?
\end{quest}

We answer their question in the \textit{positive} by explicitly constructing \textit{many} examples of \textit{2-filling rays} on $S$ when $S$ is the plane minus a Cantor set. 

\begin{thm}
There are uncountably many distinct mapping class group orbits of $2$-filling rays on $S$.
\end{thm}

This is in stark contrast to the finite-type case. In fact, when $S$ is of finite type a filling geodesic ray $\gamma$ necessarily accumulates onto a minimal lamination filling a subsurface of $S$ containing the puncture. In our situation, the limiting laminations are in some sense filling and have many dense leaves, yet still do not have strong enough properties to correspond to points of $\partial L(S;p)$.

We give two different approaches to the construction of $2$-filling rays. One is a hands-on combinatorial approach (Section \ref{sec:straightforward}). The other uses geodesic laminations, train tracks, and flat surface automorphisms (Sections \ref{sec:abstractfol} -- \ref{sec:laminations}). In Section \ref{sec:correspondence} we show that these two different approaches actually produce the same $2$-filling rays.

In Lemma \ref{lem:cliques} we show that $2$-filling rays naturally lie in cliques of mutually disjoint rays, some of which are $2$-filling and some of which are not $2$-filling. One may naturally wonder then what the structure of these cliques may be and in particular if there is any constraint on the number of $2$-filling rays and non-$2$-filling rays in a given clique. We give partial answers to this question. In Section \ref{sec:straightforward} we construct cliques containing \textit{any given finite number} of $2$-filling rays together with \textit{exactly one} non-$2$-filling ray. In Section \ref{sec:mnf} we construct for each $n\geq 1$ a clique --- this time on a surface with $2n$ nonplanar ends --- consisting of \textit{exactly $n$} $2$-filling rays and \textit{exactly $n$} non-$2$-filling rays.

Finally, our constructions are quite robust. In Section \ref{sec: 2-fill abound}, we show that there are uncountably many distinct mapping class group orbits of $2$-filling rays on the plane minus a Cantor set. In Section \ref{sec: uniqueness}, we show that \textit{any} finite clique of $2$-filling rays disjoint from a single non-$2$-filling ray arises from the construction of Section \ref{sec:straightforward}. 

At the end of the paper we give a list of open problems that we hope might help to guide future research into geodesic rays on infinite type surfaces and the boundary of the loop graph.

\subsection{Organization of the paper} In Section \ref{sec: rays} we introduce the basic concepts and set up notation. We introduce \emph{two-side approachable long rays} and establish fundamental results in Section \ref{sec: two-side}, and then use them to give the first construction of $2$-filling rays in Section \ref{sec:straightforward}. We further show that $2$-filling rays abound in Section \ref{sec: 2-fill abound} and prove that the construction is sometimes unique in Section \ref{sec: uniqueness}.

The second construction is given in Sections \ref{sec:tracks}--\ref{sec:laminations}. Background and basic concepts of this part are given in Section \ref{sec:lambg}. This construction relies on a geodesic lamination with desired properties listed in Section \ref{sec:tracks}. We start by constructing an abstract weighted train track in Section \ref{sec:abstractfol} and establishing properties of the associated foliation in Sections \ref{sec:pA} and \ref{sec:trainpaths}. Then in Section \ref{sec:laminations} we embed this train track in the plane minus a Cantor set to obtain the desired geodesic lamination. A nice property of the embedding that we need is proved in the appendix (Section \ref{appendix}).

In Section \ref{sec:correspondence} we show the correspondence of the two constructions. In Section \ref{sec:mnf} we give an example to show that there could be more than one non-filling ray disjoint from a $2$-filling ray on a surface of infinite genus.

Finally in Section \ref{sec: open} we give a list of open questions about $2$-filling rays.

\subsection*{Acknowledgements.} We thank the following people for invaluable conversations about laminations and rays on infinite type surfaces: Danny Calegari, Yair Minsky, Rodrigo Trevi\~{n}o, Pat Hooper, Ethan Farber, Kathryn Lindsey, and Anja Randecker. Special thanks to Yan Mary He and Kasra Rafi for insightful questions and suggestions.

	\section{Background on rays and loops}
	\label{sec: rays}
	
	Throughout this paper, let $\Omega$ be the plane minus a Cantor set. We often think of $\Omega$ as the sphere with a Cantor set and another isolated point $\infty$ removed.
	Fix an orientation on $\Omega$. We equip it with a complete hyperbolic metric in the following way. Choose a pants decomposition of $\Omega$ and let $\mathcal{P}$ be the resulting set of pants curves. If $c_1,c_2,c_3\in \mathcal{P}$ bound a pair of pants then we equip it with the unique complete hyperbolic metric in which the pants curves are geodesics of length one. If $c_1$ and $c_2$ bound a pair of pants together with $\infty$ then we equip it with the unique complete hyperbolic metric in which $c_1$ and $c_2$ are geodesics of length one. The resulting pairs of pants may be glued together by any desired isometries. 
	
	We are interested in geodesic rays starting at $\infty$. A ray is \emph{simple} if it does not self-intersect. Simple rays fall into three classes: loops, short rays, and long rays.
	Here a \emph{loop} is an oriented geodesic ray that starts and ends at $\infty$. 
	We assume loops to be simple unless stated otherwise.	
	The same geodesic ray with reversed orientation is a different loop. We use $\bar{\gamma}$ to denote the geodesic $\gamma$ with reversed orientation.
	A \emph{short ray} is a proper simple ray that escapes to a certain point in the Cantor set of ends.
	All simple rays other than loops and short rays are non-proper, i.e. have nontrivial limit sets in $\Omega$. We refer to these rays as \emph{long rays}.
	As usual, the limit set of a geodesic $\gamma$ is $\operatorname{cl}(\gamma)\setminus \gamma$, where $\operatorname{cl}(\gamma)$ denotes the closure of $\gamma$. See the lower half of Figure \ref{fig: conical} for examples.
		
	It is convenient to think of a loop (respectively a short ray) topologically as an isotopy class of embedded arcs on $S^2$ that are disjoint from the Cantor set and $\infty$ in their interiors and that go from $\infty$ to $\infty$ (resp. from $\infty$ to some point in the Cantor set). The mapping class group of $\Omega$ acts transitively on the set of loops and the set of short rays. In contrast, there is a continuum of different orbits of long rays.
	
	The \emph{conical circle} $S^1_C$, which we now define, is a space naturally parameterizing all geodesic rays starting at $\infty$. The \emph{conical cover} $\Omega_C$ is the covering space of $\Omega$ corresponding to the $\Z$ subgroup of $\pi_1(\Omega)$ generated by a simple closed curve around $\infty$. It inherits a hyperbolic metric from that of $\Omega$, by pullback. Its Gromov boundary consists of a disjoint union of a point and a circle. The \emph{conical circle} $S^1_C$ is this circle boundary component of $\Omega_C$; See the upper right of Figure \ref{fig: conical}.
	
	In other words, if $\widetilde{\infty}$ is the fixed point on $\partial \Hbb^2$ of a generator $z$ of the $\Z$ subgroup, and $\tilde{r}$ is a lift of any fixed ray $r$ to $\Hbb^2$ that starts at $\widetilde{\infty}$, then the region between $\tilde{r}$ and $z\tilde{r}$
	is a fundamental domain of $\Omega_C$. The corresponding segment on $\partial\Hbb^2$ with two endpoints identified is a copy of the conical circle $S^1_C$.
	
	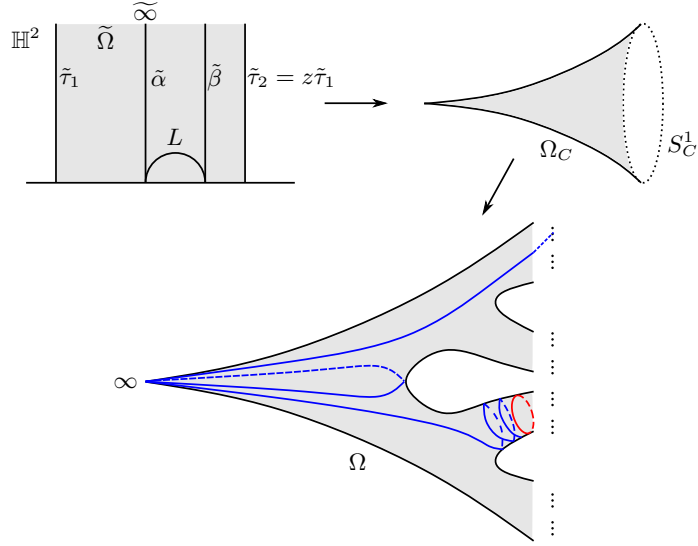
\begin{figure}
		\labellist
		\small 
		\pinlabel $\infty$ at 75 120
		\pinlabel $\Omega$ at 250 60
		\pinlabel $\Omega_C$ at 400 295
		\pinlabel $S^1_C$ at 495 300
		\pinlabel $\Hbb^2$ at 0 380
		\pinlabel $\widetilde{\Omega}$ at 60 380
		\pinlabel $\widetilde{\infty}$ at 90 400
		\pinlabel $\tilde{\tau}_1$ at 33 350
		\pinlabel $\tilde{\tau}_2=z\tilde{\tau}_1$ at 200 350
		\pinlabel $\tilde{\alpha}$ at 100 350
		\pinlabel $\tilde{\beta}$ at 142 350
		\pinlabel $L$ at 112 305
		\endlabellist
		\centering
		\includegraphics[scale=0.5]{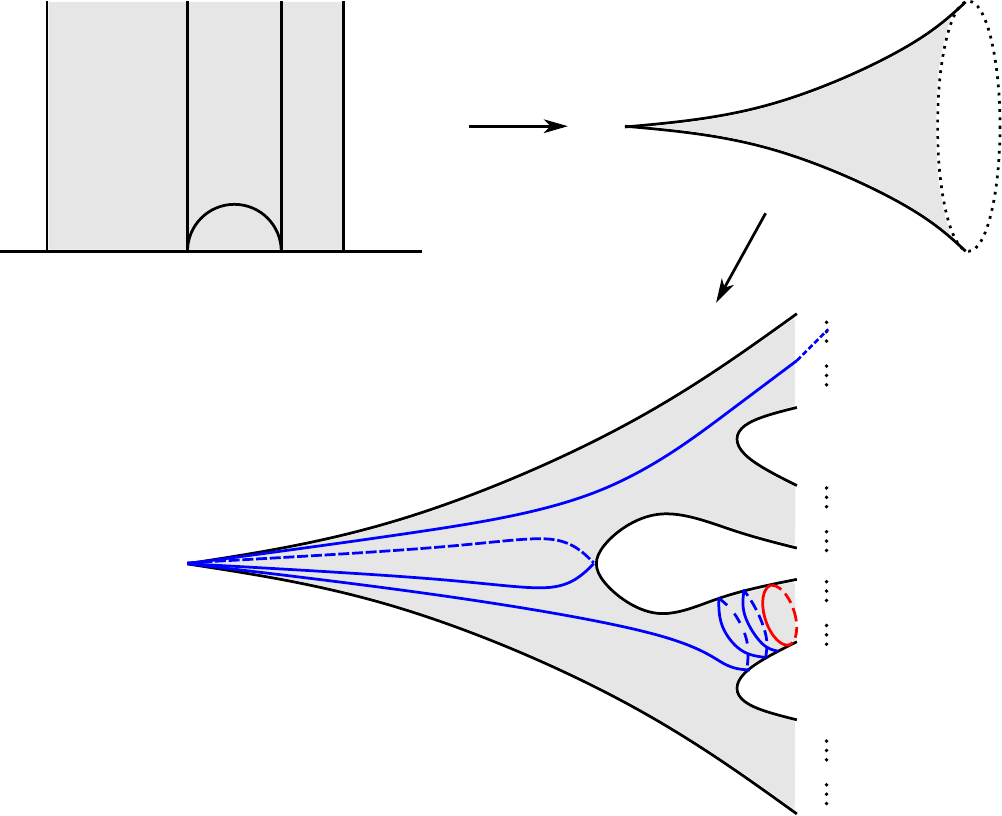}
		\caption{Bottom: Examples of a short ray, a loop, and a long ray in $\Omega$; Upper right: The conical cover and conical circle; Upper left: A fundamental domain $\widetilde{\Omega}$ of $\Omega_C$ and the geodesic $L$.}\label{fig: conical}
	\end{figure}
	
	We often think of a simple ray as a point on the conical circle. In particular we have a topology on the set of rays, where two rays are close if they fellow travel for a long time in the beginning. 
	In this topology, the property of a ray having a (transverse) self-intersection is open. Thus the set of simple rays is closed on $S^1_C$, and is also nowhere dense \cite[Lemma 3.3]{calegari_chen}.
	
	For any simple ray $\tau$, the set of rays disjoint from $\tau$ (including itself) is a \emph{closed} set $D_\tau$ on the conical circle. For each complementary interval $(\alpha,\beta)$, there is a corresponding bi-infinite geodesic $p(L)$ on $\Omega$ going from the end of $\alpha$ to the end of $\beta$. More precisely, pick a lift $\widetilde{\infty}$ of $\infty$ on $\partial \Hbb^2$ and two consecutive lifts $\tilde{\tau}_1,\tilde{\tau}_2$ of $\tau$ starting at $\widetilde{\infty}$, which bound a fundamental domain $\widetilde{\Omega}$ of $\Omega_C$ on $\Hbb^2$.	
	Let $\tilde{\alpha},\tilde{\beta}$ be the unique lifts of $\alpha,\beta$ in $\widetilde{\Omega}$ starting at $\widetilde{\infty}$. Let $L$ be the unique bi-infinite geodesic on $\widetilde{\Omega}$ going from the endpoint of $\tilde{\alpha}$ to that of $\tilde{\beta}$; See the upper left of Figure \ref{fig: conical}. Then $p(L)$ is the projection of $L$ on $\Omega$.	
	\begin{lemma}\label{lemma: converge to L}
		There are lifts of $\tau$ converging to $L$. In particular, the geodesic $p(L)$ is simple on $\Omega$.
		If neither $\alpha$ nor $\beta$ is a loop, then $L$ cannot be a lift of $\tau$, and $p(L)$ lies in the limit set of $\tau$.
	\end{lemma}
	\begin{proof}
		Let $\widetilde{\Omega}$, $\tilde{\alpha}$ and $\tilde{\beta}$ be as above. Let $a$ and $b$ be the endpoints of $\tilde{\alpha}$ and $\tilde{\beta}$ respectively.
		Let $\Omega_{\alpha\beta}$ be the sector bounded by $\tilde{\alpha}$ and $\tilde{\beta}$ in the fundamental domain for $\Omega_C$. Then $(\alpha,\beta)$ corresponds to the boundary of $\Omega_{\alpha\beta}$.
		For each lift of $\tau$ contained in $\Omega_{\alpha\beta}$, its two ends bound an open sub-interval of $(\alpha,\beta)$.
		Any two such sub-intervals are either disjoint or nested since $\tau$ is simple.
		The union of such open sub-intervals is the entire $(\alpha,\beta)$.
		Indeed, for any $\gamma\in(\alpha,\beta)$, it has a lift $\tilde{\gamma}$ starting at $\widetilde{\infty}$ which lies in $\Omega_{\alpha\beta}$. 
		The lift $\tilde{\gamma}$ must intersect some lift $\tilde{\tau}$ of $\tau$, and any such $\tilde{\tau}$ lies in $\Omega_{\alpha\beta}$ since $\tau$ is disjoint from $\alpha,\beta$.
		The open sub-interval corresponding to $\tilde{\tau}$ contains $\gamma$.
		
		Unless $L$ itself is a lift of $\tau$, no such open sub-interval is maximal, and there is an increasing nested sequence of them converging to $(\alpha,\beta)$. 
		This proves the first assertion. Thus $p(L)$ is simple since $\tau$ is. If neither $\alpha$ nor $\beta$ is a loop, the endpoints of $\tilde{\alpha},\tilde{\beta}$ are not lifts of $\infty$,
		and thus $L$ cannot be a lift of $\tau$, so $p(L)$ lies in the limit set of $\tau$.
	\end{proof}
	
	Given a ray $\tau$, the circular order on $S^1_C$ 
	induces a total order $<$ on $S^1_C\setminus\{\tau\}$, where $x<y$ if and only if $(x,y,\tau)$ is positively oriented on $S^1_C$. We say a sequence of rays $\tau_n$ converges to a given ray $\tau$ \emph{from its left} if $\tau_n$ eventually converges to the left side (i.e. the small side under the order $<$) of $S^1_C\setminus\{\tau\}$. Alternatively, if an ant is moving on $\tau$ in the positive direction, then it will see $\tau_n$ converging to $\tau$ from its left-hand side.
	Convergence from the right is defined similarly.

	There is another space related to simple rays, namely the \emph{completed ray-and-loop graph} $\Ray$. It is the graph whose vertices correspond to the simple rays and loops on $\Omega$ and whose edges join disjoint geodesics. It is shown by Bavard--Walker \cite[Theorem 2.8.1]{Bavard_Walker} that $\Ray$ has a connected component (which we call major) containing all loops and short rays, which is $\delta$-hyperbolic and infinite diameter.
	Each other component is a clique (i.e. a complete subgraph), and such cliques correspond to points on the Gromov boundary of the major component (and hence also to points on the Gromov boundary of the loop graph $L(\Omega;\infty)$).
	
	Each ray outside the major component is called \emph{high-filling}. Here a ray is (loop- and ray-) \emph{filling} if it intersects all loops and short rays.
	It is known that any filling ray $\gamma$ is either high-filling or has distance at most $2$ to some loop. See \cite[Lemma 2.7.6]{Bavard_Walker}.
	We say that a long ray is \emph{$2$-filling} if its minimal distance to the set of loops and short rays on $\Ray$ is $2$. In other words, a long ray is 2-filling if it intersects every loop and short ray, but is disjoint from some long ray which is in turn disjoint from a short ray or loop.
	
	The following lemma is the analog for $2$-filling rays of the fact that any component in the ray-and-loop graph $\Ray$ containing a high-filling ray is a clique.
	\begin{lemma}\label{lem:cliques}	
		For any $2$-filling ray $\gamma$, its star in the ray-and-loop graph $\Ray$ is a clique. In addition, all $2$-filling rays in the clique have the same star. 
	\end{lemma}
		\begin{proof}
		Since $\gamma$ is $2$-filling, any point in the star represents a long ray. Let $r_1,r_2$ be two long rays disjoint from $\gamma$. Then $r_1$ and $r_2$ are disjoint by \cite[Lemma 2.7.4]{Bavard_Walker}. If $r_1$ is also $2$-filling, then this shows that any $r_2$ disjoint from $\gamma$ is also disjoint from $r_1$ and vice versa, thus $\gamma$ and $r_1$ have the same star.
	\end{proof}
	
	Finally, we introduce a piece of notation. Let $\alpha$ be an oriented ray or loop and $p,q\in \alpha$ with $p<q$ in the orientation on $\alpha$ (possibly with $p=\infty$ and/or $q=\infty$). Then we denote by $\alpha|[p,q]$ the subarc of $\alpha$ which is oriented from $p$ to $q$. If $\alpha$ is a simple compact arc and $p,q\in \alpha$ then $\alpha|[p,q]$ similarly denotes the subarc of $\alpha$ between $p$ and $q$.

\section{Background on train tracks, laminations, and foliations}
\label{sec:lambg}

For us, a train track will denote a locally finite graph with the following structure. At any vertex $v$ the set $\B(v)$ of incident edges at $v$ is partitioned into nonempty sets $\B_i(v)$ and $\B_o(v)$ which we call \textit{incoming} and \textit{outgoing}, respectively. Moreover, the sets $\B_i(v)$ and $\B_o(v)$ carry total orders $<_i$ and $<_o$, respectively. If $T$ is a train track then the vertices of $T$ will be called \textit{switches} and the edges will be called \textit{branches}. The set of branches of $T$ will be denoted by $\B(T)$. A \textit{train path} on $T$ is a (finite or infinite) edge path on $T$ with the following property. Any two consecutive branches are incident to a common switch $v$ and we require one of the branches to be incoming and the other to be outgoing at $v$. As usual, we may consider $T$ as a topological 1-complex endowed with the structure of a smooth manifold away from the switches and at any switch $v$ the structure of a well-defined tangent line so that 
\begin{itemize}
\item all of the branches incident to $v$ are tangent,
\item if all incident branches are oriented to point towards $v$, then the tangent vectors to the incoming (respectively outgoing) branches all point in the same direction,
\item the tangent vectors to the incoming branches and outgoing branches point in opposite directions.
\end{itemize} A train path on $T$ is then a smooth immersion of an interval into $T$.

A \textit{weight system} on $T$ is a function $w:\B(T)\to \R_{\geq 0}$ with the following property. If $v$ is a switch then we have that the sum of the weights of the incoming branches incident to $v$ is equal to the sum of the weights of the outgoing branches incident to $v$. We will call a \textit{weighted train track} a pair $(T,w)$ where $T$ is a train track and $w$ is a weight system on $T$.

Associated to a weighted train track $(T,w)$ there is a \textit{union of foliated rectangles} defined as follows. For each branch $b\in \B(T)$ we consider the rectangle $R(b)=[0,1]\times [0,w(b)]$. These rectangles are glued together as follows. Any switch $v$ defines an interval $I(v)=[0,\ell]$ where we set \[\ell=\sum_{b\in \B_i(v)} w(b)=\sum_{b\in \B_o(v)} w(b).\] If $b_1<_o b_2 <_o \ldots <_o b_n$ are the outgoing edges at $v$ then $I(v)$ is divided into consecutive closed subintervals $I_1,\ldots, I_n$ of lengths $w(b_1),\ldots,w(b_n)$, respectively and where $0\in I_1$. 
The left vertical side $\{0\}\times [0,w(b_i)]$ of $R(b_i)$ is glued isometrically to the interval $I_i$. Similarly, the \textit{right} vertical sides of the rectangles corresponding to the incoming branches at $v$ are glued to $I(v)$ isometrically according to the total order $<_i$.

Denote the union of foliated rectangles by $G$. Each rectangle $R(b)$ of $G$ is foliated by the horizontal line segments $[0,1]\times \{h\}$ for $h\in [0,w(b)]$. This endows $G$ with the structure of a \textit{singular foliation}. That is, $G$ is foliated by horizontal lines away from a discrete set of points (the \textit{singularities} of the foliation, where at least three rectangles meet) and its \textit{boundary} $\partial G$ which is defined to be the union of the horizontal boundary components $[0,1]\times \{0\}$ and $[0,1]\times \{w(b)\}$ of the rectangles $R(b)$. For the rest of the discussion, endow the topological space underlying $G$ with an orientation.

A \textit{saddle connection} of $G$ is an embedding of a compact interval into $G$ which is a union of horizontal line segments, having singularities at its endpoints and no singularities in its interior. A \textit{singular ray} of $G$ is an embedding of the half-line $[0,\infty)$ into $G$ which is a union of horizontal line segments with a singularity at its endpoint and no singularity in its interior.

A \textit{leaf} $l$ of $G$ is an embedding of $\R$ into $G$ which is the union of a sequence \[\ldots \sigma_{-2} \sigma_{-1} \sigma_0 \sigma_1 \sigma_2 \ldots\] of horizontal line segments of $G$ and satisfies the following properties. First of all, each $\sigma_i$ traverses a rectangle $R_{b_i}$ in $G$ and we require that \[\ldots b_{-2} b_{-1} b_0 b_1 b_2 \ldots\] is a train path on $T$. Secondly, there is a choice of \textit{left} or \textit{right} (assume left for simplicity) such that the following condition is satisfied. Suppose that some $\sigma_i$ has endpoints $p_i\in R_{b_{i-1}}\cap R_{b_i}$ and $q_i \in R_{b_i} \cap R_{b_{i+1}}$. Suppose that $\sigma_i$ is contained in the interior of $R_{b_i}$ but that $q_i$ is a singularity. Then the rectangle $R_{b_{i+1}}$ traversed by $\sigma_{i+1}$ lies to the \textit{left} at $q_i$ as we traverse $\sigma_i$ from $p_i$ to $q_i$. If on the other hand, $\sigma_i$ is contained in the interior of $R_{b_i}$ but $p_i$ is a singularity, orienting $\sigma_i$ from $q_i$ to $p_i$, we have that $R_{b_{i-1}}$ lies to the \textit{right} at $p_i$. A \emph{half leaf} of $l$ is an equivalence class of rays contained in $l$, where two sub-rays are considered to be equivalent if their symmetric difference is compact.

A leaf will be called \textit{singular} if it contains a singularity and \textit{non-singular} otherwise. The leaves of $G$ define train paths on $T$ and we denote by $\TP(T,w)$ the resulting set of train paths. If $t$ is the train path defined by some singular leaf, we will call it a \textit{boundary path}.

We say that a train path $t\in \TP(T,w)$ \emph{accumulates onto} the path $t'\in \TP(T,w)$ if every finite sub-train path $b_1\ldots b_k$ of $t'$ is contained in $t$. We say that $t$ is \emph{dense} in $\TP(T,w)$ if it accumulates onto every $t'\in \TP(T,w)$. There is a (typically non-Hausdorff) topology on $\TP(T,w)$ with sub-basis consisting of all sets of the form \[\{t\in \TP(T,w): b_1\cdots b_k \text{ is contained in } t\}\] where $b_1\cdots b_k$ is a finite train path on $T$. With this topology, $t$ accumulates onto $t'$ if and only if every neighborhood of $t'$ contains $t$.


Finally, we define a \emph{flat surface} to consist of the following data:
\begin{itemize}
\item a topological surface $\Sigma$;
\item a countable closed subset $P$ of $\Sigma$;
\item an atlas of charts from open subsets $U\subset \Sigma \setminus P$ to $\mathbb{C}$ such that all transition functions between these charts have the form $z\mapsto \pm z+c$ in coordinates, where $c\in \mathbb{C}$ is a constant.
\end{itemize}
The surface $\Sigma \setminus P$ inherits a Euclidean metric. This metric is typically incomplete and we require it to extend to $P$, so that $P$ is identified with a subset of the completion of $\Sigma \setminus P$. The points of $P$ are called \textit{singularities} of the flat surface $\Sigma$.

The surface $\Sigma$ also inherits a \textit{horizontal foliation} $\mathcal{F}^h$ defined as follows. If $U \subset \Sigma \setminus P$ and $\phi:U\to \mathbb{C}$ is one of the charts defined above, then the line segments $\operatorname{Im}(z)=y$ in $\phi(U)$ pull back to a family of line segments on $U$. The leaves of $\mathcal{F}^h$ are the maximal concatenations of such line segments. Similarly, the vertical line segments $\operatorname{Re}(z)=x$ in $\phi(U)$ pull back to line segments on $U$ and the leaves of the \textit{vertical foliation} $\mathcal{F}^v$ are the maximal concatenations of these line segments. Since all transition functions have the form $z\mapsto \pm z+c$, both $\mathcal{F}^h$ and $\mathcal{F}^v$ are well-defined and are indeed foliations of $\Sigma$.

\section{Two-side approachable long rays}\label{sec: two-side}
In this section we introduce the so-called \emph{two-side approachable} long rays. Given such a long ray, we will construct $2$-filling rays disjoint from it in the next section.

\begin{defn}
	A long ray $\tau$ is \emph{two-side approachable} if there are loops $\ell_i$ and $r_i$ converging to $\tau$ such that $\ell_i$ converges to $\tau$ from the left, $r_i$ converges to $\tau$ from the right, and such that $\ell_i$ and $r_i$ are all disjoint from $\tau$.
\end{defn}

\begin{exmp}
	Figure \ref{fig: two_side} depicts a simple example of a two-side approachable long ray $\tau$ that spirals and limits to a geodesic arc $\alpha$ connecting two points in the Cantor set. There are geodesics that follow $\tau$ for a long time and then turn around to the left to form a loop disjoint from $\tau$ and slightly to the left of $\tau$. See the loop $\ell$ in Figure \ref{fig: two_side}. Similarly there are loops disjoint from $\tau$ and sightly to the right of $\tau$. Thus $\tau$ is indeed two-side approachable.
	
	See Section \ref{sec: 2-fill abound} for more complicated examples, where we construct a continuum of mapping class group orbits of two-side approachable long rays. 
\end{exmp}

\begin{figure}
	\labellist
	\small 
	\pinlabel $\infty$ at 72 -5
	\pinlabel $\alpha$ at 72 100
	\pinlabel $\tau$ at 123 28
	\pinlabel $\ell$ at 92 28
	\endlabellist
	\centering
	\includegraphics[scale=0.8]{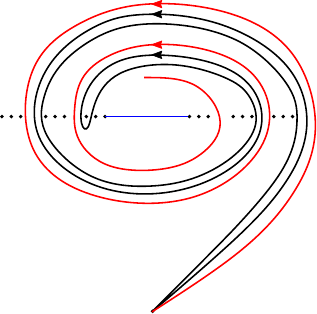}
	\caption{A two-side approachable long ray $\tau$ with a disjoint loop $\ell$ slightly on its left}\label{fig: two_side}
\end{figure}
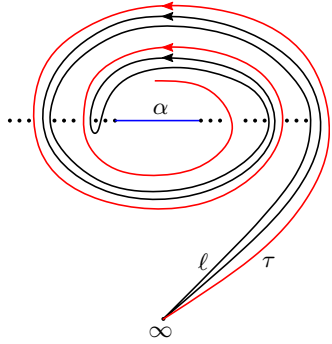

Given a two-side approachable long ray $\tau$, a priori the loops $\ell_i,r_i$ as in the definition might intersect each other, but one can apply surgeries to make them pairwise disjoint and have other nice properties for our construction in the next section.

To state these properties, let $I_\tau$ be the closed interval obtained from cutting the conical circle $S^1_C$ at $\tau$. Recall that the circular order on $S^1_C$ induces a total order $<$ on $I_\tau$, where $x<y$ if and only if $(x,y,\tau)$ is positively oriented. Then $\{\ell_i\}$ and $\{r_i\}$ represent sequences on $I_\tau$ converging to the left and right endpoints of $I_\tau$, respectively.

We refer to the component of $\Omega\setminus\ell_i$ (resp. $\Omega\setminus r_i$) not containing $\tau$ as the \emph{interior} of $\ell_i$ (resp. $r_i$), and refer to the other component as the \emph{exterior}. We will further make $\ell_i$ decreasing, $r_i$ increasing, and together satisfy
\begin{equation}\label{eqn: order of loops}
	\cdots < \ell_i < \bar{\ell}_i < \cdots < \ell_1 < \bar{\ell}_1 < \bar{r}_1 < r_1 < \cdots < \bar{r}_i < r_i <\cdots,
\end{equation}
where $\bar{\ell}_i$ and $\bar{r}_i$ represent loops $\ell_i$ and $r_i$ with the reversed orientation respectively. Geometrically, given that the loops $\ell_i$ and $r_i$ are disjoint, the order guarantees them to have mutually disjoint interiors. See Figure \ref{fig: loops_std_form}. 

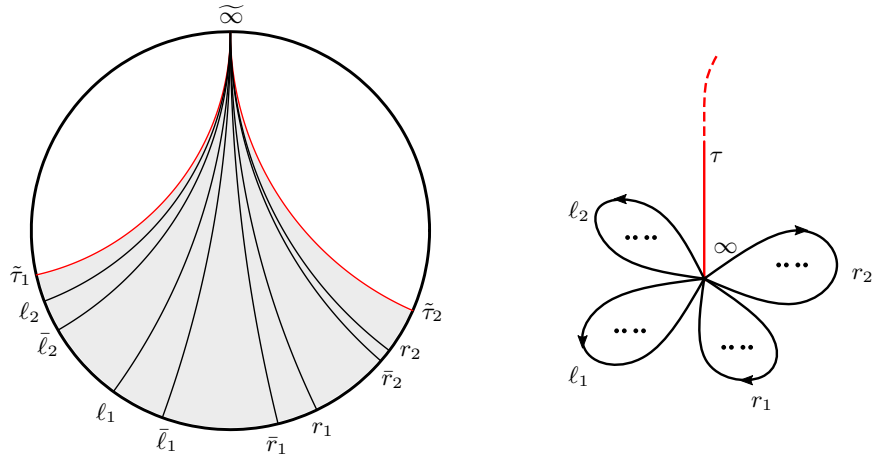
\begin{figure}
	\labellist
	\small 
	\pinlabel $\widetilde{\infty}$ at 110 228
	\pinlabel $\tilde{\tau}_1$ at -5 85
	\pinlabel $\ell_2$ at 0 65
	\pinlabel $\bar{\ell}_2$ at 10 47
	\pinlabel $\ell_1$ at 42 10
	\pinlabel $\bar{\ell}_1$ at 75 -5
	\pinlabel $\tilde{\tau}_2$ at 220 65
	\pinlabel $r_2$ at 208 43
	\pinlabel $\bar{r}_2$ at 198 28
	\pinlabel $r_1$ at 160 1
	\pinlabel $\bar{r}_1$ at 135 -8
	
	\pinlabel $\tau$ at 375 150	
	\pinlabel $\infty$ at 380 100
	\pinlabel $\ell_2$ at 300 120	
	\pinlabel $\ell_1$ at 300 30
	\pinlabel $r_1$ at 400 15
	\pinlabel $r_2$ at 455 85
	\endlabellist
	\centering
	\includegraphics[scale=0.7]{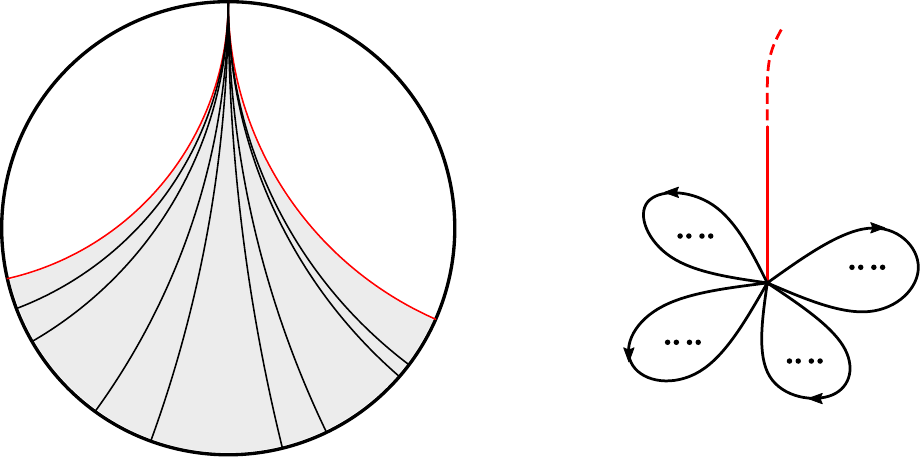}
	\caption{On the right we have loops $\ell_i,r_i$ with disjoint interiors that converge to and are disjoint from a two-side approachable long ray $\tau$, where black dots indicate Cantor subsets. The figure on the left depicts their lifts to $\Hbb^2$, where $\tilde{\tau}_1$ and $\tilde{\tau}_2$ are consecutive lifts of $\tau$.} \label{fig: loops_std_form}
\end{figure}

The surgery to promote $\ell_i$ and $r_i$ is based on the following lemmas. Consider two geodesics $r_1$ and $r_2$ intersecting transversely at a point $p$. Let $r$ be a piecewise geodesic which first traverses $r_1$ to the point $p$ and then traverses $r_2$. We say that $r$ is making a \emph{right} (resp. \emph{left}) turn if the positive unit vector of $r_2$ at $p$ is on the right (resp. left) of $r_1$. See the left of Figure \ref{fig: surgery} for an example of a right turn, where we further straighten $r$ to a geodesic.

\begin{lemma}\label{lemma: concat at intersection}
	Let $\tau$ be a simple ray, and let $r_1,r_2$ be geodesics disjoint from $\tau$ so that $r_1$ starts from $\infty$. Let $r$ be the straightening of a piecewise geodesic that first follows $r_1$ to an intersection $p$ of $r_1$ and $r_2$ and then follows $r_2$. Then $r$ is disjoint from $\tau$ and $r<r_1$ (resp. $r>r_1$) on $I_\tau$ if $r$ makes a right (resp. left) turn at $p$. Moreover, 
	\begin{enumerate}
		\item if $r_1,r_2$ are simple and the initial arc of $r_1$ up to $p$ does not intersect $r_2$ then $r$ is simple;
		\item a bi-infinite geodesic $\ell$ is disjoint from $r$ if it is disjoint from $r_1$ and $r_2$;
		\item if $\ell\notin \{ \tau,r,r_1\}$ is a geodesic starting from $\infty$ that does not intersect $r_2$ transversely, then $r>\ell$ on $I_\tau$ if and only if $r_1>\ell$.
	\end{enumerate}		
\end{lemma}
\begin{proof}
	Consider a fundamental domain $\widetilde{\Omega}$ of the conical cover in the universal cover with boundary geodesics being consecutive lifts of $\tau$ starting at the same lift $\widetilde{\infty}$ of $\infty$. See the left of Figure \ref{fig: surgery}. Let $\tilde{r}_1$ be the unique lift of $r_1$ in $\widetilde{\Omega}$ starting at $\widetilde{\infty}$, and let $\tilde{r}_2$ be the unique lift of $r_2$ intersecting $\tilde{r}_1$ at the unique lift of $p$ along $\tilde{r}_1$. Then $\tilde{r}_2$ stays in $\widetilde{\Omega}$ since $r_2$ is disjoint from $\tau$. Now a lift $\tilde{r}$ of $r$ is given by the third side of the geodesic triangle with two sides on $\tilde{r}_1,\tilde{r}_2$ shown in Figure \ref{fig: surgery}. Note that any infinite geodesic intersecting $\tilde{r}$ must intersect one of the other two sides of the geodesic triangle. The result easily follows from this.
\end{proof}

\begin{figure}
	\labellist
	\small 
	\pinlabel $\widetilde{\infty}$ at 110 228
	\pinlabel $\tilde{\tau}_1$ at -5 85
	\pinlabel $\widetilde{\Omega}$ at 40 60
	\pinlabel $\tilde{r}$ at 90 100
	\pinlabel $\tilde{\tau}_2$ at 220 65
	\pinlabel $\tilde{r}_1$ at 137 100
	\pinlabel $\tilde{r}_2$ at 117 57
	\pinlabel $p$ at 147 57
	
	\pinlabel $\widetilde{\infty}$ at 425 228
	\pinlabel $\widetilde{\Omega}$ at 355 60
	\pinlabel $\tilde{r}$ at 405 105
	\pinlabel $\overline{\tilde{r}_1}$ at 475 75
	\pinlabel $\tilde{\tau}$ at 495 62
	\pinlabel $\tilde{\tau}'$ at 452 25
	\pinlabel $\tilde{r}_1$ at 460 105
	\pinlabel $\tilde{r}_2$ at 432 57
	\pinlabel $\tilde{\tau}_1$ at 310 85
	\pinlabel $\tilde{\tau}_2$ at 535 65
	\pinlabel $\widetilde{\infty}_1$ at 500 15
	\endlabellist
	\centering
	\includegraphics[scale=0.6]{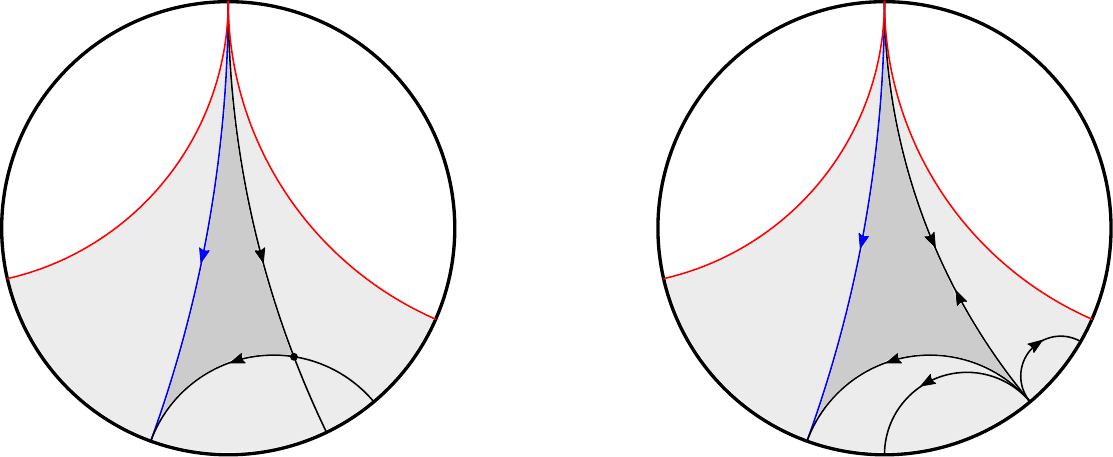}
	\caption{The figure on the left shows the concatenation of $r_1,r_2$ at their intersection $p$ by making a right turn, and $r$ is the straightening. The figure on the right shows the concatenation of $r_1,r_2$ at $\infty$ locally disjoint from $\tau$ by making a right turn, and $r$ is the straightening.}\label{fig: surgery}
\end{figure}
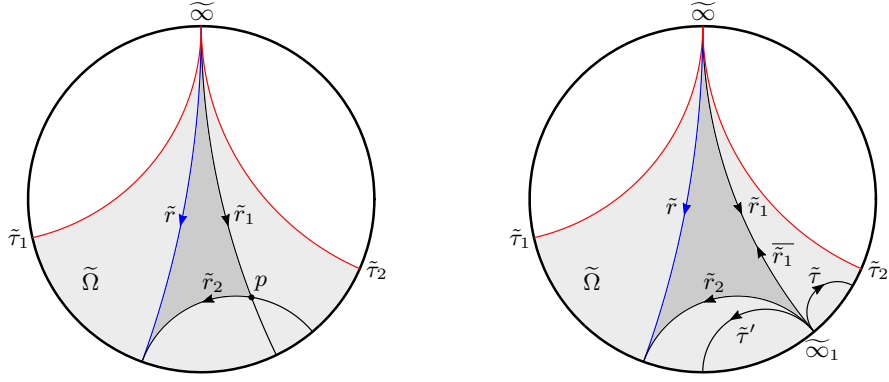

For two loops $r_1$ and $r_2$ disjoint from a simple ray $\tau$, there is a unique (possibly non-simple) loop $r$ whose homotopy class represents their product in $\pi_1(\widehat{\Omega},\infty)$ (where $\widehat{\Omega}$ denotes the filled-in surface $\Omega\cup \{\infty\}$) such that $r$ is disjoint from $\tau$ near $\infty$. An example is shown in Figure \ref{fig: 2fill_alpha2gamma} where $\gammaup{2}_{2k}$ is the straightened concatenation $\alphaup{2}_{2k}\cdot\alphaup{3}_{2k}$. We say $r$ is making a \emph{right} (resp. \emph{left}) turn if $r_2>\bar{r}_1$ (resp. $r_2<\bar{r}_1$) on $I_\tau$; see the right side of Figure \ref{fig: surgery} for an illustration of a right turn on the universal cover.
We have the following analog of Lemma \ref{lemma: concat at intersection} for this kind of surgery.
\begin{lemma}\label{lemma: concat at infty}
	Let $r_1$ and $r_2$ be loops disjoint from $\tau$, and let $r$ be the straightening of the unique concatenation of $r_1,r_2$ at $\infty$ locally disjoint from $\tau$. 
	Then $r$ is (globally) disjoint from $\tau$, and we have $r<r_1$ (resp. $r>r_1$) on $I_\tau$ if $r$ makes a right (resp. left) turn at the concatenation. Moreover, 
	\begin{enumerate}
		\item if $r_1,r_2$ are disjoint and $r_2,\bar{r}_1$ are adjacent among the four points $r_1,\bar{r}_1,r_2,\bar{r}_2$ on $I_\tau$, then $r$ is simple;
		\item a bi-infinite geodesic $\ell$ is disjoint from $r$ if it does not intersect $r_1$ or $r_2$ transversely and is not a ray starting from $\infty$ so that $\ell$ is between $\bar{r}_1$ and $r_2$ on $I_\tau$; 
		\item if $\ell\notin \{\tau,r,r_1\}$ is a geodesic starting from $\infty$ disjoint from $r_2$, then $r>\ell$ on $I_\tau$ if and only if $r_1>\ell$.
	\end{enumerate}
\end{lemma}
\begin{proof}
	The proof is similar to the previous one. We first visualize the lift of $r$ in the fundamental domain $\widetilde{\Omega}$ in this setting. Let $\widetilde{\Omega}$, $\widetilde{\infty}$ and $\tilde{r}_1$ be as before. Then the endpoint of $\tilde{r}_1$ is another lift $\widetilde{\infty}_1$ of $\infty$, viewing from which $\overline{\tilde{r}_1}$ is a lift of $\bar{r}_1$. Then there are two consecutive lifts $\tilde{\tau},\tilde{\tau}'$ of $\tau$ starting from $\widetilde{\infty}_1$ so that $(\tilde{\tau},\overline{\tilde{r}_1},\tilde{\tau}')$ has positive circular order. See the right of Figure \ref{fig: surgery}. Now there is a unique lift $\tilde{r}_2$ of $r_2$ starting at $\widetilde{\infty}_1$ so that $(\tilde{\tau},\tilde{r}_2,\tilde{\tau}')$ has positive circular order, and it is to the left of $\overline{\tilde{r}_1}$ if and only if $r_2>\bar{r}_1$ on $I_\tau$. 
	Then a lift $\tilde{r}$ of $r$ is the third side of the ideal geodesic triangle with sides $\tilde{r}_1$ and $\tilde{r}_2$ as shown on the right of Figure \ref{fig: surgery}, from which the last claim easily follows.
	The additional assumption that $r_2$ and $\bar{r}_1$ are adjacent ensures a simple isotopy representative of the concatenation when $r_1,r_2$ are disjoint, which implies that $r$ is simple. 
	Finally, a bi-infinite geodesic $\ell$ intersecting $r$ transversely must have a lift $\tilde{\ell}$ entering the ideal geodesic triangle above from the side $\tilde{r}$. Thus the only case where $\tilde{\ell}$ avoids $\tilde{r}_1$ and $\tilde{r}_2$ is when $\widetilde{\infty}_1$ is an end of $\tilde{\ell}$. In this case, with the appropriate orientation $\ell$ is a ray starting from $\infty$ sitting in between $\bar{r}_1$ and $r_2$ on $I_\tau$. This proves the second claim.
\end{proof}

\begin{lemma}\label{lemma: strengthen two-side approachable}
	Let $\tau$ be a two-side approachable long ray. With the notation above, we can choose the sequences of loops $\ell_i$ and $r_i$ so that they are 
	mutually disjoint and their order on $I_\tau$ satisfies (\ref{eqn: order of loops}).
\end{lemma}		
\begin{proof}
	We start with two sequences of loops $L_i,R_i$ as in the definition converging to the left and right endpoints of $I_\tau$ respectively. Up to taking subsequences, we may assume $L_i$, $R_i$ to be monotone on $I_\tau$ with $L_1<R_1$. 
	
	We will first inductively obtain mutually disjoint loops $\ell'_i,r'_i$ that converge to the two endpoints and satisfy a different order
	\begin{equation}\label{eqn: order of nested loops}
		\cdots < \ell'_i < \bar{r}'_i < \cdots < \ell'_1 < \bar{r}'_1 < r'_1 < \bar{\ell}'_1 < \cdots  < r'_i < \bar{\ell}'_i < \cdots.
	\end{equation}
	Geometrically this order makes the interiors of $\ell_i$ and $r_i$ nested. See Figure \ref{fig: loops_another_form} for an illustration.
	
	\begin{figure}
		\labellist
		\small 
		\pinlabel $\widetilde{\infty}$ at 110 228
		\pinlabel $\tilde{\tau}$ at -5 85
		\pinlabel $\ell'_2$ at 0 65
		\pinlabel $\bar{r}'_2$ at 10 47
		\pinlabel $\ell'_1$ at 42 10
		\pinlabel $\bar{r}'_1$ at 75 -5
		\pinlabel $\tilde{\tau}'$ at 220 65
		\pinlabel $\bar{\ell}'_2$ at 208 43
		\pinlabel $r'_2$ at 198 28
		\pinlabel $\bar{\ell}'_1$ at 160 1
		\pinlabel $r'_1$ at 135 -8
		
		\pinlabel $\tau$ at 375 150	
		\pinlabel $\infty$ at 380 105
		\pinlabel $\ell'_2$ at 288 30	
		\pinlabel $\ell'_1$ at 337 40
		\pinlabel $r_1'$ at 375 45
		\pinlabel $r_2'$ at 425 55
		\endlabellist
		\centering
		\includegraphics[scale=0.7]{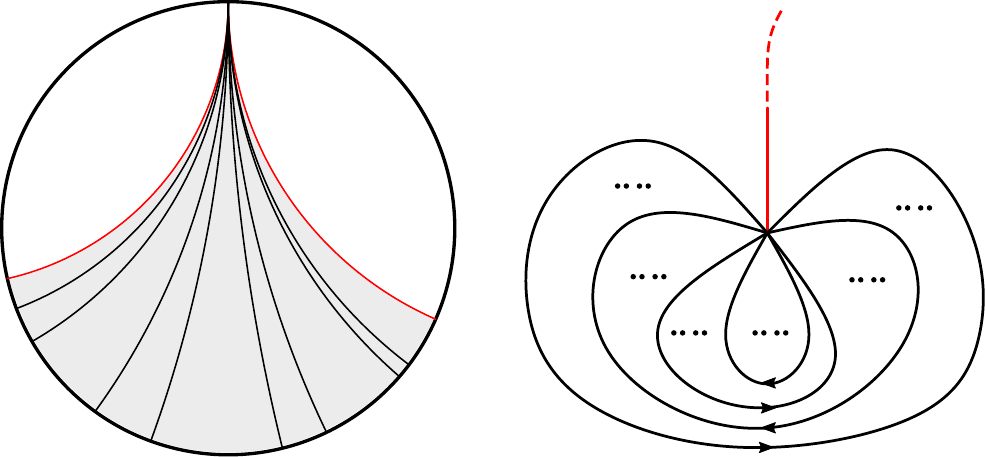}
		\caption{Loops $\ell'_i,r'_i$ with nested interiors that converge to and are disjoint from a two-side approachable long ray $\tau$, where dots indicate Cantor subsets}\label{fig: loops_another_form}
	\end{figure}
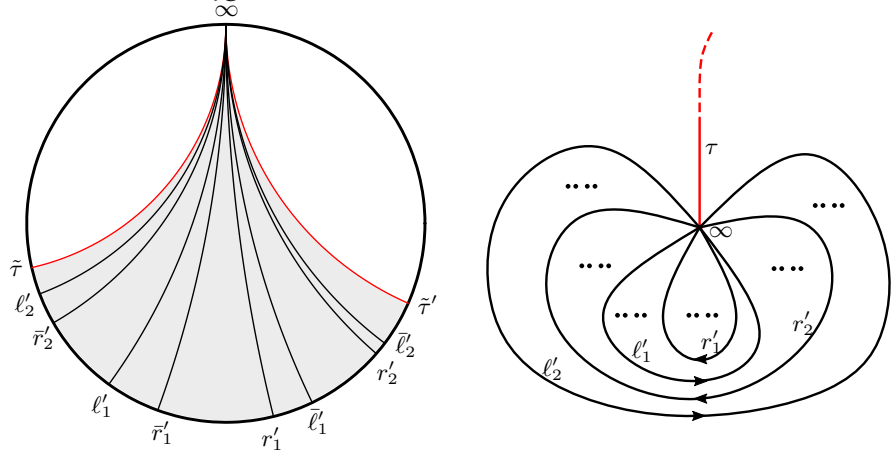
	
	To this end, let $r'_1$ be $R_1$ with a suitable orientation so that $\bar{r}'_1<r'_1$. Suppose we have obtained $r'_i$ for $1\le i\le n$ and $\ell'_i$ for $1\le i\le n-1$ such that they are mutually disjoint and satisfy the order (\ref{eqn: order of nested loops}). Since $L_j$ converges to the left endpoint of $I_\tau$, we may choose $j$ large enough so that $L_j<\bar{r}'_n$. 
	
	There are two cases: 
	\begin{enumerate}
		\item If $L_j$ intersects some of the already chosen $\ell'_i$ or $r'_i$, then the first intersection $p$ of $L_j$ with this collection of loops lies on $r'_n$ since the interiors are nested. Let $\ell'_n$ be the straightening of the piecewise geodesic that first follows $L_j$ up to $p$ and then follows $\bar{r}'_n$. See $L_j^{(1)}$ and $\ell'_n$ in Figure \ref{fig: surgery_to_standard} for an illustration. Applying Lemma \ref{lemma: concat at intersection} to $\bar{\ell}'_n$ and $\ell'_n$, we observe that $\ell'_n$ is a loop disjoint from $\tau$ and $r'_n$ such that $r'_n<\bar{\ell}'_n$ and $\ell'_n<L_j<\bar{r}'_n$. 
		\item If $L_j$ is disjoint from all the already chosen $\ell'_i$ or $r'_i$, then we have either $L_j<\bar{r}'_n<r_n'<\bar{L}_j$ or $L_j<\bar{L}_j<\bar{r}'_n<r_n'$. In the former case, we simply let $\ell'_n=L_j$. In the latter case, let $\ell'_n$ be the straightening of the unique concatenation of $L_j,\bar{r}'_n$ at $\infty$ locally disjoint from $\tau$. See $L_j^{(2)}$ and $\ell'_n$ in Figure \ref{fig: surgery_to_standard} for an illustration. 
		Applying Lemma \ref{lemma: concat at infty} to $\bar{\ell}'_n$ and $\ell'_n$, we observe that $\ell'_n$ is a loop disjoint from $\tau$ and $r'_n$ such that $r'_n<\bar{\ell}'_n$ and $\ell'_n<L_j<\bar{r}'_n$. 
	\end{enumerate}
	
	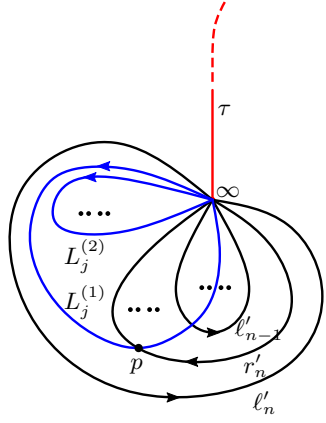
\begin{figure}
		\labellist
		\small 
		\pinlabel $\tau$ at 118 160	
		\pinlabel $\infty$ at 120 115
		\pinlabel $L_j^{(1)}$ at 42 55
		\pinlabel $L_j^{(2)}$ at 42 80
		\pinlabel $\ell'_n$ at 140 0
		\pinlabel $p$ at 70 20
		\pinlabel $\ell'_{n-1}$ at 137 40
		\pinlabel $r'_n$ at 135 20
		\endlabellist
		\centering
		\includegraphics[scale=0.7]{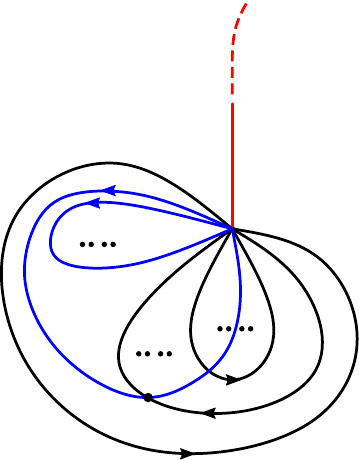}
		\caption{Two potential $L_j$'s that give rise to $\ell'_n$ after surgery}\label{fig: surgery_to_standard}
	\end{figure}
	
	In either case, we obtain a loop $\ell'_n$ with the desired properties and $\ell'_n<L_j$. A symmetric surgery to some $R_k$ for a large $k$ gives us the next loop $r'_{n+1}$ with the desired properties and $r'_{n+1}>R_k$. Hence by induction we obtain two sequences of mutually disjoint loops $\ell'_i,r'_i$ in the desired order (\ref{eqn: order of nested loops}) and they converge to the two endpoints of $I_\tau$ respectively.
	
	Now we modify $\ell'_i,r'_i$ to get the desired $\ell_i,r_i$ satisfying order (\ref{eqn: order of loops}). Let $\bar{r}_n$ be the concatenation of $r'_n$ with $\ell'_{n}$ locally disjoint from $\tau$ for $n\ge 1$. Similarly let $\bar{\ell}_n$ be the concatenation of $\ell'_n$ with $r'_{n+1}$ locally disjoint from $\tau$ for $n\ge 1$. Then by Lemma \ref{lemma: concat at infty}, we have $\bar{\ell}'_n>r_n>\bar{r}_n>r'_n$ and $\bar{r}'_{n+1}<\ell_n<\bar{\ell}_n<\ell'_n$ for all $n$. It follows that the sequences of loops $\ell_i,r_i$ are mutually disjoint, converge to the two endpoints of $I_\tau$, and satisfy the desired order (\ref{eqn: order of loops}) on $I_\tau$.
\end{proof}
\begin{rmk}\label{rmk: not simple}
	It is not even necessary to insist that $\ell_k$ and $r_k$ are simple in the definition of a two-side approachable ray $\tau$. This is because there are similar surgeries that eliminate self-intersections of any $\ell_k$ (resp. $r_k$) while keeping it disjoint from $\tau$ and making it closer to $\tau$.	We give a sketch.
	
	Suppose $r_k$ self-intersects. Let $p$ be the first self-intersection point on $r_k$ as one starts out from $\infty$ following $r_k$. This point $p$ cuts $r_k$ into the concatenation of geodesics $\alpha,\beta,\gamma$, where $\alpha$ (resp. $\gamma$) is the starting (resp. ending) geodesic path of $r_k$ from $\infty$ to $p$ (resp. from $p$ to $\infty$) and $\beta$ is the geodesic loop in between. There are two possible modifications: the straightening of $r'_k=\alpha\cdot\beta\cdot\bar{\alpha}$ or $r''_k=\alpha\cdot\bar{\beta}\cdot\bar{\alpha}$. 
	Both are disjoint from $\tau$ and have self-intersection numbers no more than that of $\beta$ and strictly less than that of $r_k$. Moreover, one of the two modifications makes a left turn at $p$ and the other makes a right turn. 
	Hence one of them is greater than $r_k$ on $I_\tau$. Continuing such modifications provides a (simple) loop disjoint from $\tau$ and gets even closer.		
\end{rmk}

\section{$2$-filling rays disjoint from two-side approachable long rays}
\label{sec:straightforward}
The goal of this section is to give an explicit and straightforward construction of $2$-filling rays and prove the following theorem.

\begin{thm}\label{thm: construction for two side approachable}
	For any two-side approachable long ray $\tau$, there is a $2$-filling ray $\gamma$ such that $\tau$ is the only ray disjoint from $\gamma$. 
	Moreover, for any $n\ge1$, there is a set $\uga=\{\gammaup{1},\ldots, \gammaup{n}\}$
	of $n$ mutually disjoint $2$-filling rays such that the set of rays disjoint from any $\gammaup{i}$ is $\{\tau\}\cup (\underline{\gamma}\setminus\{\gammaup{i}\})$. Equivalently, the star of each $\gammaup{i}$ on the ray-and-loop graph $\Ray$ is a clique with vertex set $\{\tau\}\cup \underline{\gamma}$.
\end{thm}

We first describe the construction of $\uga$ for each $n\ge 1$. Let $\ell_m,r_m$ be disjoint loops converging to $\tau$ as in Lemma \ref{lemma: strengthen two-side approachable}. Choose two increasing sequences of positive integers $p_k$ and $q_k$ such that $p_{k+1}-p_k\ge n$ and $q_{k+1}-q_k\ge n$. 
We repeat the following two steps, depending on the parity of $j$, to inductively define $n$ sequences of loops $\gammaup{1}_j,\cdots,\gammaup{n}_j$. In the following, $\cdot$ denotes the concatenation near $\infty$ locally disjoint from $\tau$ introduced in Section \ref{sec: two-side}. 
\begin{flalign*}
	\textbf{Step 1:}&\text{ Let }\alphaup{n-i}_{2k}\defeq \gammaup{n-i}_{2k-1}\cdot \ell_{p_k+i}\cdot \overline{\gammaup{n-i}_{2k-1}}\text{ for all } 0\le i\le n-1, \text{ and}&\\
	&\text{ let }\alphaup{i}_{2k-1}\defeq \gammaup{i}_{2k-2}\cdot r_{q_k+i-1}\cdot \overline{\gammaup{i}_{2k-2}}\text{ for all } 1\le i\le n; \text{ or}&\\
	&\text{ let }\alphaup{i}_{1}\defeq r_{q_1+i-1}\text{ for all } 1\le i\le n \text{ for the initial case when } 2k-1=1.&\\
	\textbf{Step 2:}&\text{ Let }\gammaup{n}_{2k}\defeq \alphaup{n}_{2k} \text{ and }\gammaup{n-i}_{2k}\defeq \alphaup{n-i}_{2k}\cdot \gammaup{n-i+1}_{2k}\text{ for all } 1\le i\le n-1,\text{ and}&\\
	&\text{ let }\gammaup{1}_{2k-1}\defeq \alphaup{1}_{2k-1} \text{ and }\gammaup{i}_{2k-1}\defeq \alphaup{i}_{2k-1}\cdot \gammaup{i-1}_{2k-1}\text{ for all } 2\le i\le n.&
\end{flalign*}

The constructions of $\alphaup{n-i}_{2k}$ and $\gammaup{n-i}_{2k}$ as in the two steps above are depicted in Figures \ref{fig: 2fill_gamma2alpha} and \ref{fig: 2fill_alpha2gamma} respectively.
When $p_k$ and $q_k$ are large enough for all $k$, we will show that $\gammaup{1}_j,\cdots,\gammaup{n}_j$ converge to simple rays $\gammaup{1},\cdots,\gammaup{n}$ with the desired properties as $j\to\infty$.

\begin{figure}
	\labellist
	\small 
	
	\pinlabel $\ell_{p_k+2}$ at 10 173
	\pinlabel $\ell_{p_k+1}$ at 0 130
	\pinlabel $\ell_{p_k}$ at 25 52
	\pinlabel $\tau$ at 100 210
	\pinlabel $\gammaup{1}_{2k-1}$ at 170 52
	\pinlabel $\gammaup{2}_{2k-1}$ at 190 140
	\pinlabel $\gammaup{3}_{2k-1}$ at 190 218
	
	\pinlabel $\ell_{p_k+2}$ at 293 173
	\pinlabel $\ell_{p_k+1}$ at 283 130
	\pinlabel $\ell_{p_k}$ at 308 52
	\pinlabel $\tau$ at 383 210
	\pinlabel $\alphaup{1}_{2k}$ at 423 50
	\pinlabel $\alphaup{2}_{2k}$ at 470 115
	\pinlabel $\alphaup{3}_{2k}$ at 470 193
	\endlabellist
	\centering
	\includegraphics[scale=0.7]{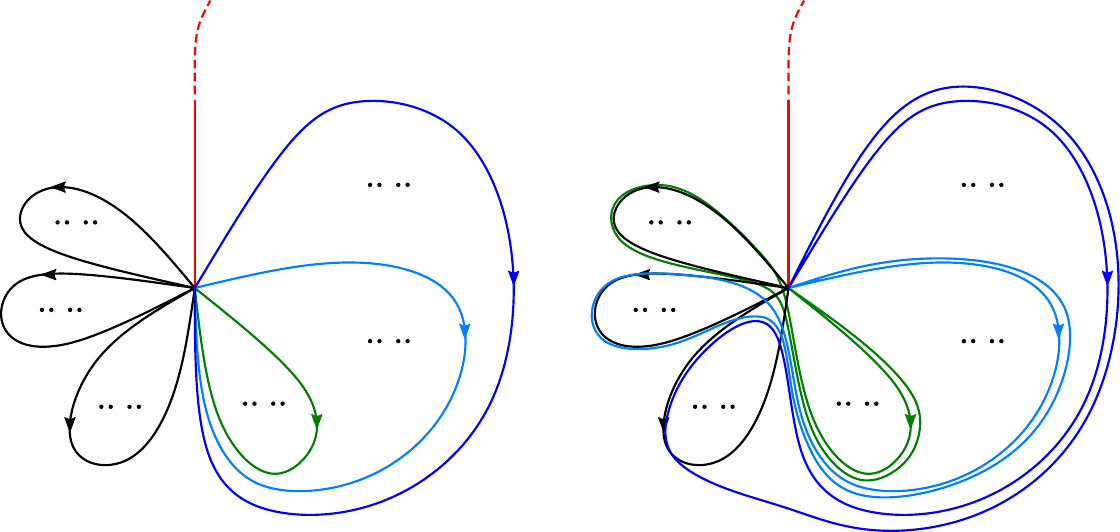}
	\caption{Constructing $\alphaup{n-i}_{2k}$ from $\gammaup{n-i}_{2k-1}$ when $n=3$ as in Step 1.}\label{fig: 2fill_gamma2alpha}
\end{figure}
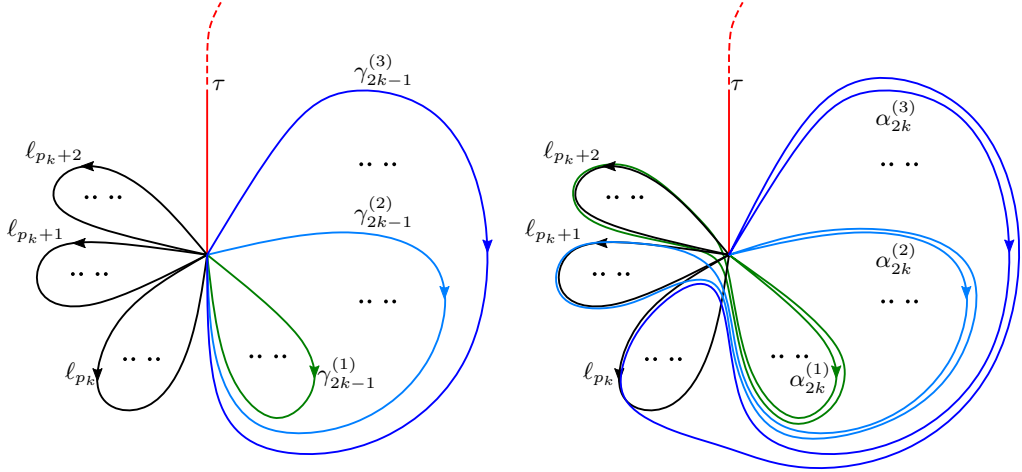

\begin{figure}
	\labellist
	\small 
	
	\pinlabel $\tau$ at 70 210
	\pinlabel $\alphaup{1}_{2k}$ at 25 48
	\pinlabel $\alphaup{2}_{2k}$ at 117 50
	\pinlabel $\alphaup{3}_{2k}$ at 145 85
	
	\pinlabel $\tau$ at 303 210
	\pinlabel $\alphaup{1}_{2k}$ at 258 48
	\pinlabel $\alphaup{2}_{2k}$ at 350 50
	\pinlabel $\gammaup{3}_{2k}=\alphaup{3}_{2k}$ at 373 85
	\pinlabel $\gammaup{1}_{2k}$ at 208 68
	\pinlabel $\gammaup{2}_{2k}$ at 293 40
	\endlabellist
	\centering
	\includegraphics[scale=0.7]{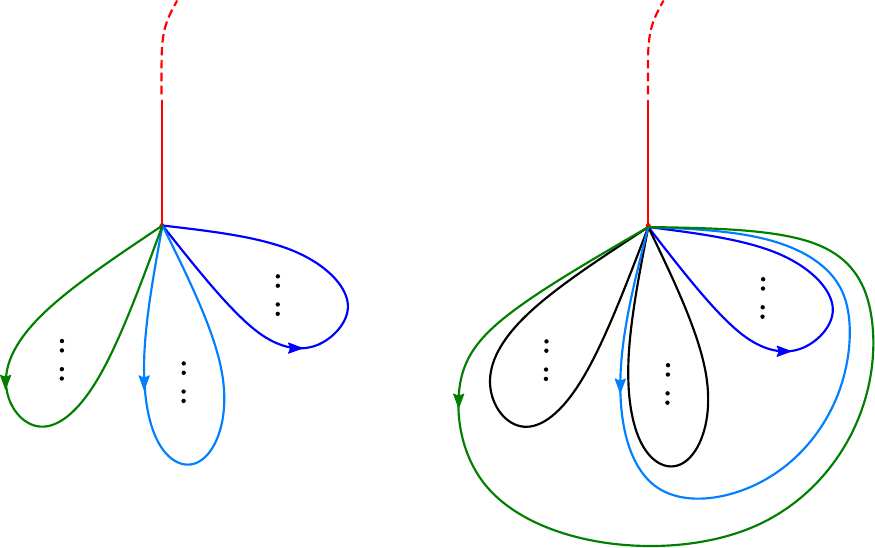}
	\caption{Constructing $\gammaup{n-i}_{2k}$ from $\alphaup{n-i}_{2k}$ when $n=3$ as in Step 2.}\label{fig: 2fill_alpha2gamma}
\end{figure}

Before we proceed to show that the construction gives us the desired $2$-filling rays, we explain how this intuitively works in the case $n=1$, where we take $p_k=q_k=k$. Figure \ref{fig: gamma_n1} shows the ray $\gamma$ right before it starts to follow $\ell_2$ for the first time. One key property of $\gamma$ is that when it starts to follow some $\ell_k$ (resp. $r_k$) for the first time it is in the middle slightly to the left (resp. right) of $\tau$. Such segments get close to the starting segments of $\tau$ and $\gamma$ on both sides, and thus force any ray other than $\gamma$ and $\tau$ to intersect $\gamma$ transversely.

\begin{figure}
	\labellist
	\small 
	\pinlabel $\infty$ at 105 48
	\pinlabel $\tau$ at 310 60
	\pinlabel $\ell_1$ at 175 90
	\pinlabel $\ell_2$ at 278 90
	\pinlabel $r_1$ at 55 55
	\pinlabel $\gamma$ at 24 55
	\pinlabel $r_2$ at 195 10
	\endlabellist
	\centering
	\includegraphics[scale=1]{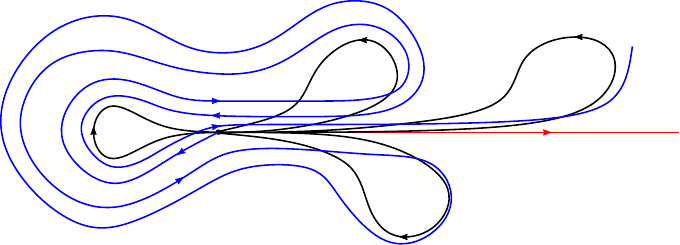}
	\caption{The $2$-filling ray $\gamma$ only disjoint from $\tau$ in our construction after the first few steps.}\label{fig: gamma_n1}
\end{figure}
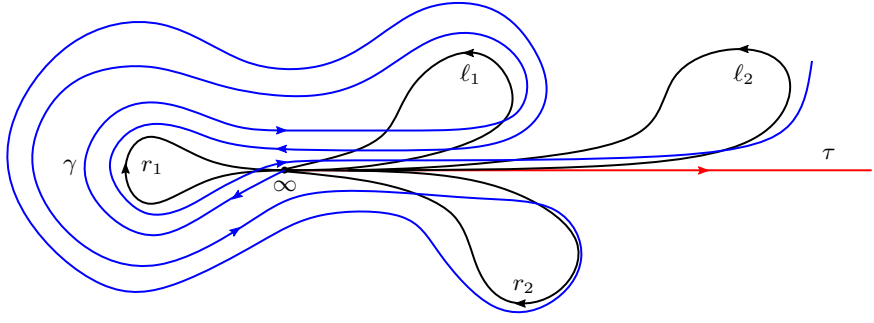

In later sections we will give another construction of $2$-filling rays using train tracks and laminations (see Theorem \ref{lamthm}). That construction is similar to the construction here with $n=1$, and one can almost see a train track in Figure \ref{fig: gamma_n1} by collapsing parallel strands. Compare with Figure \ref{traintrack}. See Section \ref{sec:correspondence} for a detailed discussion on the correspondence.

To prove Theorem \ref{thm: construction for two side approachable}, we first prove some properties of the loops $\alphaup{i}_{k}$ and $\gammaup{i}_k$.

\begin{lemma}\label{lemma: monotonicity}
	The sequences of loops $\gammaup{1}_j,\cdots,\gammaup{n}_j$ and $\alphaup{1}_j,\cdots,\alphaup{n}_j$ constructed above have the following properties:
	\begin{enumerate}
		\item For any given $j$, the loops $\gammaup{i}_j$ (resp. $\alphaup{i}_j$) as we vary $i$ are mutually disjoint and disjoint from $\tau$;\label{item: simple}
		\item For any $1\le i \le n$, we have 
		$$\gammaup{i}_{2k-1}<\gammaup{i}_{2k}\le \alphaup{i}_{2k}<\overline{\alphaup{i}_{2k}}<\gammaup{i}_{2k-2},$$
		on $I_\tau$ for all $k> 1$ and similarly
		$$\gammaup{i}_{2k-1}<\overline{\alphaup{i}_{2k+1}}<\alphaup{i}_{2k+1}\le\gammaup{i}_{2k+1}<\gammaup{i}_{2k},$$
		on $I_\tau$ for all $k\ge 1$; \label{item: monotone}
		\item For all $k\ge 1$ and any $1\le i\le n$, $\gammaup{i}_{2k+1}$ and $\gammaup{i}_{2k}$ (resp. $\gammaup{i}_{2k}$ and $\gammaup{i}_{2k-1}$) can be made arbitrarily close on $I_\tau$ by choosing $p_k$ (resp. $q_{k}$) large enough;\label{item: close}
		\item For all $k\ge 1$ we have $\bar{\ell}_{p_k+n}<\alphaup{1}_{2k}<\overline{\alphaup{1}_{2k}}<\alphaup{2}_{2k}<\cdots<\alphaup{n}_{2k}<\overline{\alphaup{n}_{2k}}<\bar{r}_{q_{k+1}}$ on $I_\tau$, and 
		$\bar{\ell}_{p_k+n}<\gammaup{1}_{2k}<\cdots<\gammaup{n}_{2k}<\overline{\gammaup{n}_{2k}}<\cdots<\overline{\gammaup{1}_{2k}}<\bar{r}_{q_{k+1}}$.
		Similarly
		$\bar{\ell}_{p_k}<\overline{\alphaup{1}_{2k-1}}<\alphaup{1}_{2k-1}<\overline{\alphaup{2}_{2k-1}}<\cdots<\overline{\alphaup{n}_{2k-1}}<\alphaup{n}_{2k-1}<\bar{r}_{q_k+n}$, and
		$\bar{\ell}_{p_k}<\overline{\gammaup{n}_{2k-1}}<\cdots<\overline{\gammaup{1}_{2k-1}}<\gammaup{1}_{2k-1}<\cdots<\gammaup{n}_{2k-1}<\bar{r}_{q_k+n}$.
		\label{item: order}
		\item Both $\alphaup{i}_{2k-1}$ and $\gammaup{i}_{2k-1}$ are disjoint from $\ell_s$ and $r_t$ for all $s\ge p_k$ and $t\ge q_k+i$, and similarly
		$\alphaup{n-i}_{2k}$ and $\gammaup{n-i}_{2k}$ are disjoint from $\ell_s$ and $r_t$ for all $s\ge p_k+i+1$ and $t\ge q_{k+1}$.
		\label{item: disjoint}
	\end{enumerate}
\end{lemma}
\begin{proof}
	We prove bullets (\ref{item: simple}), (\ref{item: monotone}), (\ref{item: order}) and (\ref{item: disjoint}) together by induction on $j$. 
	
	For $j=1$, we have $\alphaup{i}_{1}\defeq r_{q_1+i-1}$ mutually disjoint and satisfying (\ref{eqn: order of loops}). We verify these statements about $\alphaup{i}_{1}$ and $\gammaup{i}_{1}$.
	By Lemma \ref{lemma: concat at infty} we see that the $\gammaup{i}_{1}$'s are mutually disjoint simple loops since the $\alphaup{i}_{1}$'s are disjoint and satisfy (\ref{eqn: order of loops}). 
	The second claim in Lemma \ref{lemma: concat at infty} also implies that each $\gammaup{i}_1$ is disjoint from $\tau$, $\ell_s$, and $r_t$ for all $s\ge p_1$ and $t\ge q_1+i$. This verifies bullets (\ref{item: simple}) and (\ref{item: disjoint}). Bullet (\ref{item: order}) follows from (\ref{eqn: order of loops}) and the last claim in Lemma \ref{lemma: concat at infty}, where we treat each $\overline{\gammaup{i}_1}$ as the concatenation $\overline{\gammaup{i-1}_1}\cdot \overline{\alphaup{i}_{1}}$ for $i\ge2$.
	Finally, bullet (\ref{item: monotone}) is vacuous for $\alphaup{i}_{1}$ and $\gammaup{i}_{1}$.

	Suppose the statements about loops up to $\alphaup{i}_j$ and $\gammaup{i}_j$ for all $1\le i\le n$ and some $j\ge 1$ are all correct. We verify the results as we add $\alphaup{i}_{j+1}$ and $\gammaup{i}_{j+1}$ for all $1\le i\le n$ to the list. We assume $j=2k-1$ for some $k\ge1$ in the sequel. The case where $j$ is even can be proved similarly in a symmetric way.
	
	By the induction hypothesis and (\ref{eqn: order of loops}), $\gammaup{n-i}_{2k-1}$ is disjoint from $\tau$, $\ell_s$ and $r_t$ for $s\ge p_k$ and $t\ge q_{k+1}\ge q_k+n$, and we have
	$\ell_{p_k+i}<\bar{\ell}_{p_k+i}<\overline{\gammaup{n-i}_{2k-1}}<\gammaup{n-i}_{2k-1}$.
	Applying Lemma \ref{lemma: concat at infty} twice to 
	$\alphaup{n-i}_{2k}\defeq \gammaup{n-i}_{2k-1}\cdot (\ell_{p_k+i}\cdot \overline{\gammaup{n-i}_{2k-1}})$ as the result of two concatenations,
	we see that 
	$\alphaup{n-i}_{2k}$ is a loop disjoint from $\tau$, $\ell_s$ and $r_t$ for all $s\ge p_k+i+1$ and $t\ge q_{k+1}$ as in bullets (\ref{item: simple}) and (\ref{item: disjoint}). 
	Lemma \ref{lemma: concat at infty} also implies that $\bar{\ell}_{p_k+n}<\alphaup{n-i}_{2k}, \overline{\alphaup{n-i}_{2k}}<\bar{r}_{q_{k+1}}$ as in bullet (\ref{item: order}).
	
	Next we show that $\alphaup{n-i}_{2k}$ and $\alphaup{n-j}_{2k}$ are disjoint for any $i<j$. This can be seen by observing the disjoint representatives in Figure \ref{fig: 2fill_gamma2alpha}.
	Alternatively, note that $\gammaup{n-j}_{2k-1}$ and $\ell_{p_k+j}$ are disjoint from $\ell_{p_k+i}$ and $\gammaup{n-i}_{2k-1}$, and that neither $\gammaup{n-j}_{2k-1}$ nor $\ell_{p_k+j}$ sits between $\bar{\ell}_{p_k+i}$ and $\overline{\gammaup{n-i}_{2k-1}}$. So we deduce from Lemma \ref{lemma: concat at infty} that $\ell_{p_k+i}\cdot \overline{\gammaup{n-i}_{2k-1}}$ is disjoint from $\gammaup{n-j}_{2k-1}$ and $\ell_{p_k+j}$, 
	and that neither $\gammaup{n-j}_{2k-1}$ nor $\ell_{p_k+j}$ sits between $\overline{\gammaup{n-i}_{2k-1}}$ and $\ell_{p_k+i}\cdot \overline{\gammaup{n-i}_{2k-1}}$. Thus by applying Lemma \ref{lemma: concat at infty} again, we see that $\alphaup{n-i}_{2k}$ is also disjoint from $\gammaup{n-j}_{2k-1}$ and $\ell_{p_k+j}$, and that $\ell_{p_k+j}<\bar{\ell}_{p_k+j}<\overline{\gammaup{n-j}_{2k-1}}<\gammaup{n-j}_{2k-1}<\alphaup{n-i}_{2k}$. 
	By a similar process, we can further deduce that $\alphaup{n-i}_{2k}$ and $\alphaup{n-j}_{2k}=\gammaup{n-j}_{2k-1}\cdot (\ell_{p_k+j}\cdot \overline{\gammaup{n-j}_{2k-1}})$ are disjoint.
	
	Now we prove $\gammaup{n-i}_{2k-1}<\gammaup{n-i}_{2k}\le \alphaup{n-i}_{2k}<\overline{\alphaup{n-i}_{2k}}<\gammaup{n-i}_{2k-2}$ as in bullet (\ref{item: monotone}) by finding suitable lifts of $\gammaup{n-i}_{2k}$, $\alphaup{n-i}_{2k}$, and $\overline{\alphaup{n-i}_{2k}}$.
	The induction hypothesis already guarantees $\gammaup{n-i}_{2k-1}<\gammaup{n-i}_{2k-2}$.
	On a fundamental domain $\widetilde{\Omega}$ of $\Omega_C$ between two consecutive lifts of $\tau$ starting from a chosen $\widetilde{\infty}$, 
	we have the lifts of $\gammaup{n-i}_{2k-2}$, $\gammaup{n-i}_{2k-1}$ starting at $\widetilde{\infty}$ shown in Figure \ref{fig: gamma_1}.
	Then we have lifts of $\ell_{p_k+i}$ and $\bar{\ell}_{p_k+i}$ starting at the endpoint of the lift of $\gammaup{n-i}_{2k-1}$. Their relative positions are correct since
	$\ell_{p_k+i}<\bar{\ell}_{p_k+i}<\overline{\gammaup{n-i}_{2k-1}}$, and they do not intersect the lift of $\gammaup{n-i}_{2k-2}$ since $\ell_{p_k+i}$ is disjoint from $\gammaup{n-i}_{2k-2}$.
	Then we have lifts of $\overline{\gammaup{n-i}_{2k-1}}$ starting at the endpoints of the lifts of $\ell_{p_k+i}$ and $\bar{\ell}_{p_k+i}$ respectively. 
	They both go to the left as shown in Figure \ref{fig: gamma_1} since $\ell_{p_k+i}<\bar{\ell}_{p_k+i}<\overline{\gammaup{n-i}_{2k-1}}$.
	From this, we obtain the lifts of $\alphaup{n-i}_{2k}$ and $\overline{\alphaup{n-i}_{2k}}$ starting from $\widetilde{\infty}$.
	This shows that $\gammaup{n-i}_{2k-1}<\alphaup{n-i}_{2k}<\overline{\alphaup{n-i}_{2k}}<\gammaup{n-i}_{2k-2}$. 
	It remains to find the lift of $\gammaup{n-i}_{2k}$.
	
	For any $1\le i<j\le n$, since $\gammaup{n-j}_{2k-1}<\gammaup{n-i}_{2k-1}$ and $\gammaup{n-i}_{2k-1}$ is disjoint from $\ell_{p_k+j}$, 
	the above configuration implies that
	$\overline{\alphaup{n-j}_{2k}}<\alphaup{n-i}_{2k}$. Based on this relation, the lift of $\alphaup{n-i+1}_{2k}$ starting at the end of the lift of $\alphaup{n-i}_{2k}$ must head to the left as shown in Figure \ref{fig: gamma_1}. It must stay inside the half-disk bounded by the lift of $\ell_{p_k+i}$ since $\alphaup{n-i+1}_{2k}$ is disjoint from $\ell_{p_k+i}$.
	Continuing this process, we obtain lifts of $\alphaup{n-i+1}_{2k},\ldots, \alphaup{n}_{2k}$ this way to construct a lift of $\gammaup{n-i}_{2k}$ starting from $\widetilde{\infty}$, shown in Figure \ref{fig: gamma_1}. 
	This implies that $$\gammaup{i}_{2k-1}<\gammaup{i}_{2k}\le \alphaup{i}_{2k}<\overline{\alphaup{i}_{2k}}<\gammaup{i}_{2k-2}$$
	as in bullet (\ref{item: monotone}).
	
	Note that we also proved the inequality $\alphaup{1}_{2k}<\overline{\alphaup{1}_{2k}}<\alphaup{2}_{2k}<\cdots<\alphaup{n}_{2k}<\overline{\alphaup{n}_{2k}}$ as in bullet (\ref{item: order}) along the way. 
	The inequalities about $\alphaup{n-i}_{2k}$'s that we have established, together with Lemma \ref{lemma: concat at infty}, implies that the $\gammaup{n-i}_{2k}$'s are simple loops and satisfy the disjointness in bullets (\ref{item: simple}) and (\ref{item: disjoint}). The inequalities in bullet (\ref{item: order}) concerning $\gammaup{n-i}_{2k}$'s also follow this way; also see Figure \ref{fig: 2fill_alpha2gamma} for an illustration.
	
	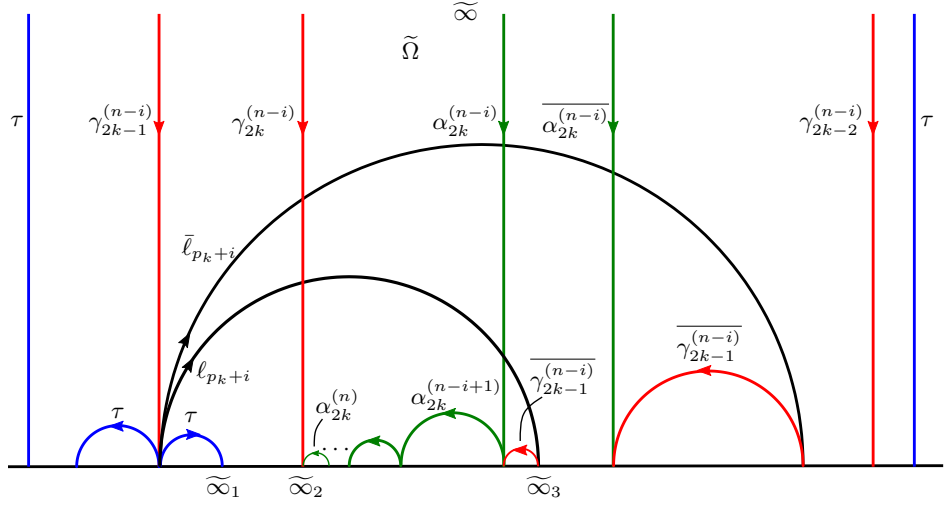
\begin{figure}
		\labellist
		\small 
		\pinlabel $\widetilde{\infty}$ at 250 250
		\pinlabel $\widetilde{\Omega}$ at 220 230
		\pinlabel $\tau$ at 5 190
		
		\pinlabel $\gammaup{n-i}_{2k-1}$ at 62 190
		\pinlabel $\tau$ at 60 30
		\pinlabel $\bar{\ell}_{p_k+i}$ at 110 120
		\pinlabel $\ell_{p_k+i}$ at 118 52
		\pinlabel $\tau$ at 100 28
		\pinlabel $\widetilde{\infty}_1$ at 118 -10
		\pinlabel $\gammaup{n-i}_{2k}$ at 140 190
		\pinlabel $\widetilde{\infty}_2$ at 163 -10
		\pinlabel $\alphaup{n}_{2k}$ at 180 35
		\pinlabel $\cdots$ at 180 10
		\pinlabel $\alphaup{n-i}_{2k}$ at 250 190
		\pinlabel $\overline{\alphaup{n-i}_{2k}}$ at 310 190
		\pinlabel $\alphaup{n-i+1}_{2k}$ at 245 40
		\pinlabel $\widetilde{\infty}_3$ at 293 -10
		\pinlabel $\overline{\gammaup{n-i}_{2k-1}}$ at 303 50
		\pinlabel $\overline{\gammaup{n-i}_{2k-1}}$ at 383 68
		\pinlabel $\gammaup{n-i}_{2k-2}$ at 450 190
		\pinlabel $\tau$ at 502 190
		
		\endlabellist
		\centering
		\includegraphics[scale=0.7]{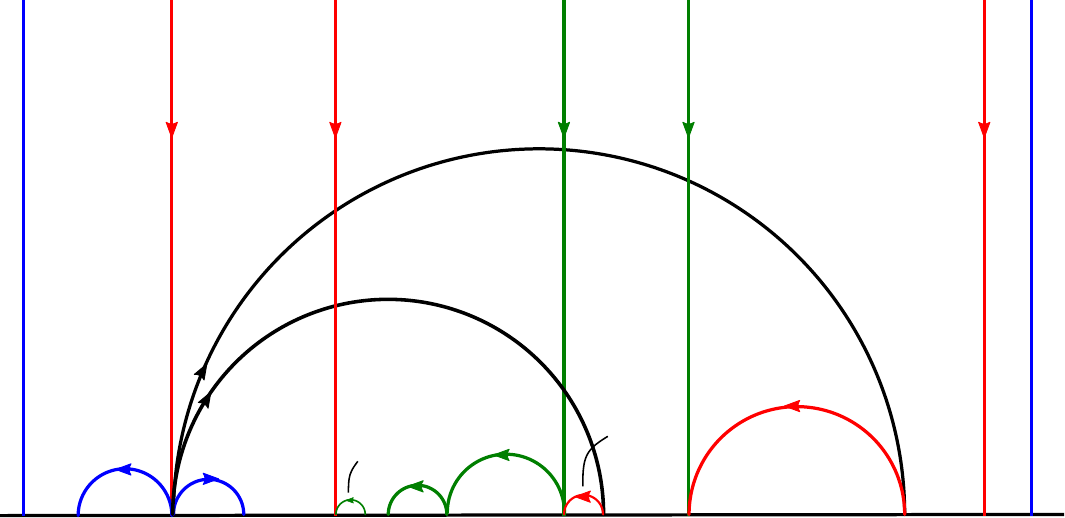}
		\vspace{5pt}
		\caption{Obtaining the lifts of $\alphaup{n-i}_{2k}$, $\overline{\alphaup{n-i}_{2k}}$ and $\gammaup{n-i}_{2k}$ starting from $\widetilde{\infty}$ in $\widetilde{\Omega}$, shown on the upper half-plane with the point at infinity being $\widetilde{\infty}$.}\label{fig: gamma_1}
	\end{figure}
	
	This completes the inductive step and proves bullets (\ref{item: simple}), (\ref{item: monotone}), (\ref{item: order}) and (\ref{item: disjoint}).
	
	To see bullet (\ref{item: close}), note that by bullet (\ref{item: monotone}), $\gammaup{n-i}_{2k+1}$ sits in between $\gammaup{n-i}_{2k-1}$ and $\gammaup{n-i}_{2k}$. 
	Since $\gammaup{n-i}_{2k+1}$ is disjoint from $\tau$ by bullet (\ref{item: simple}), in Figure \ref{fig: gamma_1}, the lift of $\gammaup{n-i}_{2k+1}$ starting from $\widetilde{\infty}$ must have endpoint between $\widetilde{\infty}_1$ and $\widetilde{\infty}_2$, and thus between $\widetilde{\infty}_1$ and $\widetilde{\infty}_3$. As $p_k\to \infty$, the loop $\ell_{p_k+i}$ converges to $\tau$ and $\widetilde{\infty}_3$ converges to $\widetilde{\infty}_1$, thus $\gammaup{n-i}_{2k+1}$ and $\gammaup{n-i}_{2k}$ can be made arbitrarily close by choosing $p_k$ large.
	
\end{proof}

\begin{lemma}\label{lemma: reversed order}
	For all $k> 1$ and any $1\le i\le j\le n$ we have 
	$$\overline{\alphaup{j}_{2k}}<\overline{\alphaup{j-1}_{2k}\cdot \alphaup{j}_{2k}}<\cdots<\overline{\alphaup{i}_{2k}\cdots \alphaup{j}_{2k}}<\gammaup{j}_{2k-2}$$
	on $I_\tau$ and similarly for all $k\ge 1$ we have
	$$\overline{\alphaup{i}_{2k+1}}>\overline{\alphaup{i+1}_{2k+1}\cdot \alphaup{i}_{2k+1}}>\cdots>\overline{\alphaup{j}_{2k+1}\cdots \alphaup{i}_{2k+1}}>\gammaup{i}_{2k-1}.$$
\end{lemma}
\begin{proof}
	Recall the construction of the lift of $\overline{\alphaup{j}_{2k}}=\gammaup{j}_{2k-1}\cdot \bar{\ell}_{p_k+n-j}\cdot \overline{\gammaup{j}_{2k-1}}$ starting from $\widetilde{\infty}$ in Figure \ref{fig: gamma_1}. Denote its endpoint as $\widetilde{\infty}_{(j)}$ in Figure \ref{fig: gamma_2}. 
	To obtain the lift of 
	$\overline{\alphaup{j-1}_{2k}\cdot \alphaup{j}_{2k}}=\overline{\alphaup{j}_{2k}}\cdot \overline{\alphaup{j-1}_{2k}}$ starting from $\widetilde{\infty}$, we first visualize the lift of 
	$\overline{\alphaup{j-1}_{2k}}$ starting from $\widetilde{\infty}_{(j)}$. Note that $\overline{\alphaup{j-1}_{2k}}$
	is the concatenation of $\gammaup{j-1}_{2k-1}$, $\bar{\ell}_{p_k+n-j+1}$ and $\overline{\gammaup{j-1}_{2k-1}}$, all of which are disjoint from $\overline{\gammaup{j}_{2k-1}}$.
	Thus the lift of $\overline{\alphaup{j-1}_{2k}}$ starting from $\widetilde{\infty}_{(j)}$ is shown as in Figure \ref{fig: gamma_2}, which stays inside the semicircle corresponding to $\overline{\gammaup{j}_{2k-1}}$. Here we have used that $\gammaup{j-1}_{2k-1}<\gammaup{j}_{2k-1}<\alphaup{j}_{2k}$ as in Lemma \ref{lemma: monotonicity}.
	
	By concatenating the lift of $\overline{\alphaup{j}_{2k}}$ from $\widetilde{\infty}$ and the lift of $\overline{\alphaup{j-1}_{2k}}$ starting from $\widetilde{\infty}_{(j)}$, we get the lift of 
	$\overline{\alphaup{j-1}_{2k}\cdot \alphaup{j}_{2k}}$ starting from $\widetilde{\infty}$, and observe that
	$\overline{\alphaup{j}_{2k}}<\overline{\alphaup{j-1}_{2k}\cdot \alphaup{j}_{2k}}<\gammaup{j}_{2k-2}$.
	
	We can continue lifting $\overline{\alphaup{m}_{2k}}$ for all $j-1<m\le i$. The same approach proves the first inequality in the lemma. A symmetric argument proves the other claimed inequality.
	\begin{figure}
		\labellist
		\small 
		\pinlabel $\widetilde{\infty}$ at 180 265
		\pinlabel $\widetilde{\Omega}$ at 220 235
		\pinlabel $\tau$ at 5 195
		
		\pinlabel $\gammaup{j}_{2k-1}$ at 62 195
		\pinlabel $\tau$ at 60 35
		\pinlabel $\bar{\ell}_{p_k+n-j}$ at 130 154
		\pinlabel $\bar{\ell}_{p_k+n-j+1}$ at 293 112
		\pinlabel $\tau$ at 100 32
		\pinlabel $\overline{\alphaup{j-1}_{2k}\cdot \alphaup{j}_{2k}}$ at 245 262
		\pinlabel $\overline{\alphaup{j}_{2k}}$ at 160 195
		\pinlabel $\overline{\alphaup{j-1}_{2k}}$ at 150 63
		\pinlabel $\gammaup{j-1}_{2k-1}$ at 210 125
		\pinlabel $\overline{\gammaup{j}_{2k-1}}$ at 300 150
		\pinlabel $\overline{\gammaup{j-1}_{2k-1}}$ at 260 -8
		\pinlabel $\widetilde{\infty}_{(j)}$ at 182 -3
		\pinlabel $\overline{\alphaup{j-2}_{2k}}$ at 280 40
		\pinlabel $\cdots$ at 300 13
		\pinlabel $\overline{\alphaup{i}_{2k}}$ at 310 28
		\pinlabel $\overline{\alphaup{i}_{2k}\cdots \alphaup{j}_{2k}}$ at 325 262
		\pinlabel $\gammaup{j}_{2k-2}$ at 453 190
		\pinlabel $\tau$ at 502 190
		
		\endlabellist
		\centering
		\includegraphics[scale=0.7]{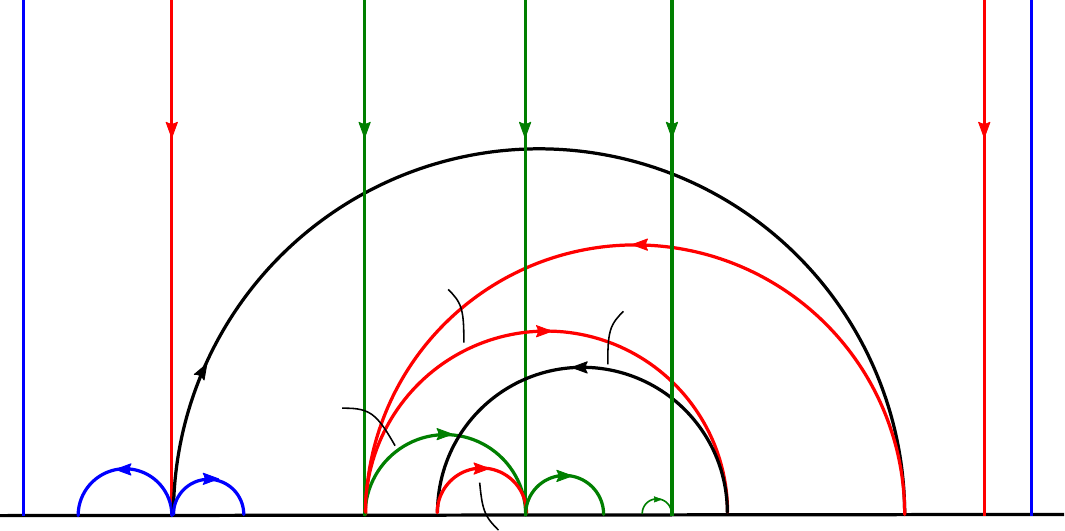}
		\vspace{8pt}
		\caption{Visualizing the lifts of $\overline{\alphaup{j-1}_{2k}\cdot\alphaup{j}_{2k}}$ and $\overline{\alphaup{i}_{2k}\cdots\alphaup{j}_{2k}}$ starting from $\widetilde{\infty}$, shown on the upper half-plane with the point at infinity being $\widetilde{\infty}$.}\label{fig: gamma_2}
	\end{figure}
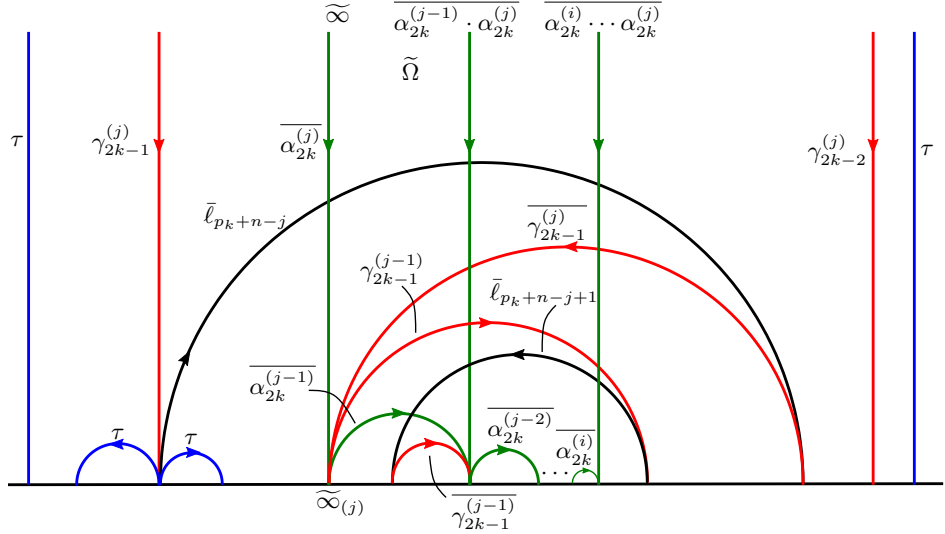
\end{proof}

\begin{cor}\label{cor: reversed convergence}
	We have $\gammaup{1}_{2k-1}<\overline{\gammaup{i}_{2k+1}}\le\overline{\gammaup{1}_{2k+1}}<\gammaup{1}_{2k+1}$ on $I_\tau$ for any $1\le i\le n$ and any $k\ge 1$, 
	and similarly
	$\gammaup{n}_{2k}<\overline{\gammaup{n}_{2k}}\le\overline{\gammaup{i}_{2k}}< \gammaup{n}_{2k-2}$ for all $k>1$.
\end{cor}
\begin{proof}
	We have $\gammaup{n}_{2k}<\overline{\gammaup{n}_{2k}}$ by bullet (\ref{item: order}) of Lemma \ref{lemma: monotonicity}. Since $\overline{\gammaup{i}_{2k}}=\overline{\alphaup{i}_{2k}\cdots\alphaup{n}_{2k}}$, the rest of the inequality with even subscripts follows from this and the first inequality in Lemma \ref{lemma: reversed order} by taking $j=n$. Similarly the inequality with odd subscripts also follows from Lemmas \ref{lemma: reversed order} and \ref{lemma: monotonicity}.
\end{proof}

\begin{lemma}\label{lemma: convergence}
	By choosing $p_k$ and $q_k$ large enough for all $k$, the sequence $\gammaup{i}_k$ converges to a simple ray $\gammaup{i}$ disjoint from $\tau$ for all $1\le i\le n$. 
	In this case, we have 
	\begin{enumerate}
		\item $\gammaup{1}<\gammaup{2}<\cdots<\gammaup{n}$;
		\item the $\gammaup{i}$'s are mutually disjoint; and
		\item $\overline{\gammaup{i}_{2k-1}}$ (resp. $\overline{\gammaup{i}_{2k}}$) converges to $\gammaup{1}$ (resp. $\gammaup{n}$) as $k\to \infty$ for all $1\le i \le n$.
	\end{enumerate}
\end{lemma}
\begin{proof}
	By bullet (\ref{item: monotone}) of Lemma \ref{lemma: monotonicity}, for each $i$, the sequence $\gammaup{i}_{2k}$ (resp. $\gammaup{i}_{2k-1}$) is decreasing (resp. increasing) in $k$, and $\gammaup{i}_{2k-1}<\gammaup{i}_{2k}$ for all $k$. Thus the sequence $\{\gammaup{i}_{2k}\}$ is convergent provided that $\gammaup{i}_{2k-1}$ and $\gammaup{i}_{2k}$ get close as $k$ increases, which can be done by choosing $p_k$ and $q_k$ large; see bullet (\ref{item: close}) of Lemma \ref{lemma: monotonicity}.
	
	Given the convergence, we have $\gammaup{1}<\gammaup{2}<\cdots<\gammaup{n}$ by bullet (\ref{item: order}) of Lemma \ref{lemma: monotonicity}. Since $\gammaup{i}_k$ and $\gammaup{j}_k$ are disjoint for any $k$, so are $\gammaup{i}$ and $\gammaup{j}$. Finally, since $\gammaup{1}_{2k-3}<\overline{\gammaup{i}_{2k-1}}<\gammaup{1}_{2k-1}$ and $\gammaup{n}_{2k}<\overline{\gammaup{i}_{2k}}<\gammaup{n}_{2k-2}$ by Corollary \ref{cor: reversed convergence}, we see that $\overline{\gammaup{i}_{2k-1}}$ (resp. $\overline{\gammaup{i}_{2k}}$) converges to $\gammaup{1}$ (resp. $\gammaup{n}$). 
\end{proof}

We are now in a place to prove Theorem \ref{thm: construction for two side approachable}.
\begin{proof}[Proof of Theorem \ref{thm: construction for two side approachable}]
	Since $\tau$ is two-side approachable, by Lemma \ref{lemma: strengthen two-side approachable}, we obtain disjoint loops $\ell_i,r_i$ converging to $\tau$ on the two sides and satisfying (\ref{eqn: order of loops}). Construct $n$ sequences of loops $\{\gammaup{i}_k\}$ with $1\le i\le n$ as above and choose the constants $p_k,q_k$ properly so that the sequences of loops $\{\gammaup{i}_k\}$ converge to disjoint rays $\gammaup{i}$ by Lemma \ref{lemma: convergence}. 
	
	Let $\uga=\{\gammaup{i}\}_{i=1}^n$. It remains to show that any ray $\alpha$ other than $\tau$ or those in $\uga$ intersects each $\gammaup{i}$.
	
	We first show that $\alpha$ intersects $\gammaup{i}$ for any $1\le i\le j$ if we have $\gammaup{j}<\alpha<\gammaup{j+1}$ for some $j\le n-1$. Recall that 
	$\gammaup{i}_{2k}=\alphaup{i}_{2k}\cdot\gammaup{i+1}_{2k}=\cdots=(\alphaup{i}_{2k}\cdot\alphaup{i+1}_{2k}\cdots\alphaup{j}_{2k})\cdot\gammaup{j+1}_{2k}$.
	We have $\gammaup{j}_{2k}\le\overline{\alphaup{j}_{2k}}<\overline{\alphaup{i}_{2k}\cdot\alphaup{i+1}_{2k}\cdots\alphaup{j}_{2k}}<\gammaup{j}_{2k-2}$ by Lemma \ref{lemma: reversed order} and bullet (\ref{item: monotone}) of Lemma \ref{lemma: monotonicity}. Thus $\overline{\alphaup{i}_{2k}\cdot\alphaup{i+1}_{2k}\cdots\alphaup{j}_{2k}}$ converges to $\gammaup{j}$ as $k$ goes to infinity. Combining this with the fact that $\gammaup{j+1}_{2k}$ converges to $\gammaup{j+1}$, we see lifts of $\gammaup{i}_{2k}$ with the starting point and endpoint converging to the endpoints of $\tilde{\gamma}^{(j)}$ and $\tilde{\gamma}^{(j+1)}$ respectively as $k\to\infty$, where $\tilde{\gamma}^{(j)}$ (resp. $\tilde{\gamma}^{(j+1)}$) is the lift of $\gammaup{j}$ (resp. $\gammaup{j+1}$) starting from $\widetilde{\infty}$, a chosen lift of $\infty$; see Figure \ref{fig: limit}. 
	
	Based on this lift of $\gammaup{i}_{2k}$, we obtain a lift of $\alphaup{i}_{2k+1}=\gammaup{i}_{2k}\cdot r_{q_{k+1}+i-1}\cdot \overline{\gammaup{i}_{2k}}$ starting at the same point $\widetilde{\infty}_1$, shown in Figure \ref{fig: limit}. Since $\overline{\gammaup{i}_{2k}}<\bar{r}_{q_{k+1}}\le\bar{r}_{q_{k+1}+i-1}$ and $r_{q_{k+1}+i-1}$ is disjoint from $\gammaup{j+1}_{2k-1}$, the endpoint of this lift of $\alphaup{i}_{2k+1}$ must sit in between those of $\tilde{\gamma}^{(j+1)}_{2k-1}$ and $\tilde{\gamma}^{(j+1)}_{2k}$,
	where $\tilde{\gamma}^{(j+1)}_{2k-1}$ (resp. $\tilde{\gamma}^{(j+1)}_{2k}$) is the lift of $\gammaup{j+1}_{2k-1}$ (resp. $\gammaup{j+1}_{2k}$) starting from $\widetilde{\infty}$.
	
	Since $\alphaup{i}_{2k+1}<\gammaup{i}_{2k+1}<\gammaup{i}<\gammaup{i}_{2k}$ by bullet (\ref{item: monotone}) of Lemma \ref{lemma: monotonicity}, the lift of $\gammaup{i}$ starting at $\widetilde{\infty}_1$ sits in between the above lifts of $\gammaup{i}_{2k}$ and $\alphaup{i}_{2k+1}$. As $k$ goes to infinity, this process provides lifts of $\gammaup{i}$ converging to
	the geodesic starting from the endpoint of $\tilde{\gamma}^{(j)}$ to the endpoint of $\tilde{\gamma}^{(j+1)}$. Thus any ray $\gammaup{j}<\alpha<\gammaup{j+1}$ intersects $\gammaup{i}$ for all $j\ge i$.
	
	A symmetric argument using lifts of $\gammaup{i}_{2k-1}$ and $\alphaup{i}_{2k}$ shows that any ray $\gammaup{j-1}<\alpha<\gammaup{j}$ intersects $\gammaup{i}$ for all $j\le i$. 
	
	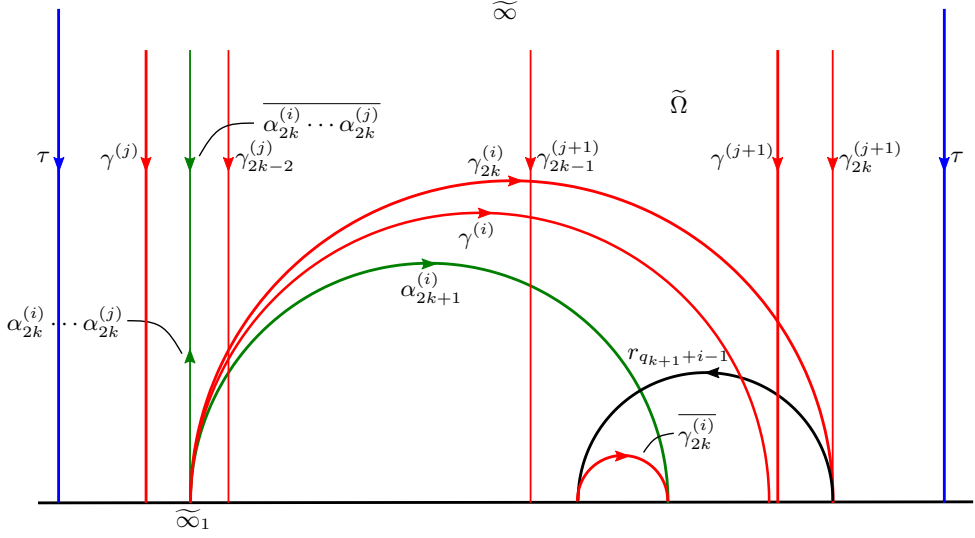
\begin{figure}
		\labellist
		\small 
		\pinlabel $\widetilde{\infty}$ at 255 270
		\pinlabel $\widetilde{\Omega}$ at 350 220
		\pinlabel $\tau$ at 3 190
		\pinlabel $\gammaup{j}$ at 45 190
		\pinlabel $\widetilde{\infty}_1$ at 85 -8
		\pinlabel $\gammaup{j}_{2k-2}$ at 124 190
		\pinlabel $\overline{\alphaup{i}_{2k}\cdots \alphaup{j}_{2k}}$ at 155 210
		\pinlabel $\alphaup{i}_{2k}\cdots \alphaup{j}_{2k}$ at 15 100
		\pinlabel $\alphaup{i}_{2k+1}$ at 215 118
		\pinlabel $\gammaup{i}$ at 240 147
		\pinlabel $\gammaup{i}_{2k}$ at 245 188
		\pinlabel $\gammaup{j+1}_{2k-1}$ at 290 190
		\pinlabel $\gammaup{j+1}$ at 385 190
		\pinlabel $\gammaup{j+1}_{2k}$ at 455 190
		\pinlabel $\tau$ at 502 190
		\pinlabel $\overline{\gammaup{i}_{2k}}$ at 360 40
		\pinlabel $r_{q_{k+1}+i-1}$ at 350 80
		\endlabellist
		\centering
		\includegraphics[scale=0.7]{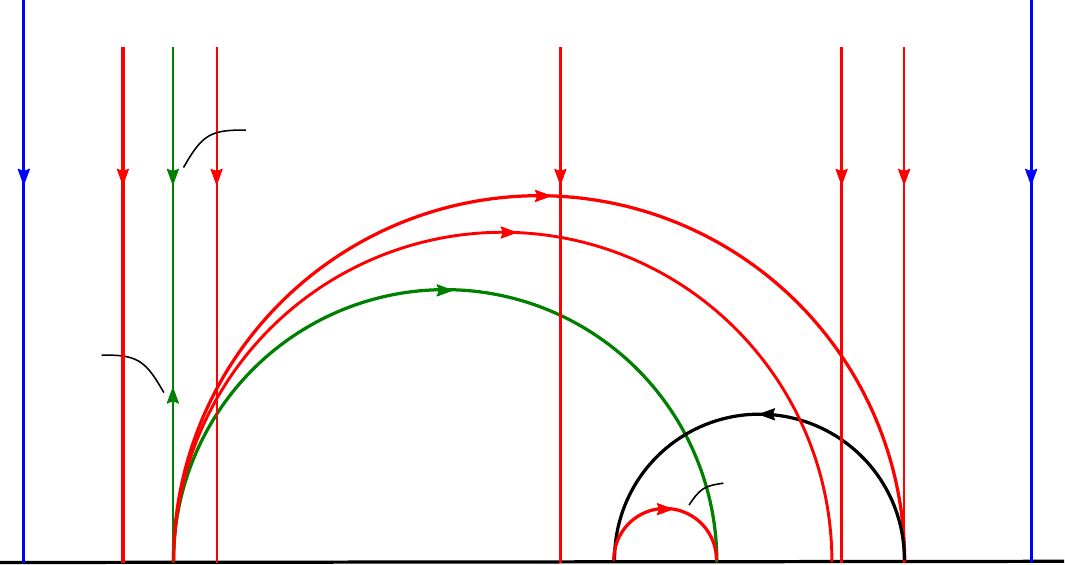}
		\vspace{5pt}
		\caption{Visualizing lifts of $\gammaup{i}_{2k}$, $\alphaup{i}_{2k+1}$ and $\gammaup{i}$ starting from $\widetilde{\infty}_1$.}\label{fig: limit}
	\end{figure}
	
	So it remains to show that $\gammaup{i}$ intersects any ray $\alpha$ satisfying $\alpha<\gammaup{1}$ or $\alpha>\gammaup{n}$ on $I_\tau$. We will focus on the case where $\alpha<\gammaup{1}$. 
	The other case can be proved in a symmetric way.
	
	Recall that $\gammaup{i}_{2k}=\alphaup{i}_{2k}\cdot\gammaup{i+1}_{2k}=\gammaup{i}_{2k-1}\cdot\ell_{p_k+n-i}\cdot\overline{\gammaup{i}_{2k-1}}\cdot\gammaup{i+1}_{2k}$.
	By Lemma \ref{lemma: monotonicity}, we have $\overline{\gammaup{i}_{2k-1}}<\gammaup{1}_{2k-1}<\gammaup{1}$, $\ell_{p_k+n-i}<\bar{\ell}_{p_k+n-i}<\overline{\gammaup{i}_{2k-1}}$ and
	$\gammaup{i}_{2k-1}<\gammaup{i}_{2k}<\gammaup{i+1}_{2k}$. Thus we obtain a lift of $\gammaup{i}_{2k}$ as shown in Figure \ref{fig: limit_2}, whose endpoint converges to the lift of $\tau$ on the left boundary of the fundamental domain $\widetilde{\Omega}$ as $k\to \infty$. 
	By Lemma \ref{lemma: convergence}, as $k$ goes to infinity, the starting point $\widetilde{\infty}_1$ of this lift converges to the endpoint of $\tilde{\gamma}^{(1)}$, the lift of $\gammaup{1}$ starting from $\widetilde{\infty}$. Since $\gammaup{i}<\gammaup{i}_{2k}$, the same convergence of endpoint holds true for the lift of $\gammaup{i}$ starting from $\widetilde{\infty}_1$.
	This provides lifts of $\gammaup{i}$ that intersect any ray $\alpha$ satisfying $\alpha<\gammaup{1}$ and completes the proof.
	
	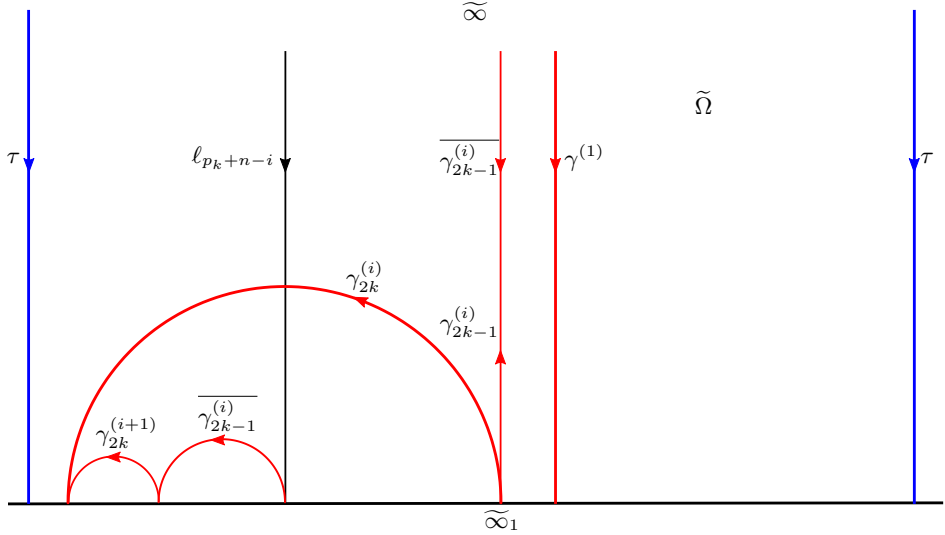
\begin{figure}
		\labellist
		\small 
		\pinlabel $\widetilde{\infty}$ at 255 270
		\pinlabel $\widetilde{\Omega}$ at 380 220
		\pinlabel $\tau$ at 3 190
		\pinlabel $\ell_{p_k+n-i}$ at 123 190
		\pinlabel $\overline{\gammaup{i}_{2k-1}}$ at 252 190
		\pinlabel $\gammaup{i}_{2k-1}$ at 252 100
		\pinlabel $\gammaup{1}$ at 315 190
		\pinlabel $\gammaup{i}_{2k}$ at 195 125
		\pinlabel $\overline{\gammaup{i}_{2k-1}}$ at 120 50
		\pinlabel $\gammaup{i+1}_{2k}$ at 65 40
		\pinlabel $\widetilde{\infty}_1$ at 270 -8
		\pinlabel $\tau$ at 502 190
		\endlabellist
		\centering
		\includegraphics[scale=0.7]{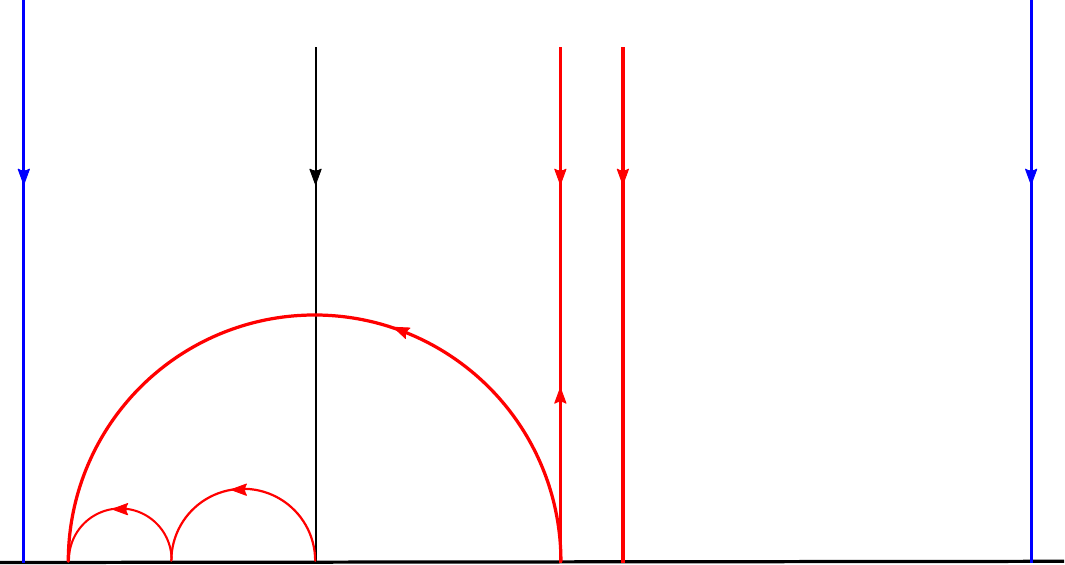}
		\vspace{5pt}
		\caption{Visualizing a lift of $\gammaup{i}_{2k}$.}\label{fig: limit_2}
	\end{figure}
\end{proof}

\begin{rmk}
	By changing the subsurface in the interior of each $\ell_i$ and $r_i$ suitably, the construction gives rise to $2$-filling rays on other surfaces of infinite type.
\end{rmk}

The construction above only produces $2$-filling rays that have \emph{finite} valence in $\Ray$, i.e. their stars in $\Ray$ are \emph{finite} cliques. We do not know the answer to the following question, which is seemingly related to the analogous question for high-filling rays recently solved by Juliette Bavard \cite{infclique}.
\begin{quest}\label{quest: infinite clique}
	Is there an infinite clique of $2$-filling rays in $\Ray$?
\end{quest}

The finite cliques of $2$-filling rays constructed above are only disjoint from a single non-filling long ray $\tau$. 
In Section \ref{sec: uniqueness} we will show that in such a situation, the long ray $\tau$ must be two-side approachable, and thus in this sense all $2$-filling rays of this type come from our construction.

In general, one could also have a finite clique of $2$-filling rays that are disjoint from several different non-filling long rays. 
This certainly can be done on surfaces with non-planar ends. See Section \ref{sec:mnf}.

As the existence of $2$-filling rays makes it more complicated to check whether a ray is high-filling (and thus contributes to a point on the Gromov boundary of the loop graph),
it is natural to ask for a (relatively simple) sufficient condition that guarantees a ray to be high-filling.
Yan Mary He and Kasra Rafi asked whether a ray is high-filling if it is filling in some stronger sense.

We believe only further requiring a filling ray to intersect all closed geodesics does not rule out the possibility that it is 2-filling. We will explain below a modification of the construction above that gives rise to a $2$-filling ray that is filling in this strong sense. 

However, the answer might become positive if we require the ray to intersect all proper geodesics (e.g. including geodesics from a point in the Cantor set to another), as $2$-filling rays might always contain proper geodesics in their limit sets.
\begin{quest}\label{quest: limit set}
	Does the limit set of a $2$-filling ray always contain a proper leaf?
\end{quest}

In the original construction with $n=1$, the $2$-filling ray $\gamma$ we obtain is disjoint from all closed geodesics in the interior of each $r_i$ or $\ell_i$. In general there could be other disjoint closed geodesics if the interiors of $r_i,\ell_i$ do not eventually ``cover'' the entire Cantor set. However, one can avoid this by choosing the two-side approachable long ray $\tau$ and $r_i,\ell_i$ appropriately.

Thus the key is to modify the construction so that $\gamma$ intersects all closed geodesics in the interior of each $r_i$ and $\ell_i$. We explain the modification near $r_1$ below. This same strategy may be applied to the other loops $r_i$ and $\ell_i$ as well. To chop up the interior of $r_1$, we repeatedly cut this disk and the Cantor subset in it into two halves by introducing infinitely many loops, where the new segments correspond to dyadic rational numbers; see Figure \ref{fig: cutCantor}. 
\begin{figure}
	\labellist
	\small 
	\pinlabel $\infty$ at -7 127
	\pinlabel $r_1$ at 80 133
	\pinlabel $1$ at 155 127
	\pinlabel $\frac{3}{4}$ at 155 95
	\pinlabel $\frac{1}{2}$ at 155 65
	\pinlabel $\frac{1}{4}$ at 155 33
	\pinlabel $0$ at 155 0
	\endlabellist
	\centering
	\includegraphics[scale=0.9]{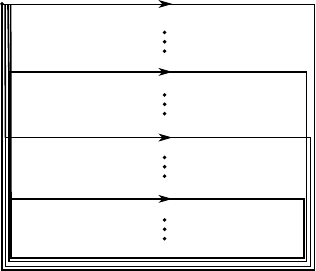}
	\caption{Loops corresponding to dyadic numbers that cut up the Cantor subset in the interior of $r_1$}\label{fig: cutCantor}
\end{figure}
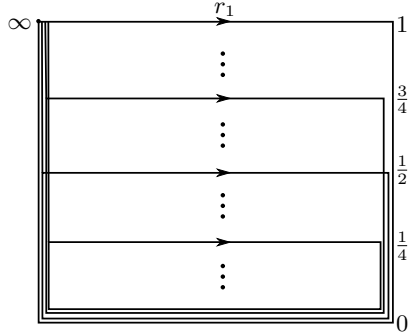

In the iterative construction of a $2$-filling ray $\gamma$ through $\gamma_k$'s (we drop the superscripts as we take $n=1$), 
when we follow $\bar{\gamma}_{2k}$ to fold back and obtain $\gamma_{2k+1}$, we modify it and let some of the segments go inside the interior of $r_1$ following some of the new segments corresponding to dyadic numbers, so that in the end $\gamma$ traverses all these segments corresponding to the dyadic numbers and thus intersects all closed geodesics in the interior of $r_1$. 

An explicit way is to do the fold-back as in the original construction except that, for the strands ``carried'' by the segment corresponding to each positive dyadic number, pull the lowest strand down to traverse the segment corresponding to the closest dyadic number with twice the denominator. Figure \ref{fig: modify} illustrates this in the case of $\gamma_3$ and $\gamma_5$. One can verify that this modification only affects the part where we fold back
following $\bar{\gamma}_{2k}$ (the red portion in the figure). Thus it does not affect the previous $\gamma_i$'s. Adopting this modification to curves near each $r_i$ and $\ell_i$ simultaneously, we should obtain in the limit a $2$-filling ray that also intersects all closed geodesics.

\begin{figure}
	\labellist
	\small 
	\pinlabel $\text{old }\gamma_3$ at 140 162	
	\pinlabel $1$ at 250 290
	\pinlabel $\ell_1$ at 5 265
	\pinlabel $r_2$ at 35 296
	\pinlabel $\frac{1}{2}$ at 250 235
	\pinlabel $0$ at 250 180
	
	\pinlabel $\text{new }\gamma_3$ at 425 162
	\pinlabel $1$ at 535 290
	\pinlabel $\ell_1$ at 290 265
	\pinlabel $r_2$ at 320 296
	\pinlabel $\frac{1}{2}$ at 535 235
	\pinlabel $0$ at 535 180
	
	\pinlabel $\text{old }\gamma_5$ at 140 -8	
	\pinlabel $1$ at 250 127
	\pinlabel $\ell_1$ at 5 92
	\pinlabel $\ell_2$ at -6 112
	\pinlabel $r_2$ at 32 135
	\pinlabel $r_3$ at 21 123
	\pinlabel $\frac{3}{4}$ at 250 95
	\pinlabel $\frac{1}{2}$ at 250 70
	\pinlabel $\frac{1}{4}$ at 250 47
	\pinlabel $0$ at 250 10
	
	\pinlabel $\text{new }\gamma_5$ at 425 -8
	\pinlabel $1$ at 535 127
	\pinlabel $\ell_1$ at 290 92
	\pinlabel $\ell_2$ at 279 112
	\pinlabel $r_2$ at 317 135
	\pinlabel $r_3$ at 306 123
	\pinlabel $\frac{3}{4}$ at 535 95
	\pinlabel $\frac{1}{2}$ at 535 70
	\pinlabel $\frac{1}{4}$ at 535 47
	\pinlabel $0$ at 535 10
	\endlabellist
	\centering
	\includegraphics[scale=0.9]{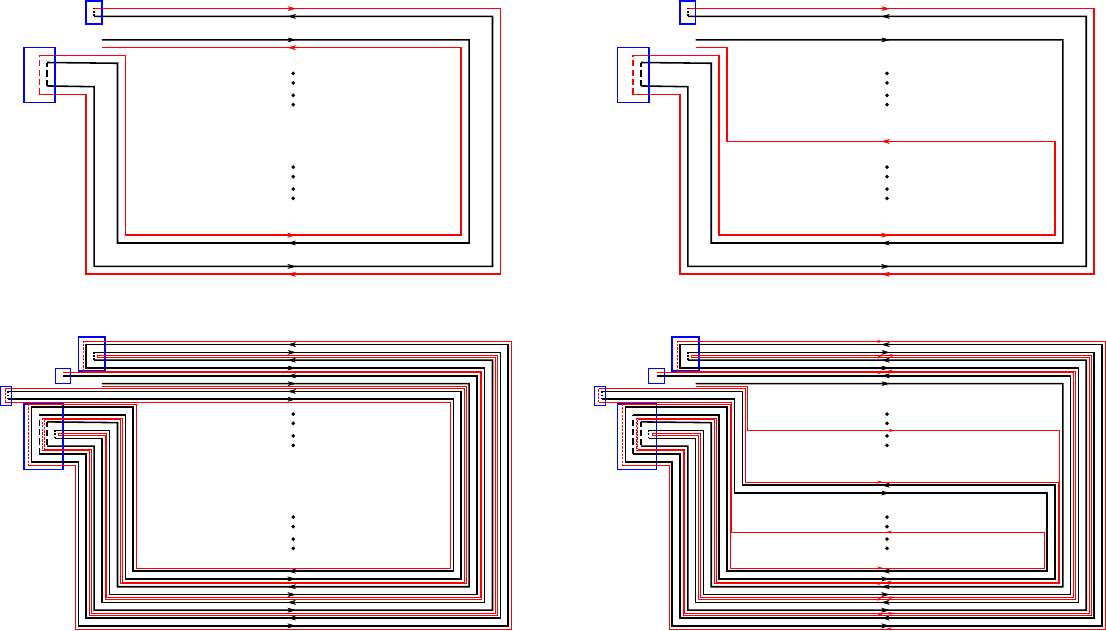}
	\caption{Comparison of $\gamma_3$ and $\gamma_5$ in original construction (left) and in the modified construction (right) near $r_1$. The portion in red represents the part obtained in the iterative construction when the curve folds back. The dotted line in each blue box indicates that the ray is away from $r_1$ and near the loop labeled.}\label{fig: modify}
\end{figure}
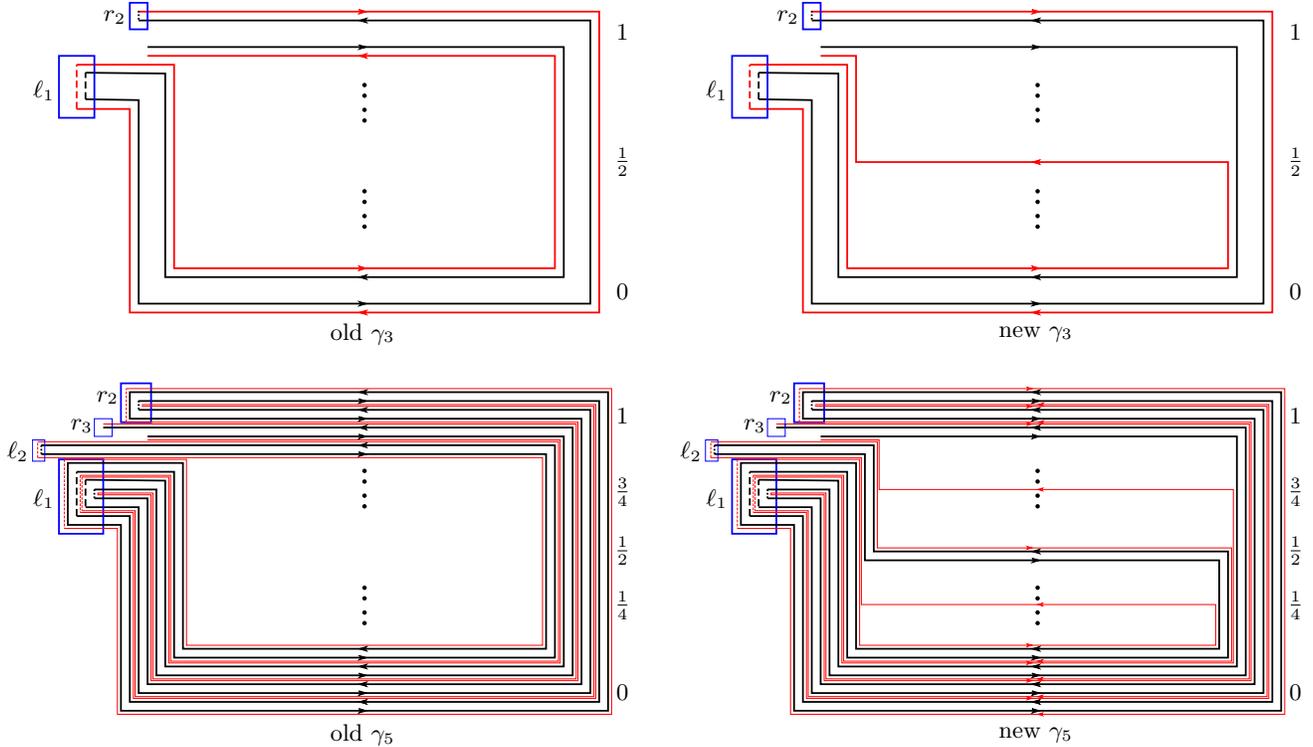

\section{$2$-filling rays abound}\label{sec: 2-fill abound}
	In this section we apply the construction introduced in Section \ref{sec:straightforward} to different two-side approachable long rays, and give a continuum of mapping class group orbits of $2$-filling rays as well as two-side approachable long rays.

	\begin{thm}
		The set of two-side approachable long rays is invariant under the action of the mapping class group $\Gamma$, and there is a continuum of orbits. In particular, there is a continuum of mapping class group orbits of $2$-filling rays.
	\end{thm}
	\begin{proof}
		To see that the set of two-side approachable long rays is invariant, suppose that $\varphi\in \Gamma$ and that the long ray $\tau$ is disjoint from the loops $\ell_i$ and $r_i$ which limit to $\tau$ on the left and right, respectively. Then $\varphi\tau$ is disjoint from the loops $\varphi \ell_i$ and $\varphi r_i$, which limit to $\varphi\tau$ on the left and right, respectively, since $\varphi$ acts on the conical circle $S_C^1$ by orientation-preserving homeomorphisms.
		It suffices to construct a continuum of orbits. Since for each two-side approachable long ray $\tau$ there is a $2$-filling ray only disjoint from $\tau$ by Theorem \ref{thm: construction for two side approachable}, this would also give a continuum of orbits of $2$-filling rays.
	
		We will distinguish the two-side approachable long rays we construct by their limit sets, which we now describe. 
		Fix an infinite increasing sequence of integers $1\le n_1< n_2\cdots$. 
		For each $k$ there is some ideal geodesic $n_k$-gon $P_k$ on $\Omega$, where the interior possibly contains points in the Cantor set.
		We arrange $\{P_k\}$ so that 
		\begin{itemize}
			\item they have disjoint interiors and distinct vertices,
			\item they limit to a single point in the Cantor set, and
			\item each vertex of $P_k$ is accumulated by points of the Cantor set in the exterior of $P_k$.
		\end{itemize}
		See Figure \ref{fig: limitset} for an example in the case $n_k=k$ for all $k$.
	
		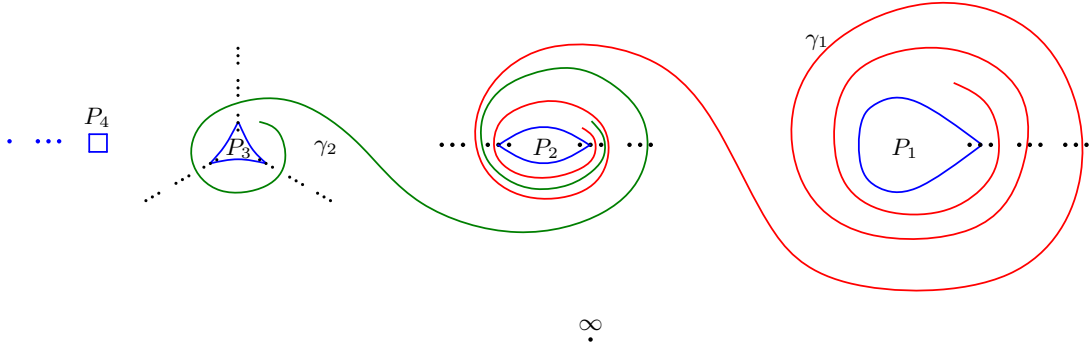
\begin{figure}
			\labellist
			\small 
			\pinlabel $\infty$ at 247 7
			\pinlabel $\gamma_1$ at 343 127
			\pinlabel $P_1$ at 380 82
			\pinlabel $\gamma_2$ at 135 82
			\pinlabel $P_2$ at 228 82
			\pinlabel $P_3$ at 98 82
			\pinlabel $P_4$ at 38 95
			\endlabellist
			\centering
			\includegraphics[scale=0.9]{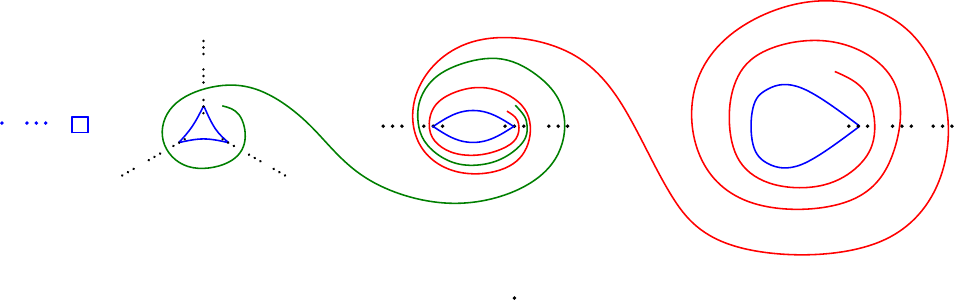}
			\caption{The closed set $L$, which contains the sequence of sets $P_k$ converging to a point in the Cantor set, and also contains bi-infinite geodesics $\gamma_k$ with only $\gamma_1$ and $\gamma_2$ shown in the figure.}\label{fig: limitset}
		\end{figure}
	
		For each $k$, add a bi-infinite geodesic $\gamma_k$ spiraling and limiting to $P_k$ and $P_{k+1}$ respectively at the two ends. 
		We arrange $\{\gamma_k\}$ so that they are mutually disjoint and also disjoint from $\cup_k P_k$. See Figure \ref{fig: limitset}.
	
		Let $L$ be the union of all $\gamma_k$ and $P_k$. Then $L$ is a lamination in $\Omega$. The union of all $P_k$ is the set of leaves in $L$ that are accumulated onto by other leaves. 
		Thus for each sequence $\{n_k\}$, the set of non-isolated leaves of $L$ combinatorially is an infinite subset of $\{n\text{-gon}\}_{n\ge1}$, and any infinite subset appears this way.
		Hence by taking all possible choices of the sequence $\{n_k\}$, we obtain a continuum of mutually nonequivalent geodesic laminations $L$. 
	
		So it suffices to construct a two-side approachable long ray $\tau$ so that its limit set is a given lamination $L$ constructed above. 
		The construction is not sensitive to the choice of $\{n_k\}$, so we will assume $n_k=k$ for simplicity.
	
		By the property of $P_k$, near each end of any $\gamma_k$ we repeatedly see points in the Cantor set to the left (resp. right) of $\gamma_k$. We will use these points in our construction to turn the ray $\tau$ around to the left-hand (resp. right-hand) side after following $\gamma_k$ for a while.
	
		The ray $\tau$ starts out following $\gamma_1$ to spiral around $P_1$. Then turn around to the right-hand side to follow $\gamma_1$ in the opposite direction and spiral around $P_2$. This time turn around to the left-hand side to go back following $\gamma_1$ and spiral around $P_1$ again for a longer time than the first time. Turn around to the right-hand side again following $\gamma_1$ and spiral around $P_2$ for a longer time, and then turn around to the right-hand side following $\gamma_2$ to spiral around $P_3$.
		The ray we obtain up to this step is shown in Figure \ref{fig: tau}.
	
		\begin{figure}
			\labellist
			\small 
			\pinlabel $\infty$ at 235 7
			\pinlabel $\gamma_1$ at 340 55
			\pinlabel $P_1$ at 370 120
			\pinlabel $\gamma_2$ at 133 115
			\pinlabel $P_2$ at 216 120
			\pinlabel $P_3$ at 86 120
			\pinlabel $P_4$ at 26 133
			\endlabellist
			\centering
			\includegraphics[scale=1]{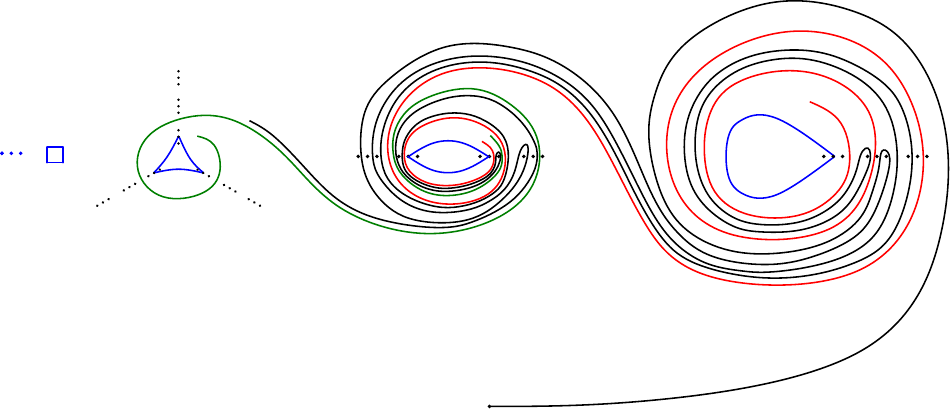}
			\caption{The ray $\tau$ we construct after the first few steps.}\label{fig: tau}
		\end{figure}
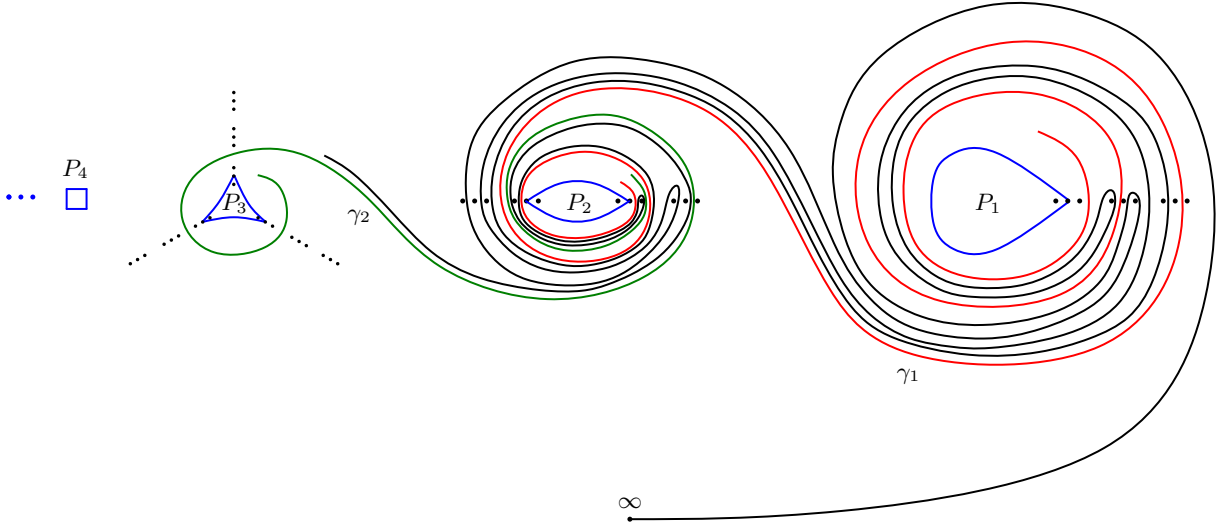
	
		To continue the construction in general, once the ray follows some $\gamma_k$ to spiral around $P_{k+1}$ (e.g. $P_2$ as above) for the first time, we turn it around to the left-hand side to follow $\gamma_k$ in the opposite direction and go back all the way until we are spiraling around $P_1$, where we spiral for a longer time than any previous time. Then there is nothing between $\tau$ up to this point and $\gamma_1$, so we can turn around to the right-hand side and follow $\gamma_1$ to spiral around $P_2$ for a longer time than any previous time, and then follow $\gamma_2$ to spiral around $P_3$ etc, until we follow $\gamma_{k+1}$ and spiral around $P_{k+2}$ for the first time. Now repeat the construction to continue.
	
		In the construction, since the ray $\tau$ spirals around each previously visited $P_k$ along $\gamma_k$ for a longer time than before and $\gamma_k$ limits to $P_k$, we can see that the limit set of $\tau$ is $L$. 
		
		Since every time we turn $\tau$ around in the construction using a set of points in the Cantor set, there are perturbations slightly to the left and right that turn around in the same way but go back all the way to $\infty$. These perturbations can be chosen to be disjoint from $\tau$. Such loops can follow $\tau$ for any desired long time, so they limit to $\tau$ on both sides. This shows that $\tau$ is two-side approachable.
	\end{proof}
	
	The proof shows that a two-side approachable ray could have various kinds of limit sets. We are curious about the following question.
	
	\begin{quest}\label{quest: lamination}
		Which kind of geodesic laminations can appear as the limit set of some two-side approachable long ray? What about $2$-filling rays?
	\end{quest}
	
	We do know that the limit set of a $2$-filling ray properly contains the limit set of some long ray disjoint from it. We first prove the following lemma.
	
	\begin{lemma}\label{lemma: filling limit set}
		For any filling ray $\gamma$ on $\Omega$, its limit set $\Lambda$ contains the limit set of any long ray $\tau$ disjoint from $\gamma$.
	\end{lemma}
	\begin{proof}
		The set $D_\gamma$ of (long) rays disjoint from (or coinciding with) $\gamma$ is a closed subset on the conical circle, and it contains at least two elements if $\tau$ exists. Moreover, $D_\gamma$ is nowhere dense since it is a set of simple rays (see Section \ref{sec: rays}).
		
		Let $(\alpha,\beta)$ be any complementary interval of $D_\gamma$, where we possibly have $\alpha$ or $\beta$ being $\gamma$.
		Fix a lift $\widetilde{\infty}$ of $\infty$ on the universal cover, and fix two consecutive lifts $\tilde{\gamma}_1,\ \tilde{\gamma}_2$ of $\gamma$ starting at
		$\widetilde{\infty}$. Between these two lifts, there is a unique lift $\tilde{\alpha}$ (resp. $\tilde\beta$) of $\alpha$ (resp. $\beta$) starting at $\widetilde{\infty}$.
		Let $a,b$ be the endpoints of $\tilde{\alpha},\tilde{\beta}$ respectively. 
		Let $L$ be the unique bi-infinite geodesic going from $a$ to $b$. Let $p(L)$ be its projection to $\Omega$.
		
		
		Since $\gamma$ intersects all rays in $(\alpha,\beta)$, there are lifts of $\gamma$ converging to $L$ by Lemma \ref{lemma: converge to L}. Thus the closure of $p(L)$ lies in the limit set $\Lambda$ of $\gamma$, which contains the limit sets of $\alpha$ and $\beta$ since $L$ and $\tilde\alpha$ (resp. $\tilde{\beta}$) start (resp. end) at the same point on the boundary.
		
		Now suppose $\tau$ is disjoint from $\gamma$ and is not on the boundary of any complementary interval of $D_\gamma$. 
		Remove the closed half-disk bounded by $L$ from the region between $\tilde{\gamma}_1$ and $\tilde{\gamma}_2$ for all $L$ associated to complementary intervals $(\alpha,\beta)$.
		Denote the resulting set by $\Omega_\gamma$, which is geodesically convex and thus contractible.
		Note that the lift $\tilde{\tau}$ of $\tau$ starting from $\widetilde{\infty}$ lies inside $\Omega_\gamma$.
		
		Consider the limit set $\operatorname{cl}(\tau)\setminus\tau$ of $\tau$ as a geodesic lamination. For any point $x$ in it, let $\ell_x$ be the leaf through $x$.
		Then any lift of $\ell_x$ does not intersect any $L$ associated to a complementary interval $(\alpha,\beta)$ since $\tau$ is disjoint from $\gamma$.
		This implies that any lift of $x$ lies outside $\Omega_\gamma$.
		
		Now for any $\epsilon>0$, there is some $y$ on $\tau$ and a geodesic segment $s$ of length less than $\epsilon$ connecting $x,y$. Let $\tilde{y}$ be the lift of $y$ on $\tilde{\tau}$ and $\tilde{s},\tilde{x}$ the corresponding lifts of $s$ and $x$. Then $\tilde{s}$ connects $\tilde{y}\in \Omega_\gamma$ and $\tilde{x}\notin \Omega_\gamma$, so it must intersect some $L$ associated to a complementary interval $(\alpha,\beta)$. The intersection point has distance to $\tilde{x}$ less than $\epsilon$.
		Since each $p(L)$ lies in $\Lambda$, this shows that $x\in\Lambda$. As $x$ is arbitrary, we conclude that the limit set $\Lambda$ of $\gamma$ also contains the limit set of $\tau$. 
	\end{proof}
	
	\begin{prop}\label{prop: limit set not minimal}
		For any $2$-filling ray $\gamma$ and any long ray $\tau$ that is not filling and disjoint from $\gamma$, the limit set $\Lambda$ of $\gamma$ contains the limit set of $\tau$ as a proper subset. In particular, $\Lambda$ cannot be minimal.
	\end{prop}
	\begin{proof}
		We use the notation as in the proof of Lemma \ref{lemma: filling limit set}. If $\gamma$ is $2$-filling, then there is some long ray $\tau$ that is not filling and disjoint from $\gamma$. For any such $\tau$ its limit set cannot contain $p(L)$ for all complementary intervals, where $L$ is the geodesic constructed above associated to the complementary interval $(\alpha,\beta)$ of $D_\gamma$. This is because otherwise the set of rays disjoint from $\tau$ is a subset of $D_\gamma$, contradicting that $\tau$ is not filling. This shows that the limit set of $\tau$ is properly contained in $\Lambda$. This limit set is non-empty since $\tau$ is not proper.
	\end{proof}
	
\section{$2$-filling rays disjoint from a single non-filling ray}\label{sec: uniqueness}
	In this section we show the construction in Section \ref{sec:straightforward} using two-side approachable long rays is in some sense the unique way to obtain a finite clique of
	$2$-filling rays that have a single non-filling ray disjoint from them. 
	In general, if a $2$-filling ray is only disjoint from finitely many rays, then it is disjoint from an \emph{approachable} long ray.
	
	\begin{defn}\label{def: approachable}
		A long ray $\tau$ is \emph{approachable} if there is a sequence of loops $\ell_i$ disjoint from $\tau$ that converges to $\tau$.
	\end{defn}
	
	One can apply surgeries to the sequence of loops $\ell_i$ in the definition to make them pairwise disjoint and put them in a standard form analogously to Lemma \ref{lemma: strengthen two-side approachable}.
	
	Clearly two-side approachable long rays are approachable. We believe the two notions are not equivalent. In Figure \ref{fig: one_side}, we have two infinite sequences of Cantor subsets converging to certain points in the Cantor set so that at each horizontal level there are two Cantor subsets, a left one and a right one. The depicted ray $\tau$ eventually reaches each level. When $\tau$ visits each level for the first time, the left (resp. right) Cantor set is accessible from the left (resp. right) of $\tau$, which yields a loop slightly to the left (resp. right) of $\tau$ that is disjoint from the previous part of $\tau$. The ray $\tau$ is constructed to fold back later and revisit this level in a way that blocks the access to the right Cantor set from the first visit of $\tau$ without blocking the access to the left Cantor set. Continuing this process, we believe the ray $\tau$ obtained is approachable from its left but not from its right.
	
	\begin{figure}
		\labellist
		\small 
		\pinlabel $\infty$ at 183 6 
		\pinlabel $\tau$ at 350 50 
		\endlabellist
		\centering
		\includegraphics[scale=0.8]{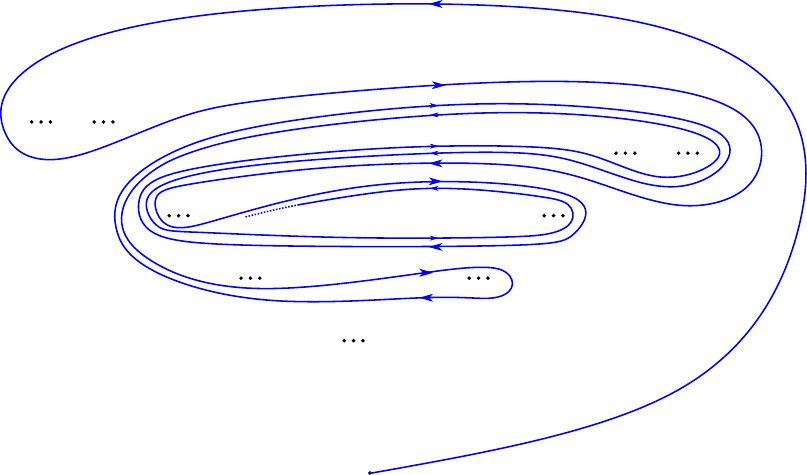}
		\caption{A ray that is approachable from its left but seemingly not approachable from its right.}\label{fig: one_side}
	\end{figure}
	
	\begin{thm}\label{thm: uniqueness of construction}
		Let $\gamma$ be a $2$-filling ray that is disjoint from finitely many rays. Then $\gamma$ is disjoint from an approachable long ray.
		In addition, if only one ray $\tau$ disjoint from $\gamma$ is not filling, then $\tau$ is two-side approachable.
	\end{thm}
	\begin{proof}
		Let $D_\gamma$ be the set of rays disjoint from $\gamma$ including itself. By Lemma \ref{lem:cliques} all rays in $D_\gamma$ are disjoint from each other. Now we think of $D_\gamma$ as a subset of the conical circle, equipped with the induced cyclic order.
		Since $D_\gamma$ is finite by our assumption, we can enumerate it in the cyclic order as $\tau, \gamma_1,\cdots,\gamma_n$ for some $n\ge1$ so that $\tau$ is non-filling and $\gamma_1$ is $2$-filling. 
		
		Pick a fundamental domain $\widetilde{\Omega}$ of the conical cover $\Omega_C$ on $\Hbb^2$ bounded by two consecutive lifts $\tilde{\tau}_1,\tilde{\tau_2}$ of $\tau$ starting from $\widetilde{\infty}$, a lift of $\infty$ on $\partial \Hbb^2$. 
		Let $\tilde{\gamma}_i$ be the lift of $\gamma_i$ starting from $\widetilde{\infty}$ that sits in between $\tilde{\tau}_1$ and $\tilde{\tau_2}$.
		Then $(\tilde{\tau}_1,\tilde{\gamma}_1,\ldots,\tilde{\gamma}_n,\tilde{\tau_2})$ is positively oriented.
		
		For any $0\le i\le n$, let $L_i$ be the bi-infinite geodesic on $\widetilde{\Omega}$ that travels from the endpoint of $\tilde{\gamma}_i$ to the endpoint of $\tilde{\gamma}_{i+1}$, where
		$\tilde{\gamma}_0$ and $\tilde{\gamma}_{n+1}$ denote $\tilde{\tau}_1$ and $\tilde{\tau}_2$ respectively.
		For each ray $\gamma_j$ that is $2$-filling, it is only disjoint from rays in $D_\gamma$ by Lemma \ref{lem:cliques}. Thus there are lifts of $\gamma_j$ converging to each $L_i$ by Lemma \ref{lemma: converge to L}.
		
		\begin{claim}\label{claim: converging infty}
			There is a sequence of lifts $g_k \tilde{\gamma}_1$ of $\gamma_1$ where $g_k \in \pi_1(\Omega)$, 
			such that $g_k \tilde{\gamma}_1$ converges to $L_0$ (on compact sets) and the starting points $g_k\widetilde{\infty}$ converge to the endpoint of $\tilde{\tau}_1$. See Figure \ref{fig: approaching}.
		\end{claim}
		\begin{proof}
			Fix any $p\in L_0$. There is some $g_k\in \pi_1(\Omega)$ and $p_k\in \tilde{\gamma}_1$ such that
			the unit tangent vector $v_k$ of $g_k\tilde{\gamma}_1$ at $g_k p_k$ is arbitrarily close to either the unit tangent vector $v$ of $L_0$ at $p$ or $-v$, 
			as points in the unit tangent bundle of $\mathbb{H}^2$.
			Our claim holds if it happens infinitely often that $v_k$ is close to $v$ instead of $-v$. 
			We show this actually is the case as follows, illustrated as in Figure \ref{fig: direction}.
			
			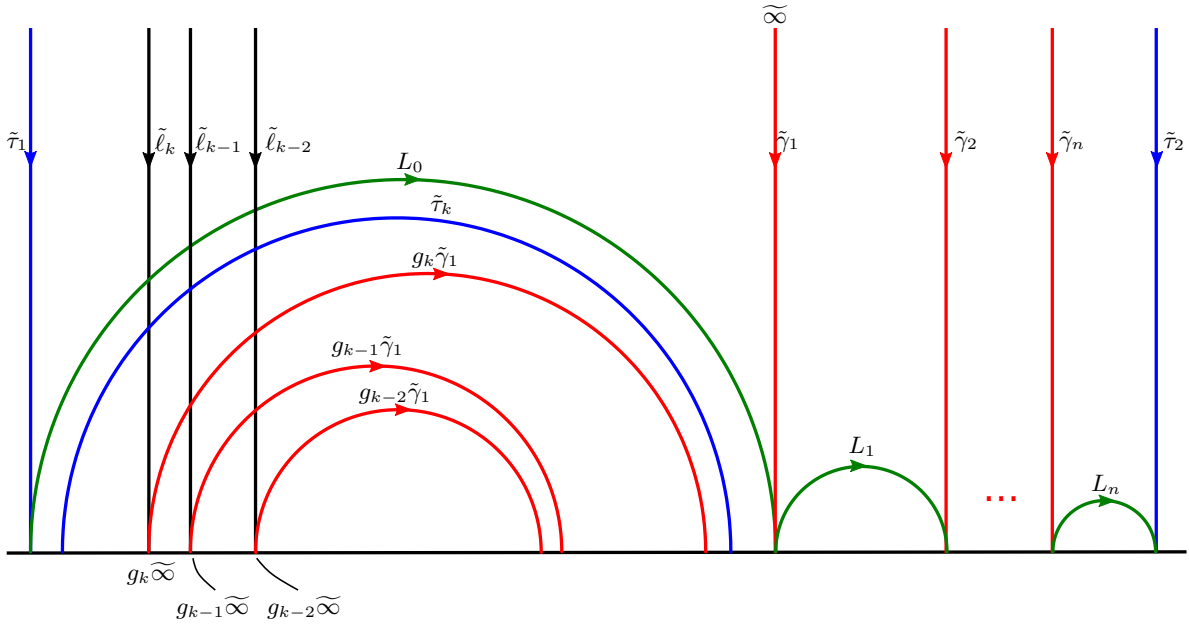
\begin{figure}
				\labellist
				\small 
				\pinlabel $\tilde{\tau}_1$ at 5 212
				\pinlabel $\tilde{\ell}_k$ at 75 212
				\pinlabel $\tilde{\ell}_{k-1}$ at 100 212
				\pinlabel $\tilde{\ell}_{k-2}$ at 133 212
				\pinlabel $g_k\widetilde{\infty}$ at 68 8
				\pinlabel $g_{k-1}\widetilde{\infty}$ at 97 -8
				\pinlabel $g_{k-2}\widetilde{\infty}$ at 140 -8
				\pinlabel $L_0$ at 190 202
				\pinlabel $\tilde{\tau}_k$ at 205 182
				\pinlabel $\tilde{\gamma}_1$ at 368 212
				
				\pinlabel $g_k\tilde{\gamma}_1$ at 202 157
				\pinlabel $g_{k-1}\tilde{\gamma}_1$ at 170 115
				\pinlabel $\widetilde{\infty}$ at 362 272
				\pinlabel $g_{k-2}\tilde{\gamma}_1$ at 183 93
				\pinlabel $L_1$ at 403 68
				\pinlabel $L_n$ at 518 50
				\pinlabel $\tilde{\gamma}_2$ at 452 212
				\pinlabel $\tilde{\gamma}_n$ at 502 212
				\pinlabel $\tilde{\tau}_2$ at 550 212
				\endlabellist
				\centering
				\includegraphics[scale=0.8]{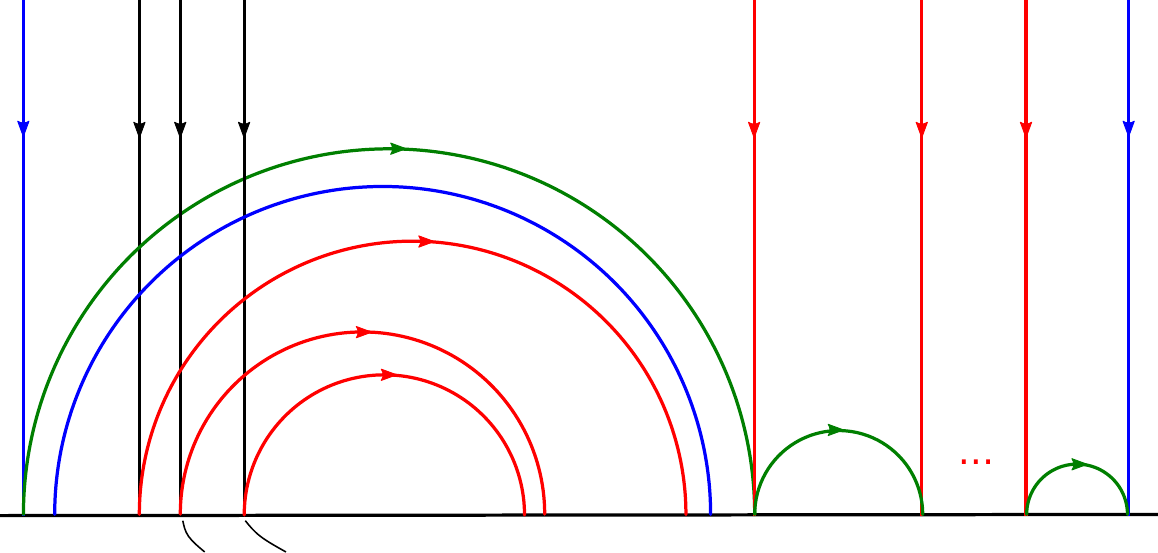}
				\vspace{8pt}
				\caption{The sequence $g_k\tilde{\gamma}_1$ converging to $L_0$ with $g_k\widetilde{\infty}$ converging to the endpoint of $\tilde{\tau}_1$. Any lift of $\tau$ intersecting $\tilde{\ell}_k$ has to be in the position of $\tilde{\tau}_k$.}\label{fig: approaching}
			\end{figure}
			
			Suppose $v_k$ is very close to $-v$ instead of $v$ whenever $g_k p_k$ is close to $p$. Then $p_k$ must be very close to the endpoint of $\tilde{\gamma}_1$, 
			as $g_k^{-1} L_0$ fellow travels with $\tilde{\gamma}_1$ on a very large
			neighborhood of $p_k$ but cannot get close to $\widetilde{\infty}$ since the projection of $L_0$ to $\Omega$ is simple (as a limiting geodesic of a simple ray $\gamma_1$).
			Hence there is some $q_k$ on $L_0$ very close to $p_k$ for $k$ large and the unit tangent vector $u_k$ is very close to $g_k^{-1} v_k$.
			It follows that $g_k u_k$ is slightly to the right of $v_k$ as shown in Figure \ref{fig: direction}. 
			It must sit in between $g_k \tilde{\gamma}_1$ and $L_0$ as in the figure since $g_k L_0$ cannot intersect any lifts of $\tau$ or $\gamma_1$.
			Now for $k'$ large enough, there is some $r_{k'}$ on $g_{k'}\tilde{\gamma}_1$ so that the unit tangent vector $w$ at $r_{k'}$ is arbitrarily close to $-u_k$.
			Hence the unit tangent vector $g_k w$ on $g_k g_{k'}\tilde{\gamma}_1$ is very close to $-g_k u_k$ and $-v_k$, and thus close to $v$ instead of $-v$.
			This contradicts our assumption and proves the claim.
			
			\begin{figure}
				\labellist
				\small 
				\pinlabel $\tilde{\tau}_1$ at 5 195
				\pinlabel $L_0$ at 110 168
				\pinlabel $p$ at 188 185
				\pinlabel $v$ at 220 185
				\pinlabel $\tilde{\gamma}_1$ at 368 195
				\pinlabel $p_k$ at 373 95
				\pinlabel $g_k^{-1}v_k$ at 383 73
				\pinlabel $q_k$ at 328 90
				\pinlabel $u_k$ at 345 55
				\pinlabel $r_{k'}$ at 310 72
				\pinlabel $w$ at 305 114
				
				\pinlabel $g_k\tilde{\gamma}_1$ at 140 35
				\pinlabel $g_kL_0$ at 110 40
				\pinlabel $g_kg_{k'}\tilde{\gamma}_1$ at 110 110
				\pinlabel $g_{k'}\tilde{\gamma}_1$ at 110 130
				\pinlabel $g_kp_k$ at 185 58
				\pinlabel $v_k$ at 160 75
				\pinlabel $g_kq_k$ at 205 75
				\pinlabel $g_ku_k$ at 175 98
				\pinlabel $g_kr_{k'}$ at 215 110
				\pinlabel $g_kw$ at 245 100
				
				\pinlabel $\widetilde{\infty}$ at 362 255
				\endlabellist
				\centering
				\includegraphics[scale=0.8]{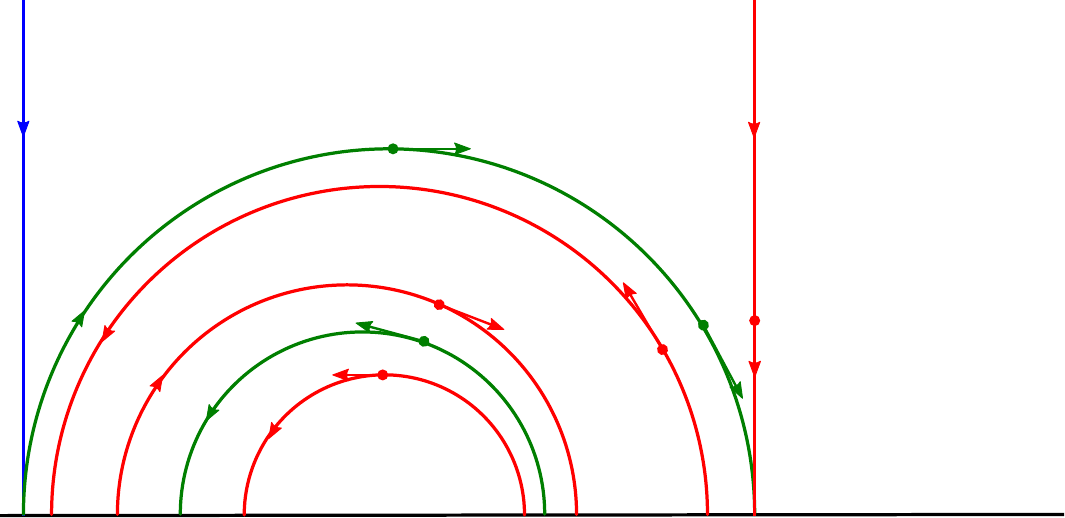}
				\caption{Construction of a lift $g_kg_{k'}\tilde{\gamma}_1$ of $\gamma_1$ that is close to $L_0$ and going in the ``same'' direction.}\label{fig: direction}
			\end{figure}
		\end{proof}
		
		Now we have a sequence of geodesics $\tilde{\ell}_k$ connecting $\widetilde{\infty}$ to $g_k\widetilde{\infty}$ and converging to $\tilde{\tau}_1$ from its left.
		Each $\tilde{\ell}_k$ projects to a (not necessarily simple) loop $\ell_k$ on $\Omega$. 
		
		\begin{claim}\label{claim: disjoint}
			$\tau$ is disjoint from $\ell_k$ for all $k$ sufficiently large.
		\end{claim}
		
		This would imply that $\tau$ is approachable as we can in addition make $\ell_k$ simple in a way similar to Remark \ref{rmk: not simple}. 
		
		If $\gamma_n$ is also $2$-filling, which is the case if $\tau$ is the only non-filling ray in $D_\gamma$, then by a symmetric argument, using $L_n$ in place of $L_0$, there is also a sequence of loops $r_k$ converging to $\tau$ from its right. In this case $\tau$ is two-side approachable since $r_k$ and $\ell_k$ can be made simple by Remark \ref{rmk: not simple}, which completes the proof of the theorem.
	\end{proof}
	
	\begin{proof}[Proof of Claim \ref{claim: disjoint}]
		Suppose that infinitely many $\ell_k$ intersect $\tau$. We will exhibit lifts of $\tau$ converging to $L_i$ for each $0\le i\le n$, from which it follows that $\tau$ is only disjoint from rays in $D_\gamma$, contradicting the fact that $\tau$ is non-filling.
			
		For each $k$ such that $\ell_k$ intersects $\tau$, some lift $\tilde{\tau}_k$ of $\tau$ intersects $\tilde{\ell}_k$. Note that $\tilde{\tau}_k$ is disjoint from $L_0$ and $g_k \tilde{\gamma}_1$, thus $\tilde{\tau}_k$ must be a geodesic in the region between $L_0$ and $g_k \tilde{\gamma}_1$ and isotopic to both. See Figure \ref{fig: approaching}. Thus letting such $k$ go to infinity, we obtain lifts $\tilde{\tau}_k$ of $\tau$ converging to $L_0$.
			
		Now for each $1\le i\le n$, fix a point $p$ on $L_i$. There are lifts $h_m \tilde{\gamma}_1$ of $\gamma_1$ and points $p_m$ on $\tilde{\gamma}_1$ such that $h_m p_m$ converges to $p$ and the tangent lines of $h_m \tilde{\gamma}_1$ at $h_m p_m$ become almost parallel to the tangent line of $L_i$ at $p$. As we explained in the proof of Claim \ref{claim: converging infty}, the point $p_m$ must be very close to the endpoint of $\tilde{\gamma}_1$ for $m$ large.
		Hence there is some $q_m$ on $L_0$ very close to $p_m$.
		Combining with the fact that there is some lift $\tilde{\tau}_k$ of $\tau$ very close to $L_0$, there is some $r_k$ on $\tilde{\tau}_k$ very close to $q_m$ and $p_m$ and such that the tangent lines of $\tilde{\tau}_k$ and $\tilde{\gamma}_1$ at $r_k$ and $p_m$ respectively are almost parallel. See Figure \ref{fig: approaching_i}. Hence $h_m r_k$ is very close to $p$ and the tangent lines of $h_m\tilde{\tau}_k$ and $L_i$ at $h_m r_k$ and $p$ respectively are almost parallel. This exhibits lifts of $\tau$ converging to $L_i$. This completes the proof.
			\begin{figure}
				\labellist
				\small 
				\pinlabel $\tilde{\tau}_1$ at 5 195
				\pinlabel $L_0$ at 50 100
				\pinlabel $\tilde{\tau}_k$ at 65 80
				\pinlabel $\tilde{\gamma}_1$ at 250 195
				\pinlabel $p_m$ at 253 60
				\pinlabel $q_m$ at 232 65
				\pinlabel $r_k$ at 208 45

				\pinlabel $\tilde{\gamma}_i$ at 310 195
				\pinlabel $p$ at 397 103
				\pinlabel $h_mp_m$ at 420 115
				\pinlabel $h_m\tilde{\gamma}_1$ at 340 108
				\pinlabel $h_mr_k$ at 400 65
				\pinlabel $h_m\tilde{\tau}_k$ at 360 38
						
				\pinlabel $L_i$ at 315 68
				\pinlabel $\tilde{\gamma}_{i+1}$ at 505 195
				\pinlabel $\tilde{\tau}_2$ at 550 195
				\endlabellist
				\centering
				\includegraphics[scale=0.8]{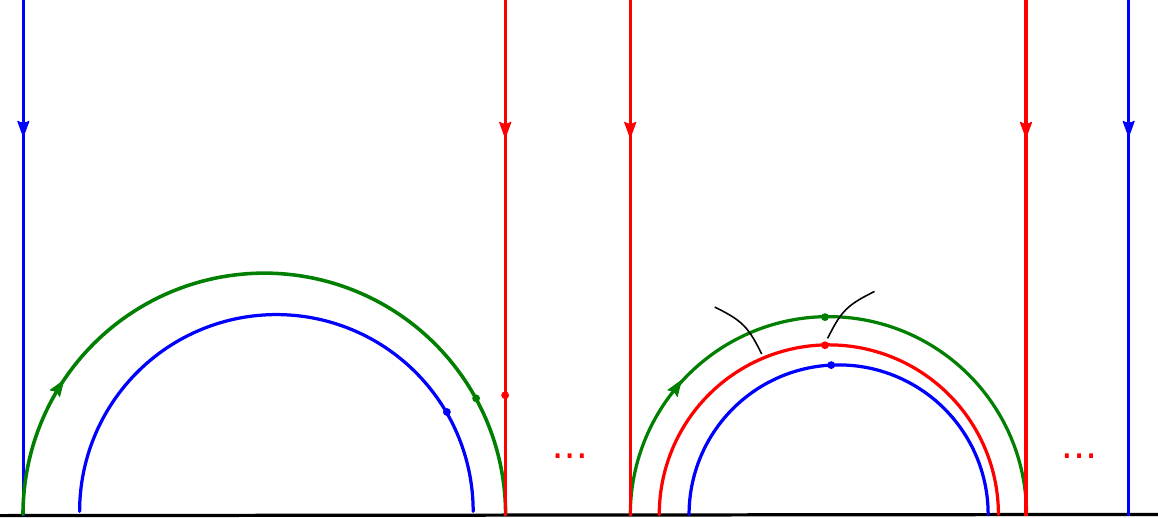}
				\caption{The sequence $h_m\tilde{\tau}_k$ converging to $L_i$}\label{fig: approaching_i}
			\end{figure}
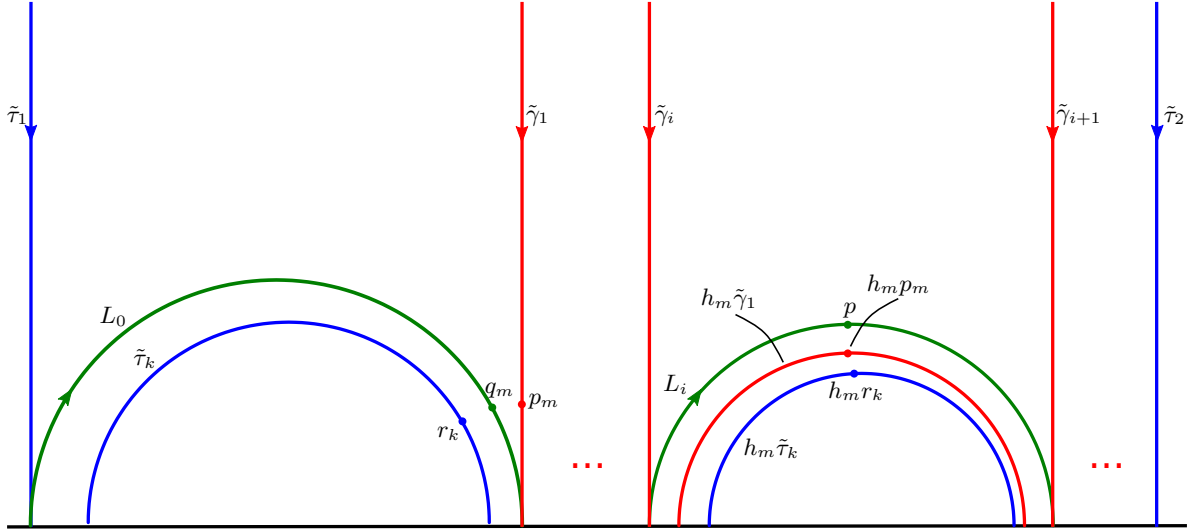
	\end{proof}

In many cases when the star of a $2$-filling ray $\gamma$ is infinite, we can still find an approachable ray disjoint from $\gamma$. We wonder if this is always the case.
\begin{quest}\label{quest: disjoint}
	Is every $2$-filling ray disjoint from some approachable long ray?
\end{quest}

\section{Geodesic laminations from train tracks}
\label{sec:tracks}

In this section we define a geodesic lamination $\Lambda$ on $\Omega$ using a train track. For the statement of Theorem \ref{lamthm}, recall that a geodesic ray $\gamma$ \textit{spirals onto} a lamination $\Lambda$ if $\operatorname{cl}(\gamma) \setminus \gamma=\Lambda$, where $\operatorname{cl}(\gamma)$ denotes the closure.

\begin{restatable}{thm}{lamthm}
\label{lamthm}
There exists a geodesic lamination $\Lambda$ on $\Omega$ with the following properties:
\begin{enumerate}
\item $\Lambda$ has three boundary leaves;
\item the region of $\Omega\setminus \Lambda$ containing $\infty$ is a once-punctured ideal bigon $b$ with ends $e^+$ and $e^-$;
\item every leaf of $\Lambda$ is dense except for a single \emph{proper} leaf $m$;
\item every \emph{half leaf} of $\Lambda$ is dense except for the two half leaves of $m$, and the two half leaves of $\Lambda$ asymptotic to $e^+$;
\item if $\tau$ is a ray from $\infty$ to $e^+$ then $\tau$ spirals onto $m$;
\item if $\gamma$ is a ray from $\infty$ to $e^-$ then $\gamma$ spirals onto $\Lambda$.
\end{enumerate}
\end{restatable}

\begin{cor}
\label{cor:2filling}
The ray $\gamma$ is 2-filling.
\end{cor}

\begin{proof}[Proof of Corollary \ref{cor:2filling}]
Since $\gamma$ spirals onto $\Lambda$, its link in $\Ray$ consists only of the ray $\tau$. In particular, $\gamma$ is not disjoint from any loop or short ray. Hence, to show that $\gamma$ is 2-filling, it suffices to show that $\tau$ is disjoint from a loop. We may choose a loop $\alpha$ disjoint from the leaf $m$ of $\Lambda$. Then since $\tau$ spirals onto $m$, it intersects $\alpha$ at most finitely many times. If $\tau \cap \alpha=\emptyset$ then we have shown that $\tau$ is disjoint from a loop and the proof is complete. Otherwise, orient $\alpha$ and $\tau$. We may choose a point $p\in \alpha\cap \tau$ such that $\alpha|[p,\infty]$ is disjoint from $\tau$. Then the concatenation $\tau|[\infty,p]\cup \alpha|[p,\infty]$ is simple.
It is not homotopic into a neighborhood of $\infty$ since $\tau$ and $\alpha$ are in minimal position. Moreover, it is disjoint from $\tau$ up to homotopy. This completes the proof.
\end{proof}

In this section we introduce the construction of $\Lambda$ via a train track. We will prove Theorem \ref{lamthm} in the following sections. In Section \ref{sec:correspondence} we show that the ray $\gamma$ described in Theorem \ref{lamthm} is actually an instance of one of the 2-filling rays constructed in Section \ref{sec:straightforward}.

\begin{figure}[h]

\begin{center}
\def\svgwidth{\textwidth}
\input{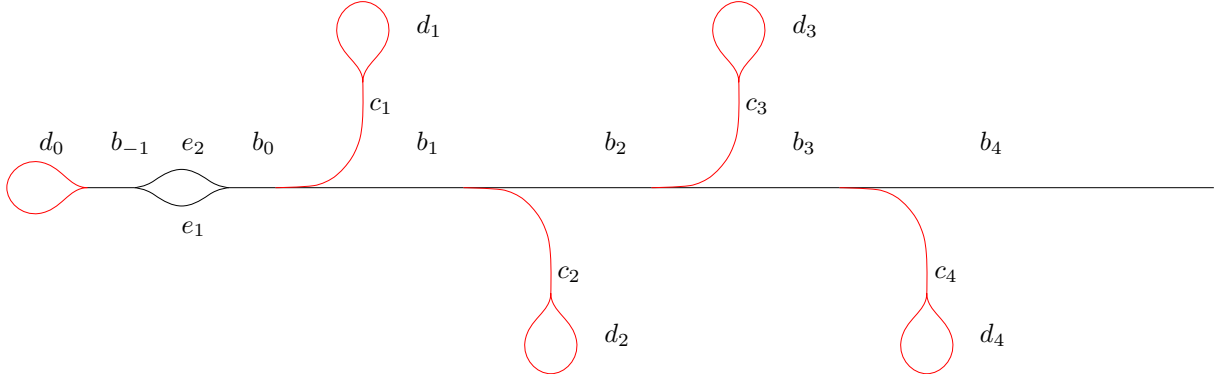}
\end{center}

\caption{The train track $T$ with branches labeled. The track has infinitely many branches stretching to the right in the picture.}
\label{traintrack}
\end{figure}

We define an abstract train track $T$ with branches labeled as in Figure \ref{traintrack}. We define a weight function $w:\mathcal{B}(T)\to [0,\infty)$ as follows. We set $w(e_1)=\frac{1}{3},$ $w(e_2)=\frac{2}{3}$ and \[w(b_{-1})=w(b_0)=1, \ \ w(b_n)=w(c_n)=\frac{1}{2^n} \text{ for } n\geq 1, \ \  w(d_n)=\frac{1}{2^{n+1}} \text{ for } n\geq 0.\]

We associate to the weighted train track $(T,w)$ a corresponding \textit{union of foliated rectangles} $G$. Namely, for each branch $b\in \mathcal{B}(T)$ we associate a rectangle $R(b)$ of width $1$ and height $w(b)$, which is endowed with its natural foliation by horizontal line segments. These rectangles are glued by isometries along their vertical sides in a pattern determined by the train track.

\begin{figure}[h]
\begin{center}
\def\svgwidth{\textwidth}
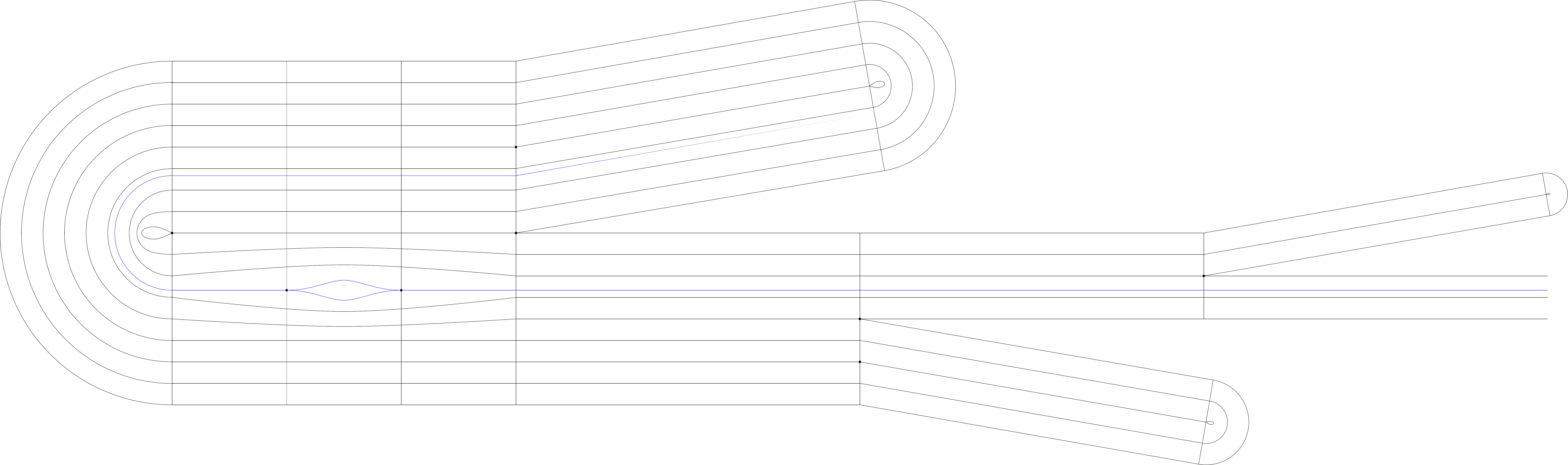
\end{center}
\caption{The union of foliated rectangles $G$.}\label{foliationfromtrack}
\end{figure}

The foliation $G$ determines a space of train paths $\TP(T,w)$ as described in Section \ref{sec:lambg}. We note that for each $n\geq 0$ exactly two distinct points of $R(d_n)$ are identified with each other in $G$ (whereas the natural map $R(d_n)\to G$ is injective on the complement of these two points). We denote by $P_n$ the resulting point. Furthermore, for $n\geq 0$, $R(b_{n+1})$ and $R(c_{n+1})$ are joined at a single point of $G$. We denote by $Q_n$ this point. The points $P_n$ and $Q_n$ are 3-pronged singularities of $G$.

\section{An abstract foliation}
\label{sec:abstractfol}

In this section we define an abstract foliation and investigate its dynamical properties. This will be used in later sections to define a geodesic lamination on the plane minus a Cantor set $\Omega$. Finally, we use this lamination to define a 2-filling ray which spirals onto it.

We consider the unit square $U=[0,1]^2$. It is foliated by the horizontal line segments $[0,1]\times \{y\}$ for $y\in [0,1]$. We will define a singular foliation $F$ by identifying certain segments of the vertical sides of $U$. If $p$ and $q$ are points of $U$ which both lie on a common vertical or horizontal side of $U$, then we denote by $[p,q]$ the subsegment of that side between $p$ and $q$.

We will now describe the side identifications on $U$. First we define a sequence of numbers $y_i\in [0,1]$ as follows. We set $y_{-2}=1$, $y_{-1}=0$, and \[y_n=\frac{1}{2}(y_{n-1}+y_{n-2}) \text{ for } n\geq 0.\] We also define $x_0=\frac{1}{2}$ and \[x_n=\frac{1}{2}(y_{n-1}+y_{n-3}) \text{ for } n\geq 1.\]

To define the side identifications on the right side $\{1\} \times [0,1]$ of $U$ we set $p_0=(1,x_0)=(1,\frac{1}{2})$. We identify the segments $[(1,0),p_0]$ and $[p_0,(1,1)]$ by a rotation of $\pi$ about the point $p_0$.

To define the side identifications on the left side of $U$, we set $p_n=(0,x_n)$ for $n\geq 1$. Furthermore, we set $q_n=(0,y_n)$ for $n\geq -2$. Note that for each $n\geq 1$, $p_n$ lies midway between $q_{n-1}$ and $q_{n-3}$ on the left side $\{0\}\times [0,1]$ of $U$. We identify the segments $[p_n,q_{n-1}]$ and $[p_n,q_{n-3}]$ by a rotation of $\pi$ about the point $p_n$. See Figure \ref{Fdefn}.

\begin{figure}[h]
\begin{center}
\def\svgwidth{\textwidth}
\begingroup%
  \makeatletter%
  \providecommand\color[2][]{%
    \errmessage{(Inkscape) Color is used for the text in Inkscape, but the package 'color.sty' is not loaded}%
    \renewcommand\color[2][]{}%
  }%
  \providecommand\transparent[1]{%
    \errmessage{(Inkscape) Transparency is used (non-zero) for the text in Inkscape, but the package 'transparent.sty' is not loaded}%
    \renewcommand\transparent[1]{}%
  }%
  \providecommand\rotatebox[2]{#2}%
  \newcommand*\fsize{\dimexpr\f@size pt\relax}%
  \newcommand*\lineheight[1]{\fontsize{\fsize}{#1\fsize}\selectfont}%
  \ifx\svgwidth\undefined%
    \setlength{\unitlength}{1680.75093427bp}%
    \ifx\svgscale\undefined%
      \relax%
    \else%
      \setlength{\unitlength}{\unitlength * \real{\svgscale}}%
    \fi%
  \else%
    \setlength{\unitlength}{\svgwidth}%
  \fi%
  \global\let\svgwidth\undefined%
  \global\let\svgscale\undefined%
  \makeatother%
  \begin{picture}(1,0.57162026)%
    \lineheight{1}%
    \setlength\tabcolsep{0pt}%
    \put(0,0){\includegraphics[width=\unitlength,page=1]{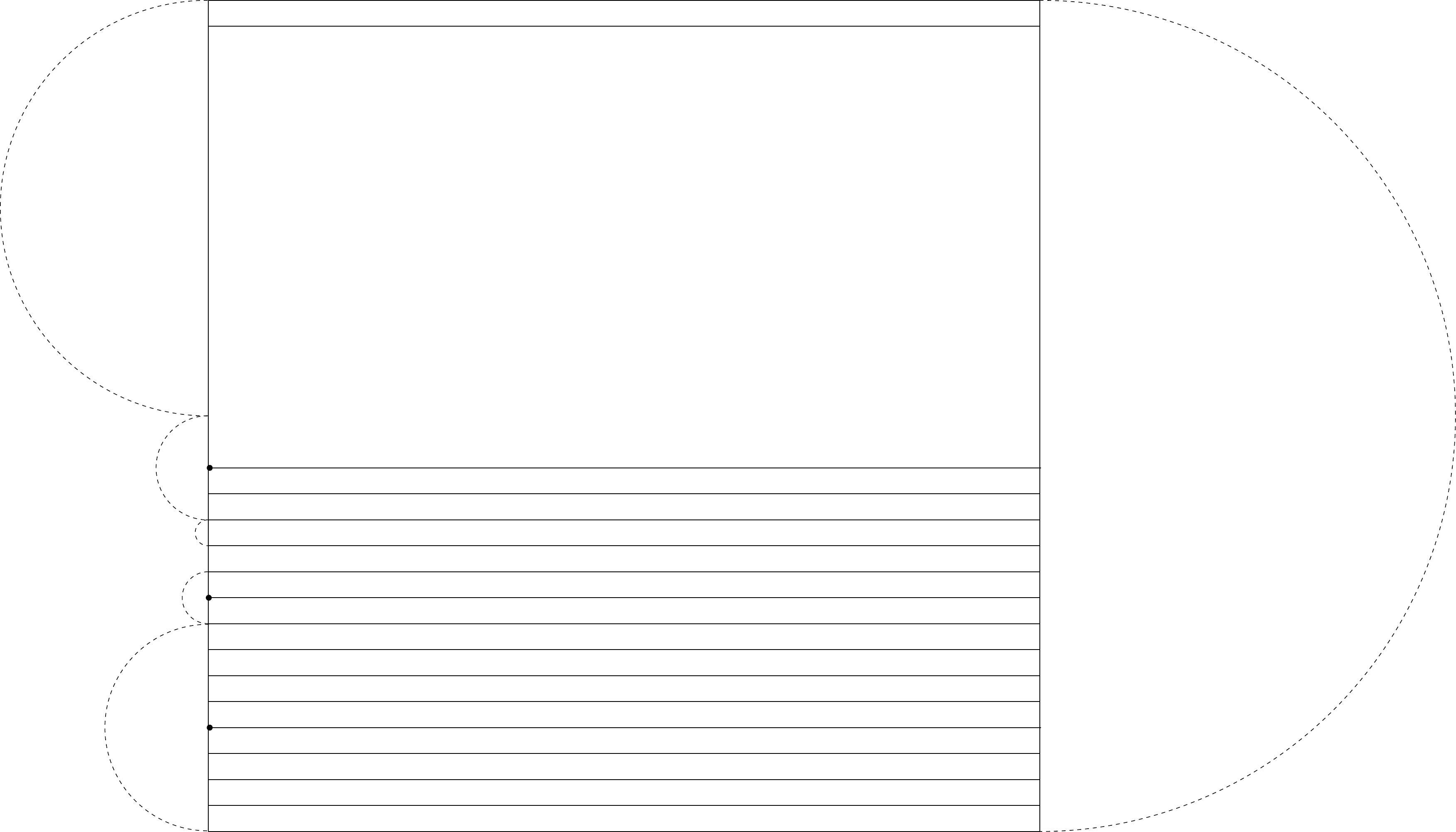}}%
    \put(0.09151306,0.42860345){\color[rgb]{0,0,0}\makebox(0,0)[lt]{\lineheight{1.25}\smash{\begin{tabular}[t]{l}$p_1$\end{tabular}}}}%
    \put(0.09151306,0.07162016){\color[rgb]{0,0,0}\makebox(0,0)[lt]{\lineheight{1.25}\smash{\begin{tabular}[t]{l}$p_2$\end{tabular}}}}%
    \put(0.09151306,0.25011177){\color[rgb]{0,0,0}\makebox(0,0)[lt]{\lineheight{1.25}\smash{\begin{tabular}[t]{l}$p_3$\end{tabular}}}}%
    \put(0,0){\includegraphics[width=\unitlength,page=2]{abstractfoliation.pdf}}%
    \put(0.73755292,0.2858101){\color[rgb]{0,0,0}\makebox(0,0)[lt]{\lineheight{1.25}\smash{\begin{tabular}[t]{l}$p_0$\end{tabular}}}}%
  \end{picture}%
\endgroup%

\end{center}

\caption{The foliation $F$. Dotted lines indicate side identifications by rotation by $\pi$ about the point $p_n$. The leaf through the accumulation point $r$ is shown in blue.}\label{Fdefn}
\end{figure}

We may also write \[y_{2n}=y_{2n-1}+\frac{1}{2^{2n+1}}, \ \ \ y_{2n+1}=y_{2n}-\frac{1}{2^{2n+2}}\] and \[x_{2n}=y_{2n-1}-\frac{1}{2^{2n+1}}, \ \ \ x_{2n+1}=y_{2n}+\frac{1}{2^{2n+2}}.\] From these facts we easily see that $y_n\to \frac{1}{3}$ and $x_n\to \frac{1}{3}$ as $n\to \infty$.


Note that,
\begin{itemize}
\item the sequences $\{p_n\}$ and $\{q_n\}$ both converge to the point $r=(0,\frac{1}{3})$, and
\item $q_{-1},q_1,q_3,\ldots$ have been identified to a single point in the quotient, and
\item $q_{-2},q_0,q_2,\ldots$ have also been identified to a single point. 
\end{itemize}
We finally identify $r$ with the common image of all the $q_i$ to form the topological space $F$, which is Hausdorff. It is homeomorphic to a closed disk by a theorem of Moore (\cite{moore}, see also \cite[Section 7]{generalized}). This fact may also be seen directly. The foliation of $U$ by horizontal lines $[0,1]\times \{y\}$ projects to a singular foliation of $F$. The points $p_n$ each project to 1-pronged singularities of $F$ whereas the point $r$ projects to an ``$\infty$-pronged'' singularity of $F$. We denote by $\pi:U\to F$ the quotient map. Thus $\pi(p_n)=\overline{p_n}$ is a 1-pronged singularity of $F$ and $\pi(r)=\pi(q_{-2})=\pi(q_{-1})=\pi(q_0)=\pi(q_1)=\ldots=\overline{r}$ is an $\infty$-pronged singularity of $F$.

\section{A flat surface and a pseudo-Anosov automorphism}
\label{sec:pA}

In this section we introduce a flat surface $\Sigma$ which is a quotient of $F$ and a pseudo-Anosov automorphism of it. We will use this pseudo-Anosov automorphism to prove facts about the foliation $F$.

The flat surface $\Sigma$ is defined as follows. We consider the unit square $U$ and points $p_n$ and $q_n$ on vertical sides of $U$, defined in the previous section. As before, we identify $[(1,0),p_0]$ and $[p_0,(1,1)]$ by a rotation by $\pi$ and $[p_n,q_{n-1}]$ and $[p_n,q_{n-3}]$ also by rotations by $\pi$. The $q_n$ are also identified with the limit point $r$.

We also consider sequences of numbers $z_n$ and $w_n$ defined by $z_n=1-x_n$ for each $n\geq 0$ and $w_n=1-y_n$ for each $n\geq -2$. We define a point $a_0=(z_0,0)=(\frac{1}{2},0)$ and identify the two segments on the bottom side of $U$, $[(0,0),a_0]$ and $[a_0,(1,0)]$ by a rotation of $\pi$ about the point $a_0$. We identify segments of the top side of $U$ as follows. For $n\geq 1$ we set $a_n=(z_n,1)$ and for $n\geq -2$ we set $b_n=(w_n,1)$. We identify the segments $[a_n,b_{n-1}]$ and $[a_n,b_{n-3}]$ by a rotation of $\pi$ about the point $a_n$. The points $a_n$ and $b_n$ are simply the image of the points $p_n$ and $q_n$ (respectively) under the reflection of $U$ across the diagonal line from $(0,1)$ to $(1,0)$. Finally, we identify all the $b_n$ with the limit point $(\frac{2}{3},1)$.

%

The surface $\Sigma$ is the quotient of $U$ under all the above identifications. It is indeed a surface and in fact homeomorphic to a sphere (again, see \cite{moore} and \cite[Section 7]{generalized}). The quotient $\Sigma$ inherits a flat metric away from the singularities and vertical and horizontal foliations $\mathcal{F}^v$ and $\mathcal{F}^h$, respectively, from the foliations of $U$ by vertical and horizontal line segments. There is an obvious quotient $\rho:F\to \Sigma$ as well. See Figure \ref{pA} for a picture of $\Sigma$.

\begin{figure}[h]
\begin{center}
\def\svgwidth{\textwidth}
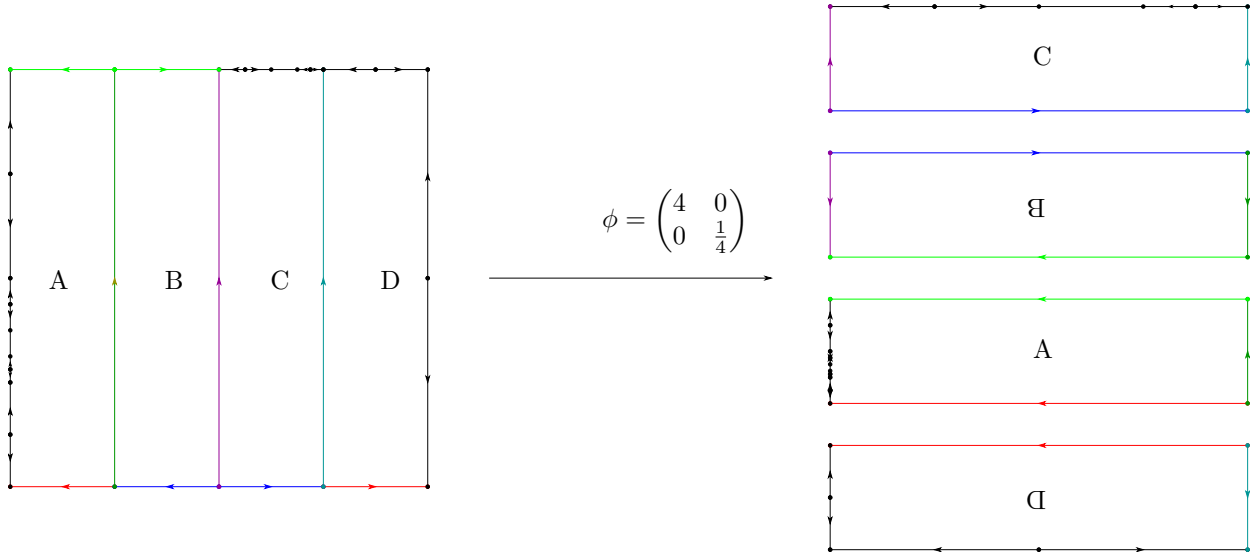
\end{center}

\caption{The surface $\Sigma$ together with the pseudo-Anosov automorphism $\phi$. Here, adjacent arrows which point away from each other are identified by a rotation by $\pi$.}
\label{pA}
\end{figure}

The horizontal foliation of $\Sigma$ is essentially the same as the foliation $F$. The only difference is that one of the singular leaves of $F$ (the singular leaf containing the quotient of the horizontal sides $[0,1]\times \{0\}$ and $[0,1]\times \{1\}$ of $U$) in the quotient $\Sigma$ consists of infinitely many saddle connections joining 1-pronged singularities to the $\infty$-pronged singularity.

The surface $\Sigma$ admits a pseudo-Anosov automorphism $\phi$ defined as follows. Consider the four sets \[A=\left[0,\frac{1}{4}\right]\times [0,1], \ B=\left[\frac{1}{4},\frac{1}{2}\right]\times [0,1], \ C=\left[\frac{1}{2},\frac{3}{4}\right]\times [0,1], \ D=\left[\frac{3}{4},1\right]\times [0,1].\] These are subrectangles of $U$ meeting along their vertical sides. Consider the following operations:
\begin{itemize}
\item cut $U$ into the subrectangles $A,B,C,D$,
\item apply the matrix $\begin{pmatrix} 4 & 0 \\ 0 & \frac{1}{4}\end{pmatrix}$ to each subrectangle $A,B,C,D$,
\item apply a rotation by $\pi$ to the subrectangles $B$ and $D$,
\item stack $C$ on top of $B$ on top of $A$ on top of $D$.
\end{itemize}
It is shown in Figure \ref{pA} that this descends to a well-defined automorphism $\phi$ of $\Sigma$.

We remark that $\phi$ arises from the \textit{generalized pseudo-Anosov} construction of de Carvalho--Hall; see \cite{generalized} (however, our automorphism $\phi$ arises from an interval endomorphism which is \textit{not} unimodal).

\begin{lemma}
\label{lem:saddleconns}
There is a saddle connection on $F$ from $\overline{p_n}$ to $\overline{r}$ for each $n\geq 0$ (see Figure \ref{saddleconns}).
\end{lemma}

\begin{figure}[h]

\begin{center}
\def\svgwidth{0.8\textwidth}
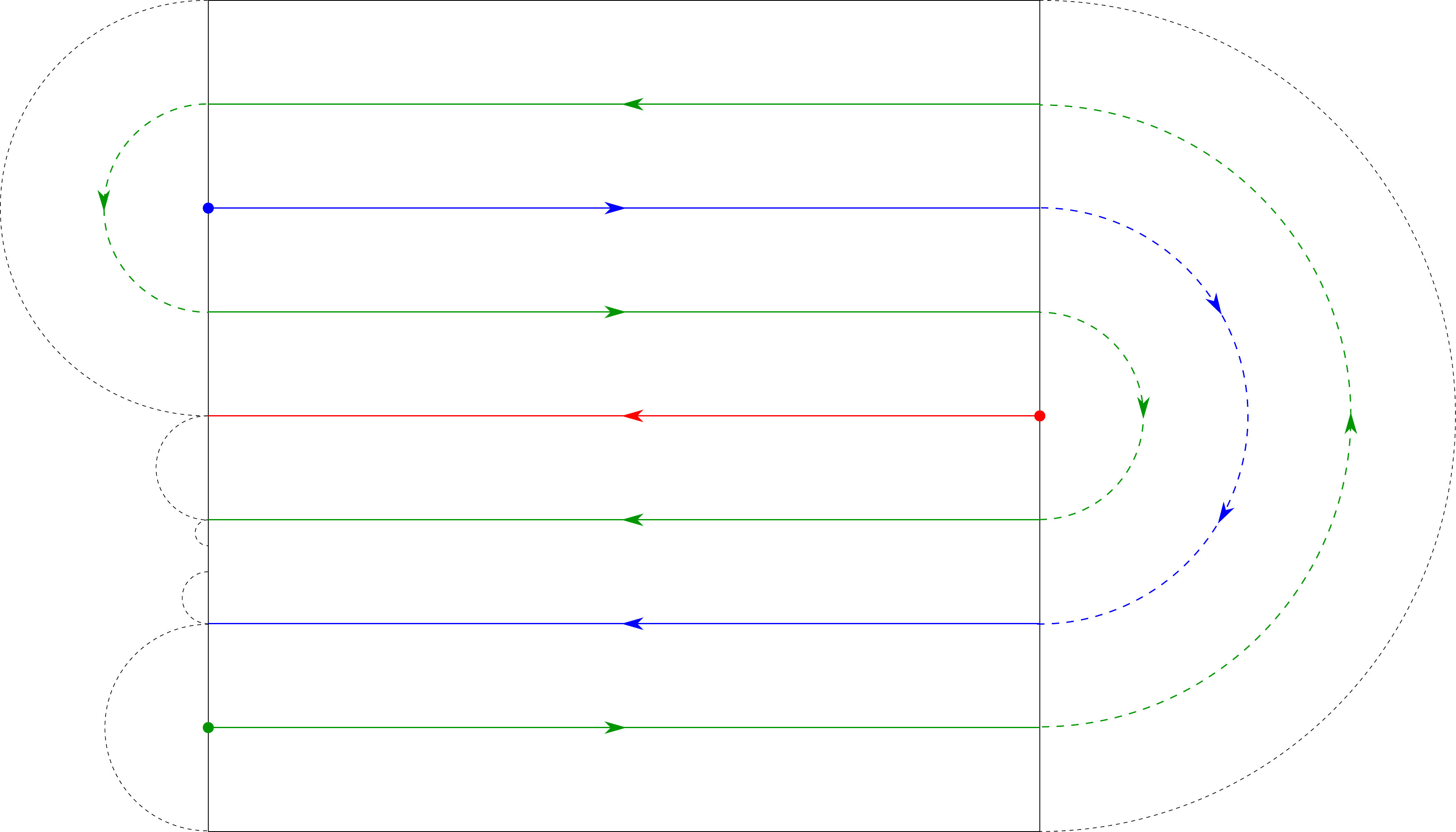
\end{center}

\caption{The horizontal saddle connections of the foliation $F$.}
\label{saddleconns}
\end{figure}

\begin{proof}
In this proof we will conflate $a_n,b_n,p_n,$ and $q_n$ with their images under $\rho \circ \pi$ in $\Sigma$. The orbit of singularities of $\Sigma$ under iteration of $\phi$ is illustrated in Figure \ref{singorbit}.
\begin{figure}[h]
\begin{center}
\def\svgwidth{0.5\textwidth}
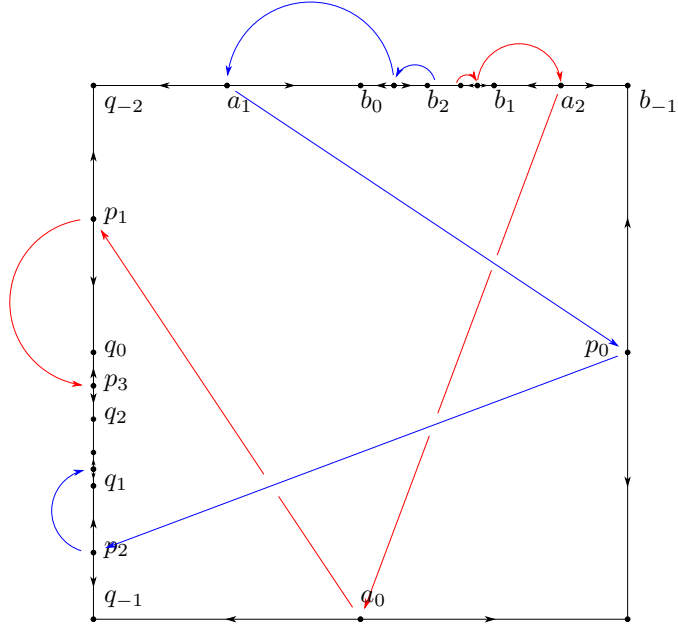
\end{center}

\caption{The orbits of singularities under $\phi$.}
\label{singorbit}
\end{figure}

In particular we see that, under iteration of $\phi$, \[\cdots a_5 \mapsto a_3 \mapsto a_1 \mapsto p_0 \mapsto p_2\mapsto p_4 \mapsto p_6 \mapsto \cdots\] and \[\cdots a_4 \mapsto a_2 \mapsto a_0 \mapsto p_1 \mapsto p_3 \mapsto p_5 \mapsto \cdots.\] We also see that \[ \cdots b_2 \mapsto b_0 \mapsto b_{-2} = q_{-2} \mapsto q_0 \mapsto q_2 \mapsto q_4 \mapsto \cdots\] and \[ \cdots        b_1 \mapsto b_{-1} \mapsto q_{-1} \mapsto q_1 \mapsto q_3 \mapsto q_5 \mapsto \cdots.\] The horizontal foliation of $\Sigma$ clearly contains a saddle connection from $p_0$ to $q_0$. Since $\phi^k(p_0)=p_{2k}$, $\phi^k(q_0)=q_{2k}$ for all $k\ge0$ and $\phi$ preserves the horizontal foliation of $\Sigma$, we see that there is a saddle connection from $p_n$ to $q_n$ whenever $n\geq 0$ is even. There is also clearly a horizontal saddle connection from $a_0$ to $q_{-1}$. Since $\phi^k(a_0)=p_{2k-1}$ and $\phi^k(q_{-1})=q_{2k-1}$ for all $k\ge1$, this shows that there is a horizontal saddle connection from $p_n$ to $q_n$ whenever $n\geq 0$ is odd.

It is easy to see that each of the above described saddle connections from $\rho(\pi(p_n))$ to $\rho(\pi(q_n))$ is the image of a saddle connection of $F$. Since $\overline{r}=\overline{q_{-2}}=\overline{q_{-1}}=\overline{q_0}=\overline{q_1}=\cdots$ on $F$ this implies that there is a saddle connection from $\overline{p_n}$ to $\overline{r}$ for each $n$.
\end{proof}

\begin{lemma}
\label{singdense}
The union of the singular leaves of $F$ is dense in $F$.
\end{lemma}

\begin{proof}[Easy proof of Lemma \ref{singdense}]
The saddle connection from $\overline{p_0}$ to $\overline{r}$ is the horizontal line segment $\pi([0,1]\times \{1/2\})$. By an easy induction, using the description of $\phi$ via cutting and restacking, we see that the saddle connection from $\overline{p_i}$ to $\overline{r}$ is the union of the horizontal line segments $\pi([0,1]\times \{\frac{j}{2^{i+1}}\})$ 
where $j$ ranges over the odd integers between $0$ and $2^{i+1}$ (see Figure \ref{saddleconns} for the first few of these saddle connections). Thus, every horizontal line segment of $F$ of the form $\pi([0,1]\times \{y\})$ where $y$ is a dyadic rational between 0 and 1 lies on a singular leaf of $F$. Since the dyadic rationals are dense in $[0,1]$, this proves the statement.
\end{proof}

We also include a less explicit, more dynamical proof of Lemma \ref{singdense} that may be of use to readers interested in generalizing the constructions of this paper. 

\begin{proof}[Dynamical proof of Lemma \ref{singdense}]
It suffices to prove the following: For an arbitrary transversal $t$, there exists a singular leaf which intersects $t$.

For convenience, we may choose $t$ to be a subsegment of the transversal $s=\{1/2\}\times [0,1]$. First, we claim that there exists a \textit{nonsingular} leaf $l$ of $F$ which intersects $t$ at least twice. For this, we consider the following interval exchange transformation (IET) $f:[0,1]\to [0,1]$. For each $n\geq 2$ $f$ is defined by sending the interval between $y_n$ and $y_{n+2}$ by a translation to the interval between $1-y_n$ and $1-y_{n+2}$. The IET $f$ may be extended in an arbitrary way to the endpoints $y_n$ and their accumulation point $1/3$. See Figure \ref{ietdefn}. The resulting map $f$ preserves the Lebesgue measure on $[0,1]$. Thus we may apply Poincar\'{e} Recurrence to it.

We flow the leaves of $F$ to the right from $s$. Note that every nonsingular leaf of $F$ intersects $s$ infinitely many times. Moreover, if the leaf through a point $(1/2,x)$ of $s$ is nonsingular, then the \textit{second} return under this flow of $(1/2,x)$ to $s$ is given by \[(1/2,x)\mapsto (1/2,f(x)).\] We may write $t=\{1/2\}\times [a,b]$. By Poincar\'{e} Recurrence, almost every point of $[a,b]$ with respect to Lebesgue measure returns to $[a,b]$ infinitely many times under iteration of $f$. Since only countably many points of $t$ intersect singular leaves of $F$, this implies that there is a point $x\in [a,b]$ such that the leaf of $F$ through $(1/2,x)$ is nonsingular and $x$ returns to $[a,b]$ infinitely many times under iteration of $f$. Consequently, the nonsingular leaf $l$ of $F$ through $(1/2,x)$ returns to $t$ infinitely many times.

\begin{figure}[h]
\begin{center}
\def\svgwidth{0.8\textwidth}
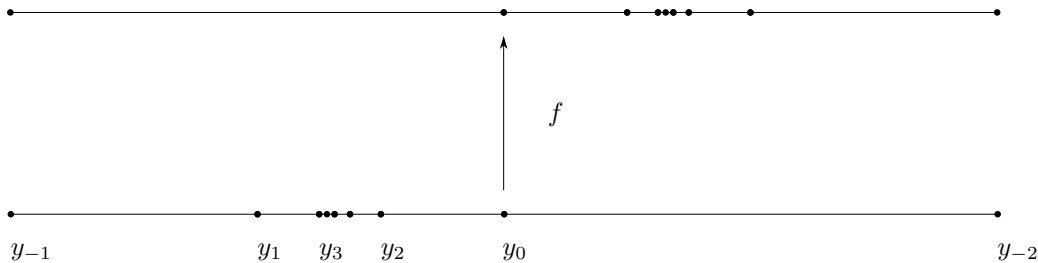
\end{center}

\caption{The IET $f$ described in the proof of Lemma \ref{singdense}.}
\label{ietdefn}
\end{figure}

Now consider a subsegment $l_0\subset l$ which intersects $t$ at its endpoints $v$ and $w$ and nowhere in its interior. Then $c=t|[v,w]\cup l_0$ is a simple closed curve. It bounds a disk in $F$ containing some finite number of 1-pronged singularities and not containing the $\infty$-pronged singularity in its interior. The winding number of the foliation about $c$ is $1/2,1,$ or $3/2$. By the Poincar\'{e}--Hopf Index Theorem, this is equal to half the number of 1-pronged singularities inside the disk bounded by $c$. In particular, this disk contains at least one 1-pronged singularity $\overline{p_n}$. Since this disk does not contain $\overline{r}$, the saddle connection from $\overline{p_n}$ to $\overline{r}$ must intersect $c$. Since this saddle connection does not intersect $l$, it must intersect $t|[v,w]$. This completes the proof.
\end{proof}

Recall that a \emph{transverse measure} to $F$ assigns to each arc $s$ transverse to the foliation a finite \emph{Borel measure} $\mu|s$. It is required to be invariant under leaf-preserving isotopies. See for instance \cite{primer} for more information.

\begin{lemma}\label{transmeas}
Let $\mu$ be a transverse measure to the horizontal foliation of $\Sigma$. Then either $\mu$ has an atom along a horizontal saddle connection of $\Sigma$ or $\mu$ is a multiple of Lebesgue measure.
\end{lemma}

\begin{proof}
The sets $A,B,C,D$ form a Markov partition for the automorphism $\phi$. Using the fact that $\phi$ admits a finite Markov partition, a proof identical to that of \cite[Theorem 12.1]{flp} shows that if $\mu$ is a transverse measure to $\mathcal{F}^h$ and $\mu$ has no atoms, then $\mu$ is a multiple of Lebesgue measure.

It remains to show that if $\mu$ is a transverse measure to $\mathcal{F}^h$ and $\mu$ has an atom then it has an atom along a singular leaf. It is easy to see that if $\mu$ has an atom then $\mu|s$ has an atom where $s$ is the vertical transversal $\pi(\{1/2\}\times [0,1])$. If $\pi(1/2,x)$ is such an atom of $\mu|s$ and $\pi(1/2,x)$ lies on a nonsingular leaf $l$ of $\mathcal{F}^h$ then $\mu|s$ is infinite since $l$ intersects $s$ infinitely many times. This is a contradiction.
\end{proof}

\section{Train path properties}
\label{sec:trainpaths}

Recall that $G$ is the union of foliated rectangles constructed from the weighted train track $(T,w)$. In this section we investigate properties of the train paths $\TP(T,w)$. There is a Lipschitz quotient map $\Pi:G\to F$, sending leaves of $G$ to leaves of $F$, defined as follows:
\begin{itemize}
\item the leaves of the rectangles $R(b_n)$, $R(c_n)$, and $R(d_n)$ are collapsed to points for each $n\geq 1$,
\item the rectangles $R(b_{-1}), R(b_0), R(e_1), $ and $R(e_2)$ are shrunk horizontally to a width of $\frac{1}{3}$,
\item the upper horizontal side of $R(e_1)$ is identified isometrically with the lower horizontal side of $R(e_2)$.
\end{itemize}
We see immediately that $\Pi(P_n)=\overline{p_n}$ and $\Pi(Q_n)=\overline{r}$ for each $n\geq 0$.

By Lemma \ref{lem:saddleconns} and its proof we immediately obtain:

\begin{cor}
\label{pqscs}
For each $n\geq 0$ there is a saddle connection in $G$ from $P_n$ to $Q_n$.
\end{cor}

By Lemma \ref{singdense} we immediately obtain:

\begin{cor}
Saddle connections are dense in $G$.
\end{cor}

Finally, the following is clear by inspection:

\begin{lemma}
\label{qqscs}
There is a saddle connection from $Q_1$ to $Q_0$. Moreover, for each $n\geq 2$ there is a saddle connection from $Q_n$ to $Q_{n-2}$. Finally, there is a saddle connection from $P_n$ to itself for each $n\geq 0$.
\end{lemma}

\begin{lemma}
\label{lem:boundarypaths}
The system of train paths $\TP(T,w)$ contains three boundary paths.
\end{lemma}

\begin{proof}
Refer to Figures \ref{traintrack} and \ref{foliationfromtrack} for the proof.
The upper horizontal side of $R(e_1)$ and the lower horizontal side of $R(e_2)$ form a bigon and give rise to two boundary paths in $\TP(T,w)$.

Since there is a loop based at $P_0$ there is a single boundary train path of $\TP(T,w)$ corresponding to the monogon with vertex at $P_0$. This train path corresponds to a leaf $l_0$ in $G$ which decomposes as $l_l * l_b * l_r$ where $l_b$ is the loop based at $P_0$. The path $l_l$ has the form
\[\ldots \to P_5 \to Q_5 \to Q_3\to P_3 \to P_3\to Q_3\to Q_1\to P_1\to P_1\to Q_1 \to Q_0 \to P_0\]
where each arrow $\to$ denotes a saddle connection. Similarly, $l_r$ has the form \[P_0 \to Q_0 \to Q_2 \to P_2 \to P_2 \to Q_2 \to Q_4 \to P_4 \to P_4 \to Q_4 \to Q_6 \to P_6 \to \ldots.\] Thus $l_0$ visits each singularity $P_n$ and $Q_n$ exactly twice.

For each singularity $P_n$ or $Q_n$, there are at most two boundary paths corresponding to leaves in $G$ which pass through that singularity. Moreover, there is exactly one if there is a boundary path corresponding to a path in $G$ which passes through that singularity twice. Since $l_0$ passes through each singularity $P_n$ and $Q_n$ exactly twice, the train path corresponding to it is the only boundary path besides the two already mentioned.
\end{proof}

We continue to denote by $l_0$ the leaf in $G$ which passes through all the singularities $P_n$ and $Q_n$. We denote by $l^+$ and $l^-$ the leaves corresponding to the sides of the bigon of $G$. Denote by $L_0,L^+,$ and $L^-$ the corresponding train paths in $\TP(T,w)$, respectively.

\begin{lemma}
The path $l_0$ is dense in $G$.
\end{lemma}

\begin{proof}
This is clear since saddle connections are dense in $G$ and $l_0$ traverses each saddle connection.
\end{proof}

\begin{lemma}
\label{halfleaves}
Every half train path of $\TP(T,w)$ is dense except for a half train path of $L^+$ and a half train path of $L^-$. For these half train paths, they simply traverse the branches $b_0,b_1,b_2,\ldots$ in order.
\end{lemma}

\begin{proof}
If $l$ is either:
\begin{itemize}
\item one of the boundary leaves $l^+$ or $l^-$ or
\item a nonsingular leaf of $G$,
\end{itemize}
then $\Pi(l)$ is a non-singular leaf of the foliation $F$. We consider the vertical transversal $s=\pi(\{\frac{1}{2}\}\times [0,1])$ to $F$. The leaf $\Pi(l)$ contains a ray which intersects $s$ infinitely many times. We consider the subsequent points of intersection $v_1, v_2, v_3,\ldots$ and the sequence of counting measures \[\mu_n=\frac{1}{n} \sum_{i=1}^n \delta_{v_i}\] on $s$, where $\delta_v$ is the Dirac unit mass at the point $v$. Up to taking a subsequence, the sequence $\mu_n$ converges to a measure $\mu$ on $s$. We see that if two subsegments $u_1$ and $u_2$ of $s$ are isotopic via a leaf-preserving isotopy then $\mu(u_1)=\mu(u_2)$. Thus, by translating arbitrary transversals to $s$ via leaf-preserving isotopies, we see that $\mu$ induces a transverse measure to $F$. Furthermore, by taking the quotient $\rho:F\to \Sigma$, we see that $\mu$ induces a transverse measure $\rho_* \mu$ to the horizontal foliation of $\Sigma$.

There are two possibilities by Lemma \ref{transmeas}. If $\rho_*\mu$ has no atoms along a horizontal saddle connection of $\Sigma$ then it is a multiple of Lebesgue measure. Hence $\mu$ itself is a multiple of Lebesgue measure along $s$. This proves that $l\cap s$ is a dense subset and therefore $l$ is dense in $G$. Otherwise, $\rho_*\mu$ has an atom along a horizontal saddle connection of $\Sigma$. We see immediately that $l$ accumulates onto a saddle connection of $G$. Thus in particular $l$ accumulates onto the leaf $l_0$. But the leaf $l_0$ is dense in $G$ and therefore $l$ itself is dense.

Now, if $l$ is not $l^+$, $l^-$ or $l_0$, then we could define the hitting measure $\mu$ by using either half leaf of $l$. This implies that both half leaves of $l$ are dense. In case $l$ is $l^+$ or $l^-$ then we see that one half leaf of $l$ is dense. The other half leaf corresponds to a train path that simply traverses the branches $b_0,b_1,b_2,\ldots$.

Finally, we already know that $l_0$ is dense. It remains to be shown that each half leaf of $l_0$ is dense. To see this, note that any half leaf of $l_0$ traverses a saddle connection $P_n\to Q_n$ for $n$ arbitrarily large. This shows that this half leaf accumulates onto either $l^+$ or $l^-$ and therefore this half leaf is also dense.
\end{proof}

\section{Lamination properties}
\label{sec:laminations}

Recall that $\Omega$ is the plane minus a Cantor set. We embed the train track $T$ on $\Omega$ as shown on the left of Figure \ref{Tembedded}. 
The blue curves in the middle of Figure \ref{Tembedded} are chosen to lie in the pants decomposition $\mathcal{P}$ from Section \ref{sec: rays}. In particular, there are sequences of these curves, $\ldots C_{-2}, C_{-1}, C_0, C_1 , \ldots$ such that:
\begin{itemize}
\item for each $i$, $C_i$ separates from $C_{i-1}$ and $C_{i+1}$,
\item for each $i\notin \{-1,0\}$, $C_{i-1}$ and $C_{i+1}$ are separated by no other element of $\mathcal{P}$,
\item $C_{-1}$ and $C_0$ bound a four-holed sphere together with $\infty$ and one other element of $\mathcal{P}$.
\end{itemize}

 Collapsing parallel branches of $T$ yields a locally finite train track $T^*$ as shown in the middle
of Figure \ref{Tembedded}. There is also a \textit{carrying map} $\zeta$ which assigns to each branch of $T$ a finite train path of $T^*$.

\begin{figure}[h]

\begin{tabular}{c c c}

\def\svgwidth{0.21\textwidth}
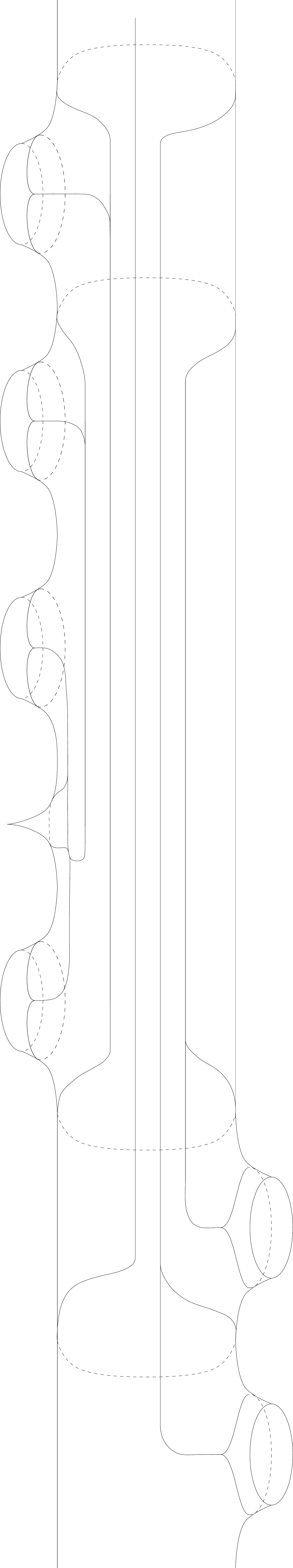

&

\def\svgwidth{0.21\textwidth}
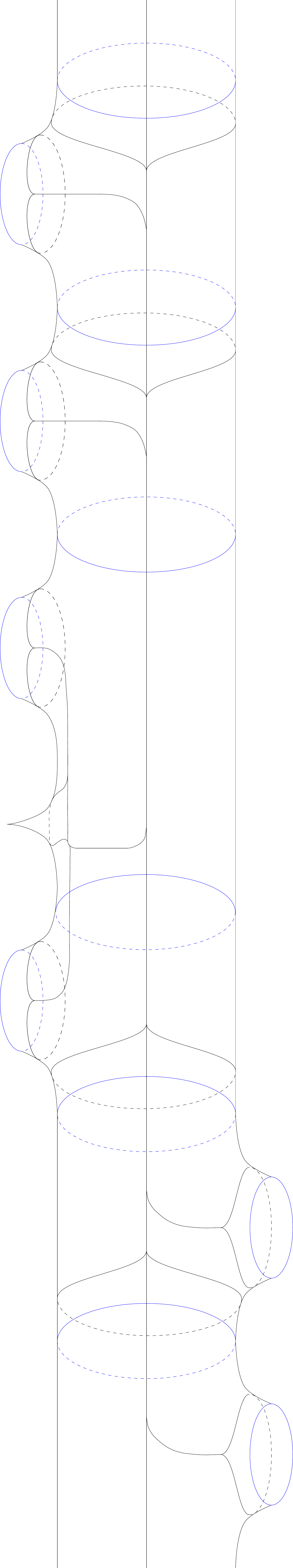

&

\def\svgwidth{0.183\textwidth}
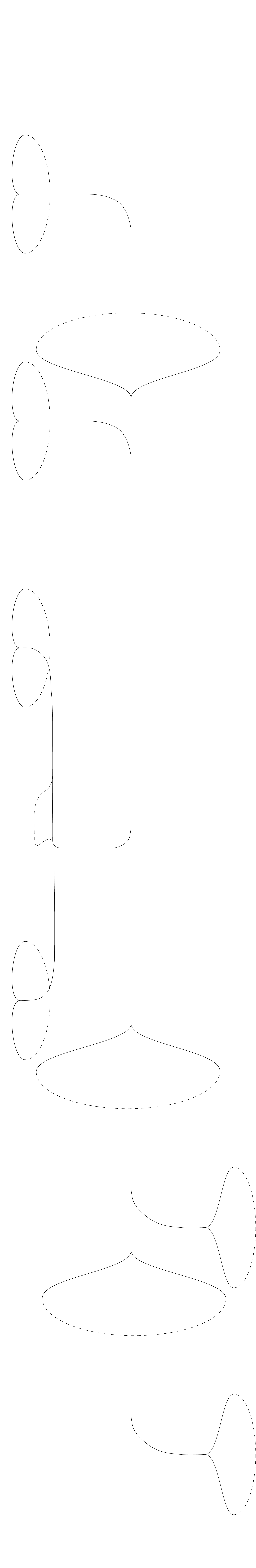

\end{tabular}

\caption{Left: the train track $T$ embedded in the surface $\Omega$. Here each boundary component bounds a disk minus a Cantor set in $\Omega$. Middle: collapsing parallel branches yields a (locally finite) track $T^*$. Right: we name the branches of $T^*$ with the labels shown.
}\label{Tembedded}
\end{figure}

We label the branches of $T^*$ as shown on the right of Figure \ref{Tembedded}. Thus we see that \[\zeta(e_n)=e_n^*, \ \ \ \zeta(d_n)=d_n^*.\] We also have \[\zeta(c_1) = c_1^*, \ \zeta(c_n)=c_n^* \text{ for $n$ even, \ and } \zeta(c_n)=f_{-n+1}^*c_n^* \text{ for $n\geq 3$ odd.}\] Finally, we have \[\zeta(b_{-1})=b_{-1}^*, \ \ \ \zeta(b_0)=b_0^*,  \ \ \ \zeta(b_1)=b_1^*f_0^*, \ \ \ \zeta(b_2)=f_1^* h_1^* f_1^* f_0^* f_{-1}^*,  \ \ \ \zeta(b_3)=h_2^* f_{-1}^* f_0^* f_1^* f_2^*, \ \ \ \ldots.\] In general, \[\zeta(b_n)=f_{n-1}^* h_{n-1}^* f_{n-1}^* f_{n-2}^* \ldots f_0^* f_{-1}^* \ldots f_{-n+1}^*\] for $n\neq 0$ even and \[\zeta(b_n)=h_{n-1}^* f_{-n+2}^* f_{-n+3}^* \ldots f_0^* f_1^* \ldots f_{n-1}^*\] for $n\neq -1, 1$ odd.

The weights $w$ on $\mathcal{B}(T)$ induce a system of weights $w^*$ on $\mathcal{B}(T^*)$ via \[w^*(b^*) = \sum_{b\in \mathcal{B}(T)} (\text{\# of occurrences of $b^*$ in $\zeta(b)$}) \cdot w(b).\] Thus we have \[w^*(e_1^*)=w(e_1)=\frac{1}{3}, w^*(e_2^*)=w(e_2)=\frac{2}{3},\] \[w^*(b_{-1}^*)=w(b_{-1})=1, w^*(b_0^*)=w(b_0)=1, w^*(b_1^*)=w(b_1)=\frac{1}{2}\] and \[w^*(c_n^*)=w(c_n)=\frac{1}{2^n}, w^*(d_n^*)=w(d_n)=\frac{1}{2^{n+1}}.\] Finally we have \[w^*(f_0^*)=1, w^*(f_1^*)=\frac{3}{4}, w^*(f_2^*)=\frac{1}{4}, w^*(f_3^*)=\frac{3}{16}, w^*(f_4^*)=\frac{1}{16}, w^*(f_5^*)=\frac{3}{64}, \ldots\] and \[w^*(f_{-n}^*)=\frac{1}{2^n}, w^*(h_n^*)=\frac{1}{2^{n+1}}\] for each $n\geq 1$. Denote by $G^*$ the union of foliated rectangles defined by  $(T^*,w^*)$.

Now we will describe a continuous injection $\xi:\TP(T,w)\to \TP(T^*,w^*)$. For this purpose, consider the preimages of the switches of $T^*$ under $\zeta$. The preimages  induce a partition of each rectangle $R(b)=[0,1]\times [0,w(b)]$ into $\#\zeta(b)$ vertical subrectangles of equal width $1/\#\zeta(b)$. Each switch preimage is either a switch of $T$ or lies in the interior of a branch of $T$. For convenience, we will consider each preimage as a (possibly new, valence two) switch of $T$. These switches partition $T$ into a set of branches, each of which is mapped homeomorphically by $\zeta$ to a branch of $T^*$. Thus we may consider $\zeta$ as a surjection $\B(T)\to \B(T^*)$. Moreover, each new branch of $T$ corresponds to one of the rectangles in the partition of the old rectangles described before. By abuse of notation, we will denote by $G$ the union of foliated rectangles corresponding to $(T,w)$ considered as a train track with the new switches. It is obtained from the old $G$ by rescaling rectangles horizontally.

For a branch $b^*\in \B(T^*)$, its preimage $\zeta^{-1}(b^*)$ is a possibly infinite subset of $\B(T)$ and \[w^*(b^*)=\sum_{b\in \zeta^{-1}(b^*)} w(b).\] Moreover, $\zeta^{-1}(b^*)$ inherits a total order $<_{b^*}$ where we order parallel branches on the left of Figure \ref{Tembedded} from left to right. For a branch $b^* \in \B(T^*)$, the rectangle $R(b^*)=[0,1] \times [0,w^*(b^*)]$ is divided into horizontal subrectangles as follows. For a branch $b\in \zeta^{-1}(b^*)$, we consider the horizontal subrectangle \[R^*(b)=[0,1] \times \left[\sum_{\substack{b' \in \zeta^{-1}(b^*) \\ b' <_{b^*} b}} w(b'), \left(\sum_{\substack{b' \in \zeta^{-1}(b^*) \\ b' <_{b^*} b}} w(b')\right) + w(b) \right]\] of height $w(b)$. The map $\chi:G\to G^*$ is defined by sending $R(b)=[0,1]\times [0,w(b)]$ isometrically to $R^*(b)$ in the natural way.


The injection $\xi:\TP(T,w)\to \TP(T^*,w^*)$ is defined by replacing each branch $b$ in the train path $t$ with the branch $\zeta(b)$. That the image $\xi(t)$ actually lies in $\TP(T^*,w^*)$ follows from the fact that $\chi$ is a map sending leaves to leaves. 

\begin{lemma}
\label{pathimage}
The image of $\xi$ consists of all of $\TP(T^*,w^*)$ except for a single path which has the form \[t_0=\ldots f_{-2}^* f_{-1}^* f_0^* f_1^* f_2^* \ldots.\]
\end{lemma}

\begin{proof}
Except for the rectangles $R(f_i^*)$, every rectangle of $G^*$ is the homeomorphic image under $\chi$ of a single subrectangle from the foliation $G$.

On the other hand, each rectangle $R(f_i^*)$ contains the homeomorphic images of infinitely many subrectangles from $G^*$. These images are concatenated into two stacks. Identifying $R(f_i^*)$ with $[0,1]\times [0,w^*(f_i^*)]$, one stack consists of rectangles whose heights decrease with increasing $y$-coordinate. The other stack consists of rectangles whose heights increase with increasing $y$-coordinate. 
See Figure \ref{rectanglestack}.

\begin{figure}[h]
\def\svgwidth{0.7\textwidth}
\begingroup%
  \makeatletter%
  \providecommand\color[2][]{%
    \errmessage{(Inkscape) Color is used for the text in Inkscape, but the package 'color.sty' is not loaded}%
    \renewcommand\color[2][]{}%
  }%
  \providecommand\transparent[1]{%
    \errmessage{(Inkscape) Transparency is used (non-zero) for the text in Inkscape, but the package 'transparent.sty' is not loaded}%
    \renewcommand\transparent[1]{}%
  }%
  \providecommand\rotatebox[2]{#2}%
  \newcommand*\fsize{\dimexpr\f@size pt\relax}%
  \newcommand*\lineheight[1]{\fontsize{\fsize}{#1\fsize}\selectfont}%
  \ifx\svgwidth\undefined%
    \setlength{\unitlength}{2970.75122263bp}%
    \ifx\svgscale\undefined%
      \relax%
    \else%
      \setlength{\unitlength}{\unitlength * \real{\svgscale}}%
    \fi%
  \else%
    \setlength{\unitlength}{\svgwidth}%
  \fi%
  \global\let\svgwidth\undefined%
  \global\let\svgscale\undefined%
  \makeatother%
  \begin{picture}(1,0.73368175)%
    \lineheight{1}%
    \setlength\tabcolsep{0pt}%
    \put(0,0){\includegraphics[width=\unitlength,page=1]{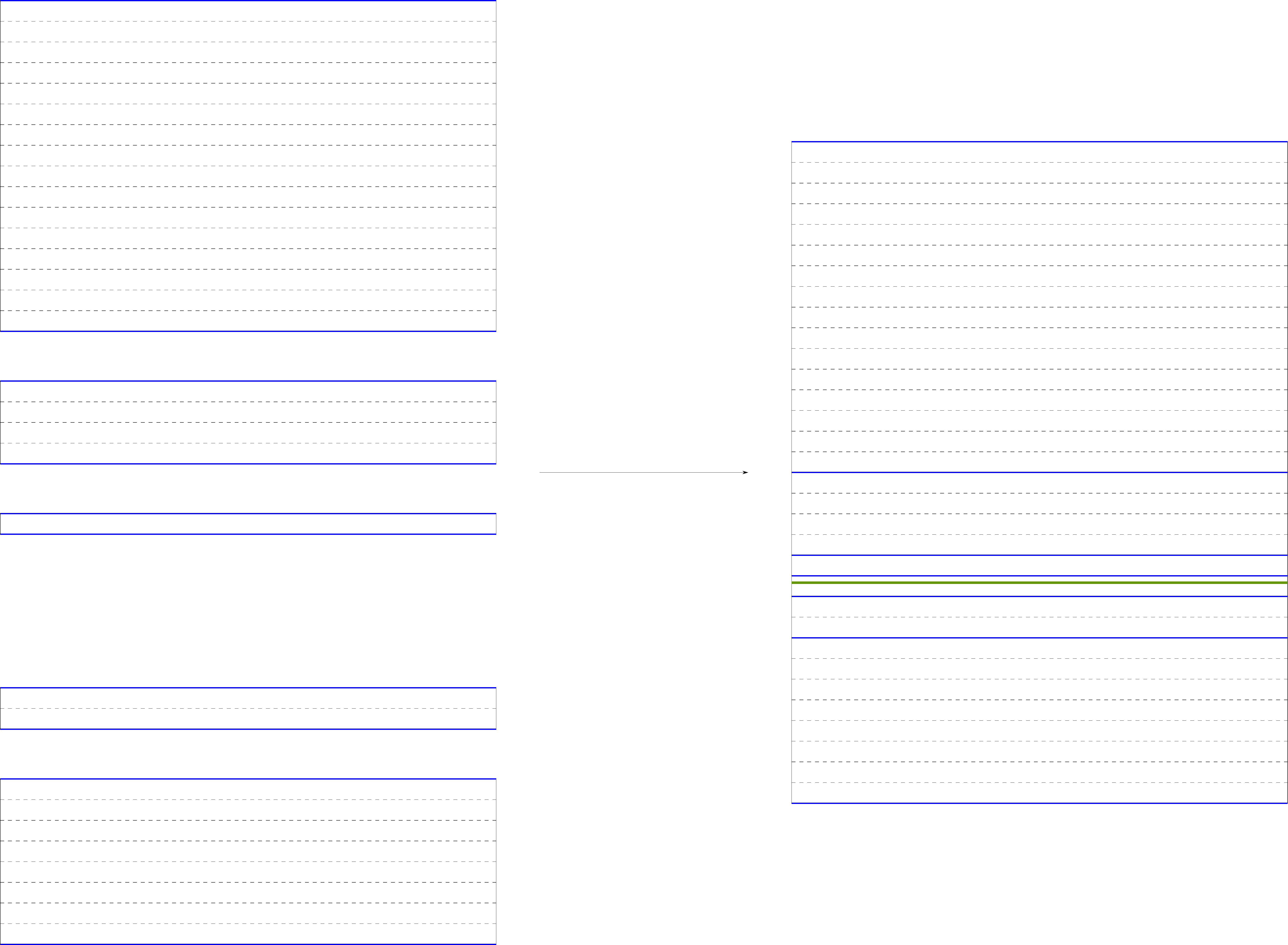}}%
    \put(0.1897229,0.2814726){\color[rgb]{0,0,0}\rotatebox{-90}{\makebox(0,0)[lt]{\lineheight{1.25}\smash{\begin{tabular}[t]{l}$\cdots$\end{tabular}}}}}%
  \end{picture}%
\endgroup%

\caption{Each subrectangle $R(f_i^*)$ consists of infinitely many subrectangles of $G$ stacked together, limiting to a single non-singular leaf of the foliation $G^*$ (shown here in green).}
\label{rectanglestack}
\end{figure}

Since $w^*(f_i^*)$ is the sum of the heights of the rectangles in these two stacks, we see that the images of these rectangles consist of all of $R(f_i^*)$ except for a single horizontal leaf segment which we call $l_i$. Consider the intersection of $l_i$ with the vertical boundary component of $R(f_i^*)$ which meets $R(f_{i+1}^*)$. We claim that $l_i$ meets this boundary component in an endpoint of $l_{i+1}$. For otherwise, $l_i$ meets this vertical boundary component in a point which lies in the image of a subrectangle of $G$. Therefore $l_i$ itself is in the image of a subrectangle of $G$. This is a contradiction. Thus, \[\ldots l_{-1} l_0 l_1 \ldots\] is a (nonsingular) leaf of $G^*$ and \[\ldots f_{-1}^* f_0^* f_1^* \ldots\] is an element of $\TP(T^*,w^*)$. It is not in the image of $\xi$ because it does not pass through some of the branches of $T^*$, whereas every train path in the image of $\xi$ passes through every branch of $T^*$.

On the other hand, every other train path in $\TP(T^*,w^*)$ is represented by a leaf of $G^*$ possibly passing through singularities, which intersects the image of $\chi$. It is therefore the image of a leaf of $G$. Therefore the train path is the image under $\xi$ of a train path in $\TP(T,w)$.

\end{proof}

\begin{cor}
\label{T*props}
The space $\TP(T^*,w^*)$ has the following properties:
\begin{enumerate}
\item every path of $\TP(T^*,w^*)$ is dense except for the train path $t_0$ from Lemma \ref{pathimage},
\item $\TP(T^*,w^*)$ contains exactly three boundary paths.
\end{enumerate}
\end{cor}

\begin{proof}
For (1), we note that since $\xi:\TP(T,w)\to \TP(T^*,w^*)$ is continuous and every train path in $\TP(T,w)$ is dense, every element of $\TP(T^*,w^*) \setminus \{t_0\}$ accumulates onto every other element. We must show that every element of $\TP(T^*,w^*)$ also accumulates onto $t_0$. Since $\xi:\TP(T,w)\to \TP(T^*,w^*) \setminus \{t_0\}$ is surjective, we may write such a train path as $\xi(t)$ where $t\in \TP(T,w)$. Since $t$ is dense, it contains the branch $b_n$ for each $n$. Thus, $f_{-m}^* f_{-m+1}^* \ldots f_{m-1}^* f_m^*$ is a subpath of $\xi(t)$ for arbitrarily large $m$. Thus $\xi(t)$ accumulates onto $t_0$.

Of course, $t_0$ itself is not dense.

Since $\zeta:T\to T^*$ is a local homeomorphism, a non-boundary path in $\TP(T,w)$ is sent by $\xi$ to a non-boundary path in $\TP(T^*,w^*)$. On the other hand, it is easy to see that $\xi$ sends the paths on the boundary of the complementary bigon to $G$ to paths on the boundary of the complementary bigon to $G^*$. Finally, $\TP(T^*,w^*)$ clearly contains at least one more boundary path corresponding to the remaining singular leaves of $G^*$. Since $\xi$ sends non-boundary paths to non-boundary paths, this boundary path in $\TP(T^*,w^*)$ must be the image of the third boundary path of $\TP(T,w)$. This proves (2).
\end{proof}

We define a union of geodesics on $\Omega$ as follows. Recall that each element of $\mathcal{P}$, and in particular every blue curve in Figure \ref{Tembedded}, has length one.  
We lift the embedded train track $T^*$ to the universal cover $\tilde{\Omega}$, which is isometric to the hyperbolic plane when given the pullback metric induced by the cover $\tilde{\Omega}\to \Omega$. In Section \ref{appendix}, we verify the following lemma:

\begin{lemma}
\label{lem:straighten}
Let $\tilde{t}$ be a lift of a train path $t\in \TP(T^*,w^*)$ to $\tilde{\Omega}$. Then $\tilde{t}$ has well-defined (distinct) endpoints on $\partial \tilde{\Omega}$.
\end{lemma}

Letting $\tilde{t}^+$ and $\tilde{t}^-$ be the endpoints of $\tilde{t}$ on $\partial \tilde{\Omega}$, there is a unique geodesic $[\tilde{t}^+,\tilde{t}^-]$ in $\tilde{\Omega}$ from one endpoint to the other. We define $\tilde{\Lambda}$ to be the union of all such geodesics $[\tilde{t}^+,\tilde{t}^-]$, where $t$ ranges over train paths in $\TP(T^*,w^*)$ and $\tilde{t}$ ranges over the lifts of $t$. We define $\Lambda$ to be the image of $\tilde{\Lambda}$ in $\Omega$.

\begin{lemma}
\label{lem:lamclosed}
The sets $\tilde{\Lambda}$ and $\Lambda$ are closed. Hence, they are geodesic laminations in $\tilde{\Omega}$ and $\Omega$, respectively.
\end{lemma}

The proof of Lemma \ref{lem:lamclosed} follows from Lemma \ref{lem:straighten} and the fact that the embedding of $T^*$ on $\Omega$ is locally finite. The proof is essentially the same as the one given in \cite[Section 1.8]{bonahon} ``Geodesic laminations weakly carried by train tracks.'' Therefore we omit it. Furthermore, a leaf $l$ of $\Lambda$ defined by the train path $t$ accumulates onto the leaf $l'$ of $\Lambda$ defined by the train path $t'$ if and only if $t$ accumulates onto $t'$ in the sense of Section \ref{sec:lambg}. We refer the reader again to \cite[Section 1.8]{bonahon}.

Theorem \ref{lamthm} is now a quick corollary. Recall the statement:

\lamthm*

\begin{proof}[Proof of Theorem \ref{lamthm}]
The leaf $m$ of $\Lambda$ corresponds to the train path $t_0$ described in Lemma \ref{pathimage} and is clearly not dense. Since the leaf space of $\Lambda$ is identical to $\TP(T^*,w^*)$, Points (1), (2), and (3) follow from Corollary \ref{T*props}.

The fact that all but four half leaves are dense in $\Lambda$ follows from Lemma \ref{halfleaves}. Clearly neither half leaf of $m$ is dense. On the other hand, consider the two leaves of $\Lambda$ bounding the bigon with ends $e^+$ and $e^-$. The two leaves are both asymptotic to $e^+$ on one end and to $e^-$ on the other end. We define $e^+$ to be the end such that either half leaf on the boundary of the bigon asymptotic to $e^+$ corresponds to a train path $\zeta(b_0)\zeta(b_1)\zeta(b_2) \ldots$. Such a train path contains subtrain paths \[f_n^* f_{n-1}^* \ldots f_0^* \ldots f_{-n+1}^* f_{-n}^*\] for $n$ arbitrarily large, proving that the corresponding leaf accumulates onto $m$. However, the train path doesn't accumulate onto any other train path, since every other train path traverses every single branch of $T^*$ whereas the train path $\zeta(b_0)\zeta(b_1)\zeta(b_2) \ldots$ traverses only branches $f_i^*$ and $h_i^*$. This proves both (4) and (5).

On the other hand, the two half leaves of $\Lambda$ asymptotic to $e^-$ are dense, again by Lemma \ref{halfleaves}. This proves (6).
\end{proof}

\section{Correspondence of the constructions}
\label{sec:correspondence}

In this section we prove that the 2-filling ray $\gamma$ constructed in Section \ref{sec:tracks} in fact has the form of one of the 2-filling rays constructed in Section \ref{sec:straightforward}. In particular, the lamination $\Lambda$ is the $\omega$-limit set of one of the 2-filling rays constructed in Section \ref{sec:straightforward}. For this purpose, we explicitly demonstrate that the ray $\tau$ of Theorem \ref{lamthm} is two-side approachable.

Consider the embedding of $T$ in the plane minus a Cantor set $\Omega$. Define a sequence of loops $r_i$ in $\Omega$ approaching $\tau$ from the right as follows. The loop $r_1$ is homotopic to the (non-simple) path given by traveling within the bigon containing $\infty$ to the branch $b_{-1}$ of $T$, traversing the branch $d_0$ \textit{clockwise}, traversing $b_{-1}$ backwards to the bigon containing $\infty$, and then returning to $\infty$ within the bigon. Similarly, for $i>1$, $r_i$ is homotopic to the path which
\begin{itemize}
\item travels from $\infty$ to the branch $b_0$, staying within the bigon of $T$ containing $\infty$,
\item traverses the branches $b_0,b_1,\ldots,b_{2i-3},c_{2i-2}$,
\item traverses the branch $d_{2i-2}$ in the \textit{clockwise} direction,
\item traverses the branches $c_{2i-2},b_{2i-3},\ldots,b_1,b_0$ backwards to the bigon containing $\infty$,
\item and finally returns to $\infty$ within the bigon.
\end{itemize}

Similarly, we define a sequence of loops $l_i$ approaching $\tau$ from the left as follows. The loop $l_i$ is homotopic to the path which
\begin{itemize}
\item travels from $\infty$ to the branch $b_0$, staying within the bigon of $T$ containing $\infty$,
\item traverses the branches $b_0,b_1,\ldots, b_{2i-2},c_{2i-1}$,
\item traverses the branch $d_{2i-1}$ in the \textit{counterclockwise} direction,
\item traverses the branches $c_{2i-1},b_{2i-2},\ldots,b_1,b_0$ backwards to the bigon containing $\infty$,
\item and finally returns to $\infty$ within the bigon.
\end{itemize}

It is easy to see that $\tau$ is two-side approachable using the loops $r_i$ and $l_i$. We claim that $\gamma$ is homotopic to the concatenation \[r_1 \cdot l_1 \cdot \overline{r_1} \cdot r_2 \cdot r_1 \cdot \overline{l_1} \cdot \overline{r_1} \cdot l_2 \cdot r_1 \cdot l_1 \cdot \overline{r_1} \cdot \overline{r_2} \cdot r_1 \cdot \overline{l_1} \cdot \overline{r_1} \cdot r_3\cdots\] which is a fixed word of the substitution \[f:r_1\mapsto r_1 \cdot l_1 \cdot \overline{r_1} \cdot r_2 \cdot r_1 \cdot \overline{l_1}\cdot \overline{r_1}, \ \ \ f:r_i\mapsto r_{i+1} \text{ for } i\geq 2, \ \ \ f:l_i\mapsto l_{i+1} \text{ for } i\geq 1\] on the infinite alphabet $\{r_1,l_1,r_2,l_2,\ldots\}$. The claim follows by studying the singular leaf of the horizontal foliation of $\Sigma$ through the point $(0,\frac{1}{3})$. This leaf is fixed by the pseudo-Anosov $\phi$ and hence its trajectory can be determined by iterating $\phi$ on the horizontal line segment $[0,1]\times \{\frac{1}{3}\}$. One sees that a short horizontal line segment traveling around $p_{2k+1}$ is sent to one traveling around $p_{2k+3}$ and a horizontal line segment traveling around $p_{2k}$ for $k\geq 1$ is sent to one traveling around $p_{2k+2}$. On the other hand, a horizontal line segment traveling across the square to the right, then around $p_0$, then across the square to the left is sent to one of the following form. The image of the segment first travels around $p_0$, then $p_1$, then $p_0$ again, then $p_2$, then $p_0$, then $p_1$, and then $p_0$ again, traversing the square a total of eight times. The claim follows from these facts.

Now, the 2-filling ray constructed using the sequences of loops $\{r_i\}$ and $\{l_i\}$ and the numbers $p_k=q_k=k$ is the limit of the sequence of loops $\{\alpha_k\}$ where $\alpha_1=r_1$, $\alpha_{2k}=\alpha_{2k-1}\cdot l_k \cdot \overline{\alpha_{2k-1}}$, and $\alpha_{2k-1}=\alpha_{2k-2} \cdot r_k \cdot \overline{\alpha_{2k-2}}$. Note that \[\alpha_3=r_1 \cdot l_1 \cdot \overline{r_1} \cdot r_2 \cdot r_1 \cdot \overline{l_1} \cdot \overline{r_1}=f(\alpha_1).\] An easy induction establishes immediately that $f(\alpha_k)=\alpha_{k+2}$ for all $k$. Thus, the sequence $\{\alpha_1,\alpha_3,\alpha_5,\ldots\}=\{\alpha_1,f(\alpha_1),f^2(\alpha_1),\ldots\}$ approaches the limiting concatenation \[r_1 \cdot l_1 \cdot \overline{r_1} \cdot r_2 \cdot r_1 \cdot \overline{l_1} \cdot \overline{r_1} \cdot l_2 \cdot r_1 \cdot l_1 \cdot \overline{r_1} \cdot \overline{r_2} \cdot r_1 \cdot \overline{l_1} \cdot \overline{r_1} \cdot r_3\cdots\] as claimed.

\section{Cliques with multiple non-filling rays}\label{sec:mnf}

In this section we prove the following theorem:

\begin{thm}
Let $\Psi_n$ be the surface with a single planar end and exactly $2n$ non-planar ends. Then $\mathcal{R}(\Psi_n)$ contains a clique containing exactly $n$ 2-filling rays and $n$ non-filling rays.
\end{thm}

To prove the theorem we will follow the methods of Sections \ref{sec:tracks}--\ref{sec:laminations}. Thus, we explicitly produce a lamination $\Lambda_n$ on $\Psi_n$ using a train track. The complementary region of $\Lambda_n$ containing the planar end $p$ will be a $2n$-gon. There are rays to each end of this $2n$-gon and we will show that $n$ of these rays are 2-filling and the other $n$ rays are not 2-filling.

As it turns out, the construction relies in an essential way on the fact that $\Psi_n$ has infinite genus. At the present time, we do not know a way to get around this. Since the methods of this section are similar to the methods of Sections \ref{sec:tracks}--\ref{sec:laminations}, we will only give the construction and sketch the proofs that the claimed properties of the construction hold.

First we define an infinite train track $T_1$ and a system of weights $w_1$ on $T_1$. See Figure \ref{fig:mnftrack}. The weights of various of the branches are displayed and the weights of the remaining branches of $T_1$ may be inferred from these using the switch conditions. Note that there are infinitely many branches which are loops based at trivalent switches and these have weights $\frac{1}{8\cdot 4^k}$ for $k\geq 0$. There are also infinitely many branches which are loops based at quadrivalent switches and these have weights $\frac{1}{4^k}$ for $k\geq 1$.

\begin{figure}[h]
\def\svgwidth{\textwidth}
\input{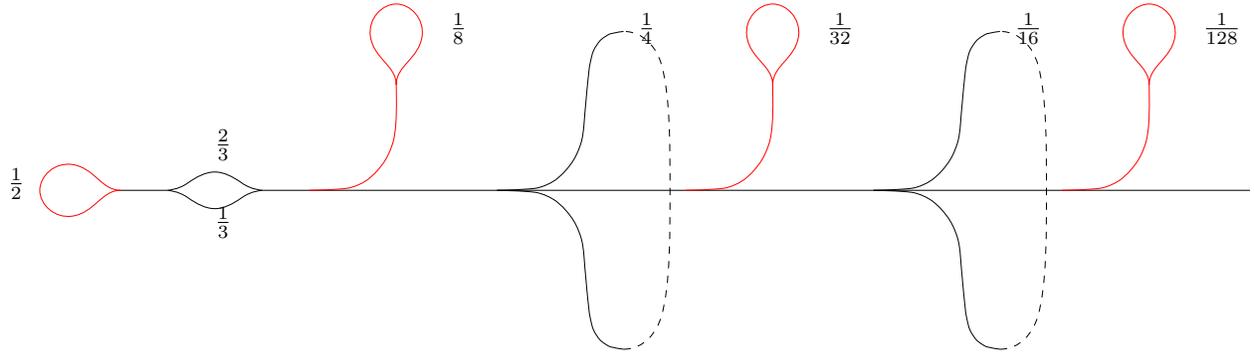}

\caption{The track $T_1$ with weights $w$ given by labels adjacent to branches.}
\label{fig:mnftrack}
\end{figure}

The track $T_1$ has a bigon which we denote by $c$ consisting of the branches of weights $\frac{1}{3}$ and $\frac{2}{3}$. For any $n\geq 2$, there is a cyclic cover $\pi_n:T_n\to T_1$ with the property that $c$ has a unique lift $\tilde{c}_n$ to $T_n$ with respect to which the restriction $\pi_n|\tilde{c}_n:\tilde{c}_n\to c$ is a degree $n$ cover of the circle $c$. See Figure \ref{fig:covertracks} for a picture of the tracks $T_2$ and $T_3$.

\begin{figure}[h]
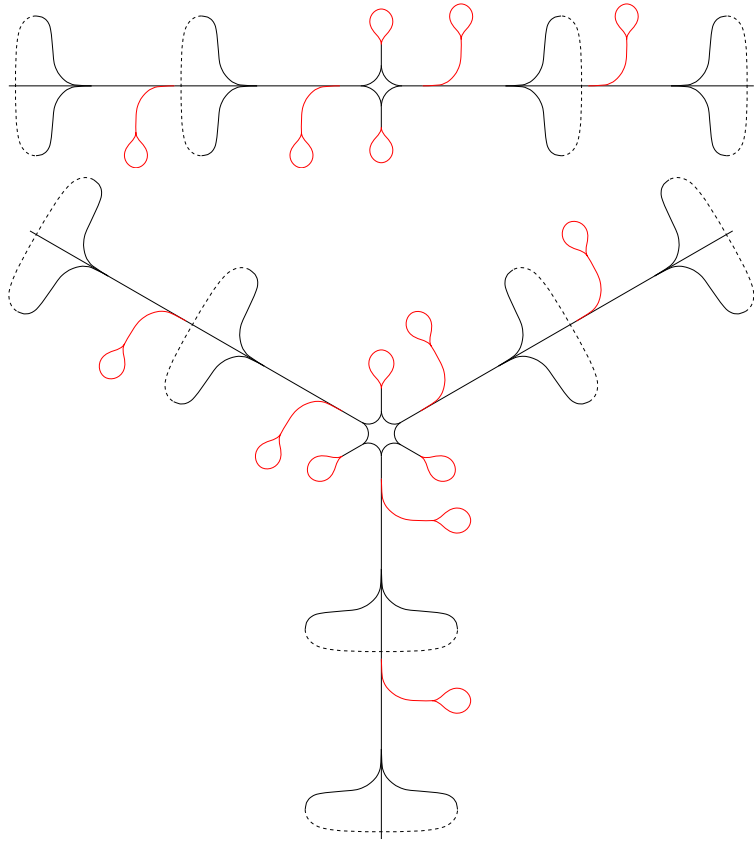

\begin{tabular}{c}
\def\svgwidth{0.6\textwidth}
\input{mnfcovertrackdeg2.pdf_tex}

\\

\def\svgwidth{0.6\textwidth}
\input{mnfcovertrackdeg3.pdf_tex}
\end{tabular}

\caption{The tracks $T_2$ and $T_3$.}
\label{fig:covertracks}
\end{figure}

The tracks $T_n$ carry weights $w_n$ induced by pulling back the weights $w_1$ on $T_1$. Namely, $w_n(b)=w_1(\pi_n(b))$ for any branch $b$ of $T_n$. For each $n$, the weighted train track $(T_n,w_n)$ defines a union of foliated rectangles $G_n$. These foliations are pictured in Figures \ref{fig:mnffol} and \ref{fig:coverfols}. Please note that in these figures a number of rectangles have been contracted to vertical line segments for ease of drawing. This has no effect on the dynamics of the foliations.

\begin{figure}[h]

\def\svgwidth{0.6\textwidth}
\begingroup%
  \makeatletter%
  \providecommand\color[2][]{%
    \errmessage{(Inkscape) Color is used for the text in Inkscape, but the package 'color.sty' is not loaded}%
    \renewcommand\color[2][]{}%
  }%
  \providecommand\transparent[1]{%
    \errmessage{(Inkscape) Transparency is used (non-zero) for the text in Inkscape, but the package 'transparent.sty' is not loaded}%
    \renewcommand\transparent[1]{}%
  }%
  \providecommand\rotatebox[2]{#2}%
  \newcommand*\fsize{\dimexpr\f@size pt\relax}%
  \newcommand*\lineheight[1]{\fontsize{\fsize}{#1\fsize}\selectfont}%
  \ifx\svgwidth\undefined%
    \setlength{\unitlength}{2403.42917615bp}%
    \ifx\svgscale\undefined%
      \relax%
    \else%
      \setlength{\unitlength}{\unitlength * \real{\svgscale}}%
    \fi%
  \else%
    \setlength{\unitlength}{\svgwidth}%
  \fi%
  \global\let\svgwidth\undefined%
  \global\let\svgscale\undefined%
  \makeatother%
  \begin{picture}(1,0.34803004)%
    \lineheight{1}%
    \setlength\tabcolsep{0pt}%
    \put(0,0){\includegraphics[width=\unitlength,page=1]{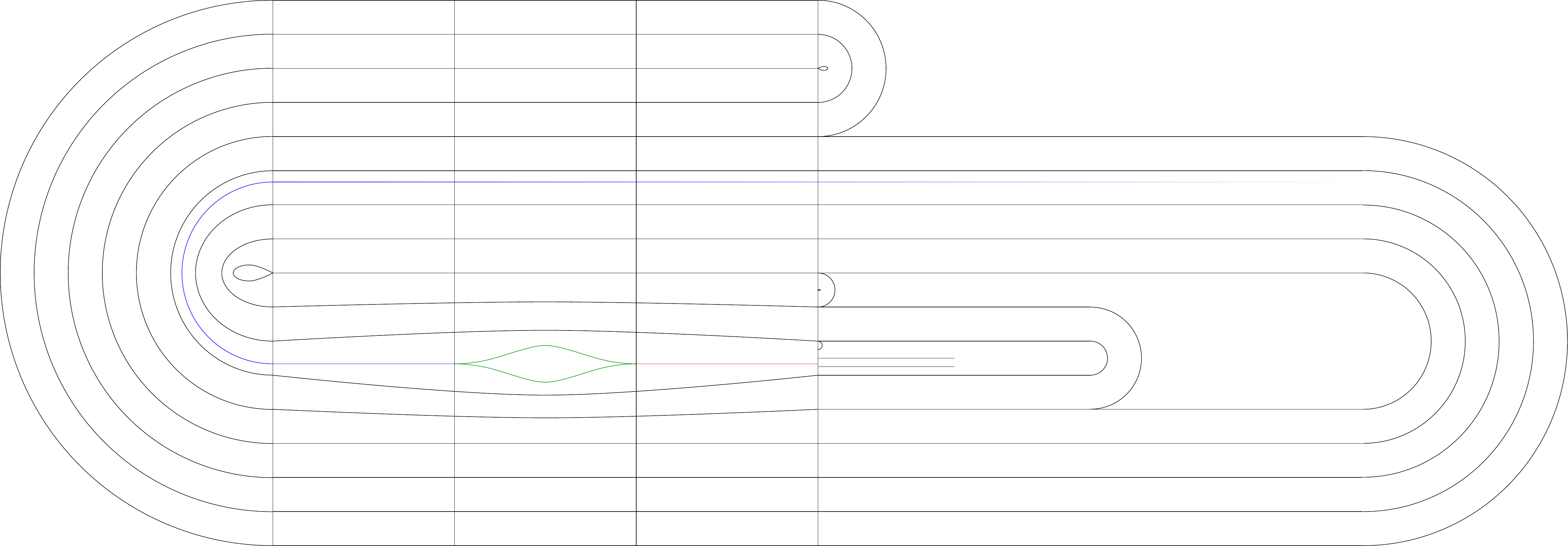}}%
    \put(0.21860933,0.12588352){\color[rgb]{0,0,1}\makebox(0,0)[lt]{\lineheight{1.25}\smash{\begin{tabular}[t]{l}$\alpha$\end{tabular}}}}%
    \put(0.42362568,0.12588352){\color[rgb]{1,0,0}\makebox(0,0)[lt]{\lineheight{1.25}\smash{\begin{tabular}[t]{l}$\beta$\end{tabular}}}}%
    \put(0,0){\includegraphics[width=\unitlength,page=2]{mnffoliationfromtrack.pdf}}%
  \end{picture}%
\endgroup%

\caption{The foliation $G_1$.}
\label{fig:mnffol}
\end{figure}

\begin{figure}[h]

\begin{tabular}{c c}
\def\svgwidth{0.42\textwidth}
\input{mnfcoverfoliationdeg2.pdf_tex}

&

\def\svgwidth{0.58\textwidth}
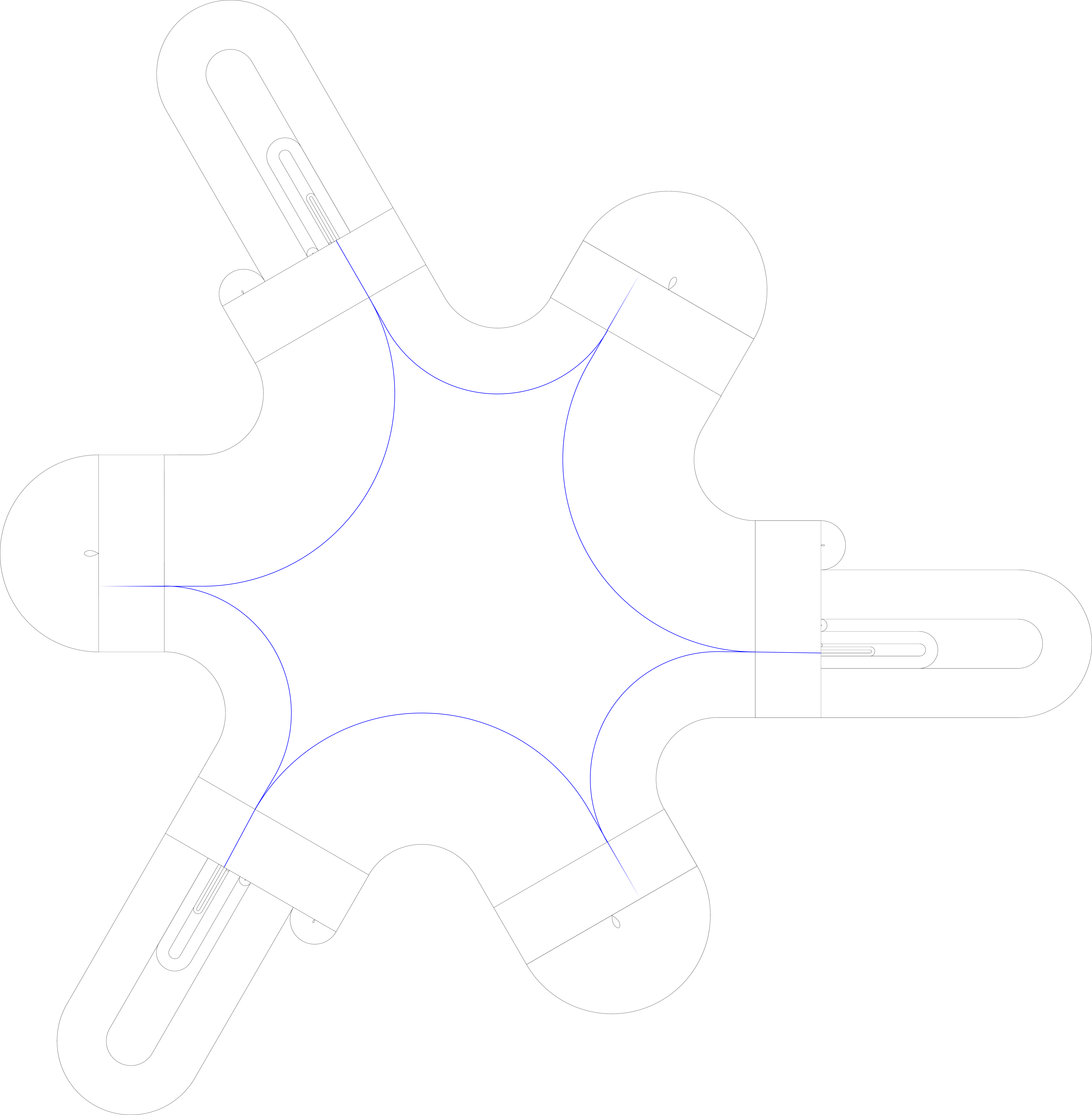
\end{tabular}

\caption{The foliations $G_2$ and $G_3$.}
\label{fig:coverfols}
\end{figure}

We will first study the dynamics of the foliation $G_1$ and then leverage this to study the foliations $G_n$. First we claim that every half leaf of $G_1$ is dense except for a single singular half leaf. For this, we introduce the pseudo-Anosov homeomorphism shown in Figure \ref{fig:mnfpA}. As before, we see that the fixed horizontal foliation of this homeomorphism has the same dynamics as the foliation $G_1$. We use the pseudo-Anosov $\phi$ to study the dynamics of $G_1$.

Analogously to Lemma \ref{singdense}, we use the orbits of singularities of $\phi$ to prove that the union of singular leaves of $G_1$ is dense. By studying transverse measures, we prove, using Lemma \ref{transmeas} that every non-singular half-leaf of $G_1$ is dense. Moreover, as in Lemma \ref{lem:boundarypaths} we see that the singularities of $G_1$ are joined by saddle connections to each other in such a way that $\TP(T_1,w_1)$ contains three boundary paths. One of these boundary paths has both half leaves dense while exactly one half leaf of each of the other two boundary leaves is dense.

\begin{figure}[h]
\def\svgwidth{\textwidth}
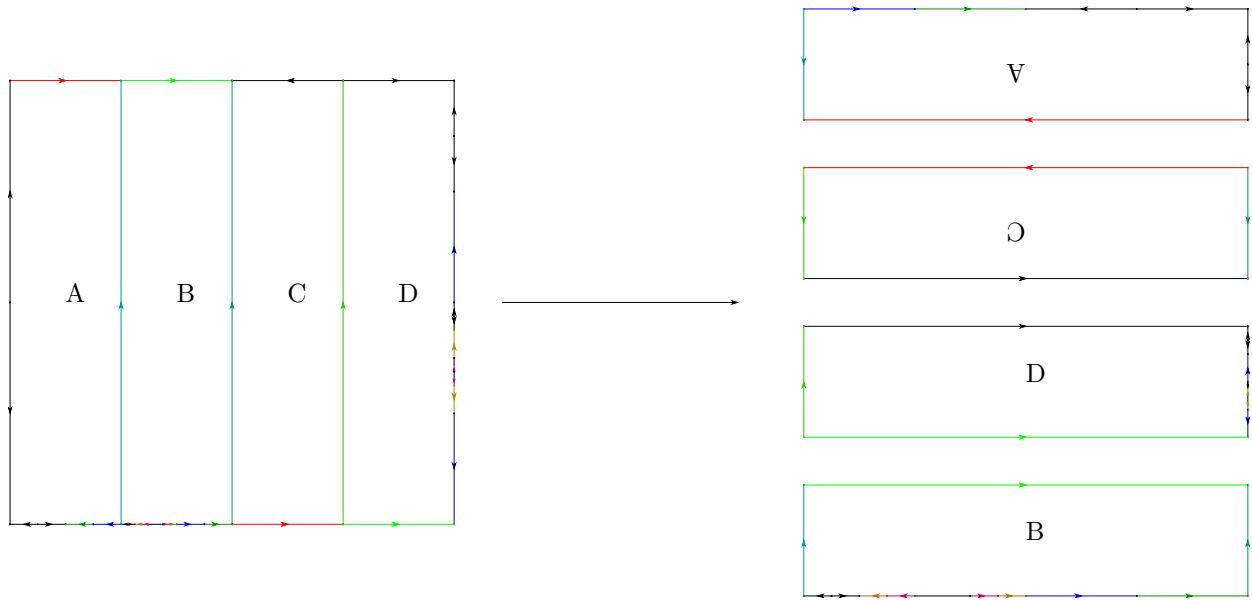

\caption{A pseudo-Anosov homeomorphism. Again, arrows based at a common vertex and pointing in opposite directions are identified by a rotation by $\pi$ about the common vertex.}
\label{fig:mnfpA}
\end{figure}

We now leverage these claims to prove the following result:

\begin{thm}
\label{thm:covertrack}
Every train path $\TP(T_n,w_n)$ is dense. Furthermore, every half train path of $\TP(T_n,w_n)$ is dense except for $2n$ half train paths.
\end{thm}

As before, the argument of \cite[Theorem 12.1]{flp} shows that every half leaf of $G_1$ is dense except for a single ray which has an endpoint on the complementary bigon. The ray $\alpha$, pictured in Figure \ref{fig:mnffol} is also dense in $G_1$. On the other hand, the other ray beginning from an endpoint of the bigon, which we call $\beta$ is not dense.

We now consider the union of foliated rectangles $G_n$. The ray $\alpha$ has exactly $n$ lifts $\alpha_1,\ldots,\alpha_n$ to $G_n$, each with an endpoint on the complementary $2n$-gon to $G_n$. Similarly, $\beta$ has $n$ lifts $\beta_1,\ldots,\beta_n$. Since every half leaf of $G_1$ is dense except for $\beta$, we have in particular that every half leaf besides $\beta$ accumulates onto $\alpha$, and that $\alpha$ itself is dense. Transporting these facts to $G_n$, we see that the \textit{union} $\alpha_1\cup \ldots \cup \alpha_n$ is dense and moreover that every half leaf of $G_n$ not lying in $\{\beta_1,\ldots,\beta_n\}$ accumulates onto \textit{some} $\alpha_i$. If we can show that $\alpha_i$ accumulates onto $\alpha_j$ for each $i,j$, then we will have that each $\alpha_i$ is dense, and thus every half leaf of $G_n$ not lying in $\{\beta_1,\ldots,\beta_n\}$ is also dense.

\begin{figure}[h]
\def\svgwidth{\textwidth}
\begingroup%
  \makeatletter%
  \providecommand\color[2][]{%
    \errmessage{(Inkscape) Color is used for the text in Inkscape, but the package 'color.sty' is not loaded}%
    \renewcommand\color[2][]{}%
  }%
  \providecommand\transparent[1]{%
    \errmessage{(Inkscape) Transparency is used (non-zero) for the text in Inkscape, but the package 'transparent.sty' is not loaded}%
    \renewcommand\transparent[1]{}%
  }%
  \providecommand\rotatebox[2]{#2}%
  \newcommand*\fsize{\dimexpr\f@size pt\relax}%
  \newcommand*\lineheight[1]{\fontsize{\fsize}{#1\fsize}\selectfont}%
  \ifx\svgwidth\undefined%
    \setlength{\unitlength}{4632.15981791bp}%
    \ifx\svgscale\undefined%
      \relax%
    \else%
      \setlength{\unitlength}{\unitlength * \real{\svgscale}}%
    \fi%
  \else%
    \setlength{\unitlength}{\svgwidth}%
  \fi%
  \global\let\svgwidth\undefined%
  \global\let\svgscale\undefined%
  \makeatother%
  \begin{picture}(1,1.02085848)%
    \lineheight{1}%
    \setlength\tabcolsep{0pt}%
    \put(0,0){\includegraphics[width=\unitlength,page=1]{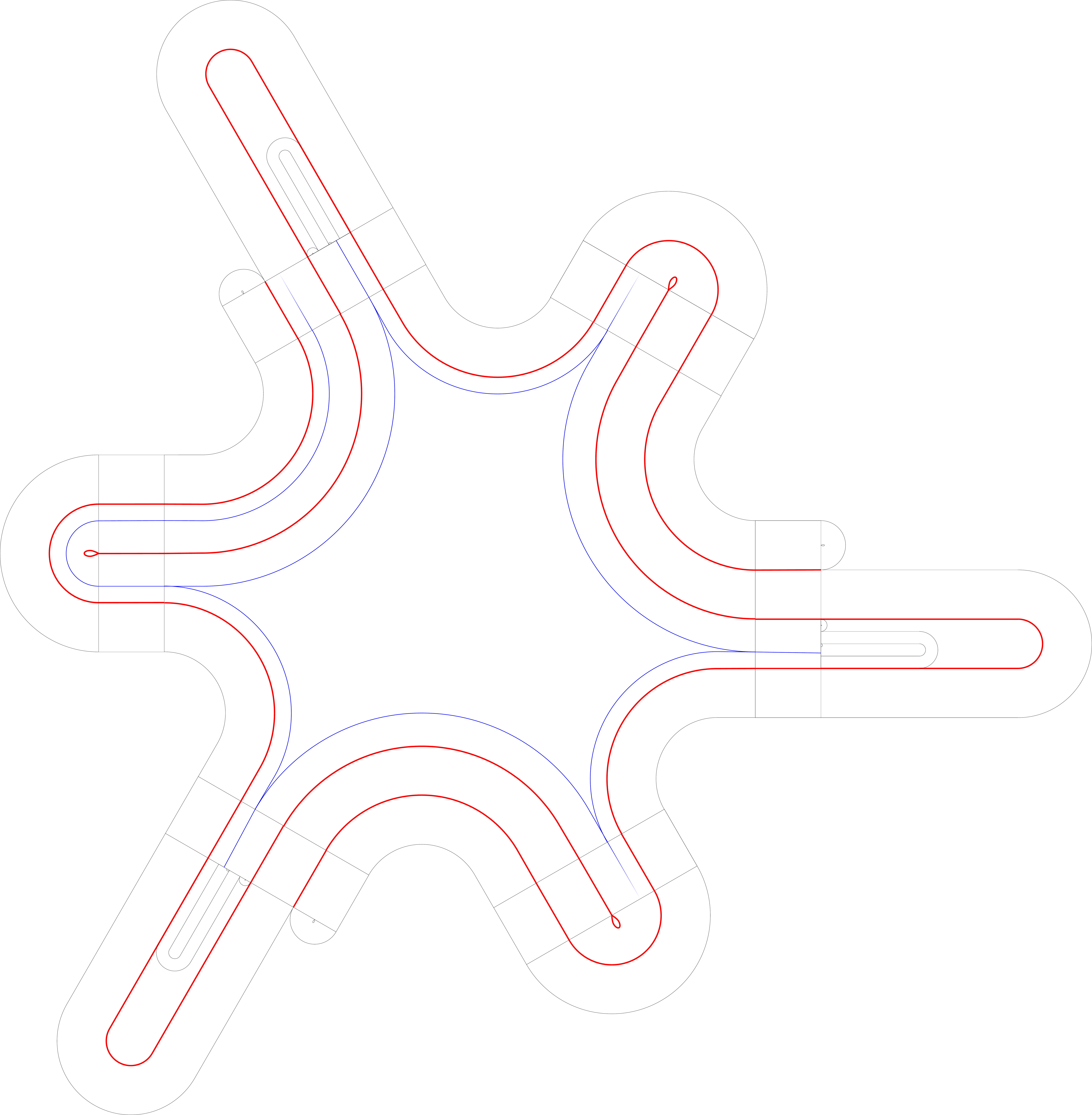}}%
    \put(0.10761939,0.40620603){\color[rgb]{0,0,0}\makebox(0,0)[lt]{\lineheight{1.25}\smash{\begin{tabular}[t]{l}$R_1$\end{tabular}}}}%
    \put(0.10761939,0.4936754){\color[rgb]{0,0,1}\makebox(0,0)[lt]{\lineheight{1.25}\smash{\begin{tabular}[t]{l}$\alpha_1$\end{tabular}}}}%
    \put(0.62777798,0.25781634){\color[rgb]{0,0,0}\makebox(0,0)[lt]{\lineheight{1.25}\smash{\begin{tabular}[t]{l}$R_3$\end{tabular}}}}%
    \put(0.49620792,0.78248177){\color[rgb]{0,0,0}\makebox(0,0)[lt]{\lineheight{1.25}\smash{\begin{tabular}[t]{l}$R_2$\end{tabular}}}}%
    \put(0.53998254,0.22111754){\color[rgb]{0,0,1}\makebox(0,0)[lt]{\lineheight{1.25}\smash{\begin{tabular}[t]{l}$\alpha_3$\end{tabular}}}}%
    \put(0.55984299,0.73175331){\color[rgb]{0,0,1}\makebox(0,0)[lt]{\lineheight{1.25}\smash{\begin{tabular}[t]{l}$\alpha_2$\end{tabular}}}}%
    \put(0,0){\includegraphics[width=\unitlength,page=2]{unzipping.pdf}}%
  \end{picture}%
\endgroup%

\caption{The singular rays $\alpha_i$ and the rectangles $R_i$ that they pass through.}
\label{fig:unzipping}
\end{figure}

For this, we unzip the foliation $G_n$ along the union of the red saddle connections shown in Figure \ref{fig:unzipping} (see \cite[Chapter 1.7]{traintracks} for the notion of unzipping, which is called \textit{splitting} there). We notice that the unzipped foliation $G_n'$ is isomorphic to $G_n$. Each rectangle has been replaced by a rectangle of $\frac{1}{4}$ the same height. Denote by $R_i$ the first rectangle which $\alpha_i$ enters (see Figure \ref{fig:unzipping}). We see by unzipping that $\alpha_i$ enters $R_{i+1}$ for each $i=1,\ldots, n$ (indices being taken modulo $n$).

The foliation $G_n'$ contains $n$ rays $\alpha_1',\ldots,\alpha_n'$ defined analogously to $\alpha_1,\ldots,\alpha_n$, which first pass through the rectangles $R_1',\ldots,R_n'$ of $G_n'$, respectively. By the isomorphism of $G_n'$ with $G_n$, we see that $\alpha_i'$ enters $R_{i+1}'$ for each $i$. However, $R_i'$ is identified with a subrectangle of $R_i$, of $\frac{1}{4}$ the height, intersecting $\alpha_i$ for each $i$. Thus we see that $\alpha_i$ enters not only $R_{i+1}$ but this subrectangle of $\frac{1}{4}$ the height intersecting $\alpha_{i+1}$. Repeating the unzipping process infinitely many times, we see that $\alpha_i$ accumulates onto $\alpha_{i+1}$ for each $i$. Consequently, each $\alpha_i$ accumulates onto $\alpha_j$ for any $j$, as desired.

Finally, since the union of saddle connections of $G_1$ is dense, we see that each boundary path of $\TP(T_n,w_n)$ accumulates onto the train path induced by some $\alpha_i$ and therefore each boundary path is also dense. We have that $\TP(T_n,w_n)$ contains $3n$ boundary paths, $2n$ of which correspond to sides of the complementary $2n$-gon of $G_n$. These $2n$ boundary paths each have one dense and one non-dense half path since each $\alpha_i$ is dense and no $\beta_i$ is dense. On the other hand, each half path of one of the $n$ remaining boundary paths of $\TP(T_n,w_n)$ accumulates onto the train path corresponding to $\alpha_i$ for some $i$, and thus is also dense. This proves Theorem \ref{thm:covertrack}.

Finally, we define the lamination $\Lambda_n$. The train track $T_n$ may be embedded on the surface $\Psi_n$. We illustrate how to do this on the surface $\Psi_2$ with four non-planar ends in Figure \ref{fig:mnftrackonsurface}. 

\begin{figure}[h]
\def\svgwidth{\textwidth}
\input{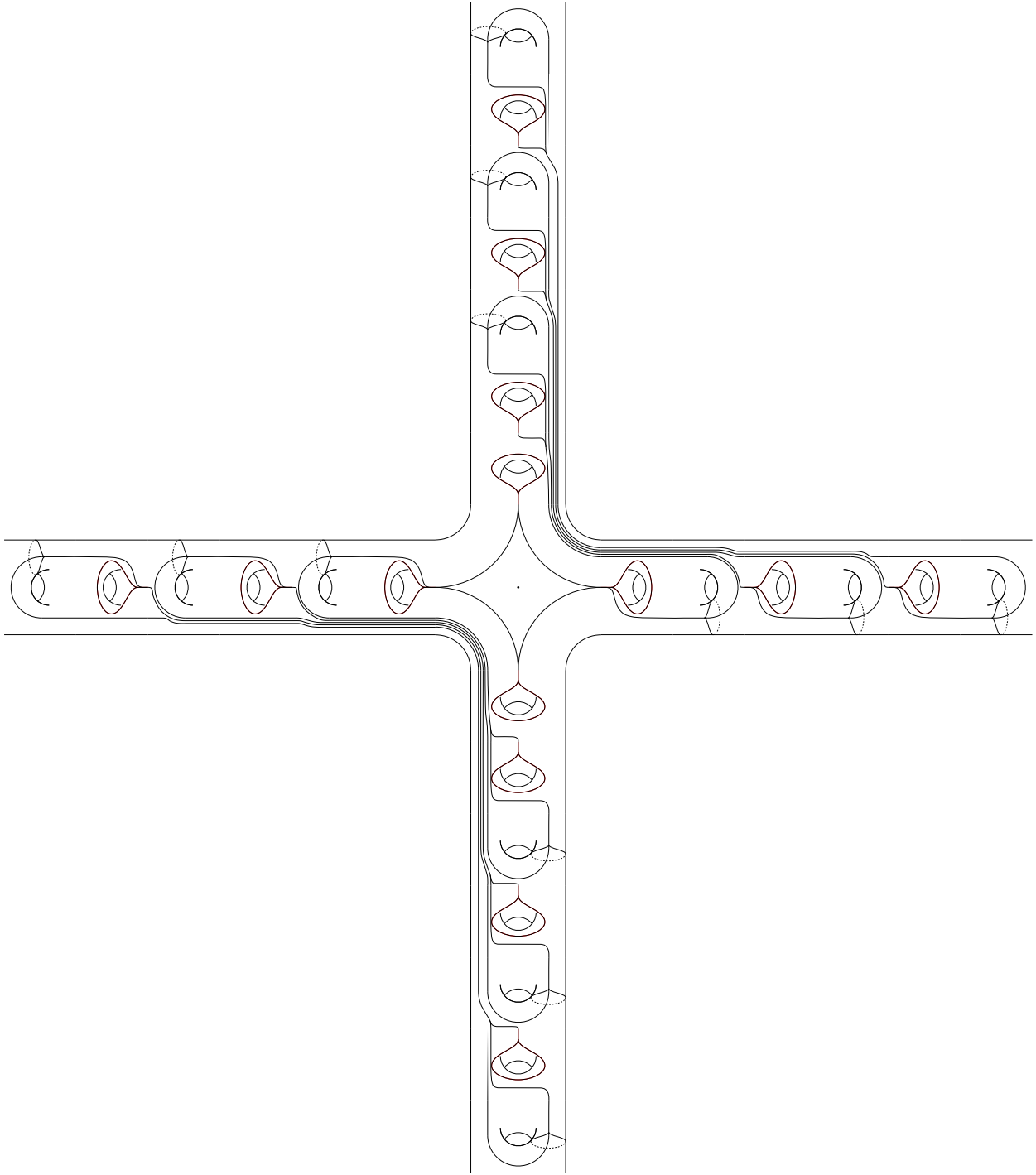}

\caption{The train track $T_2$ embedded in the surface $\Psi_2$.}
\label{fig:mnftrackonsurface}
\end{figure}

Collapsing parallel branches of $T_n$ gives a locally finite train track on $\Psi_n$ with a system of weights induced by $w_n$. We check as in Section \ref{appendix} that the train paths resulting from this locally finite train track and system of weights can be straightened to geodesics on $\Psi_n$. We verify as in Lemma \ref{lem:lamclosed} that the union is closed. This union of geodesics is the lamination $\Lambda_n$. We verify that:

\begin{itemize}
\item The lamination $\Lambda_n$ contains $n$ proper geodesics $m_1,\ldots,m_n$ (these are where the branches of $T_n$ accumulate).
\item The complementary region to $\Lambda_n$ containing the planar end $p$ is a $2n$-gon.
\item Rays to $n$ of the ends of this polygon spiral onto $\Lambda_n$ whereas rays to the other $n$ ends of the polygon each spiral onto one of the proper geodesics $m_i$.
\end{itemize}

Finally then, we form a clique of rays in $\mathcal{R}(\Psi_n,p)$ consisting of the rays to the ends of the $2n$-gon containing $p$. Exactly $n$ of these are 2-filling and the other $n$ are not 2-filling, as desired.

\section{Open questions}\label{sec: open}
Here is a list of open questions about $2$-filling rays.
\begin{enumerate}
		\item Is there a collection of infinitely many disjoint $2$-filling rays on the plane minus a Cantor set? See Question \ref{quest: infinite clique}.
		\item More generally, is there any restriction on the clique of rays disjoint from a $2$-filling ray? Can the clique contain more than one non-filling ray if the ambient surface is the plane minus a Cantor set?
		\item Does the limit set of a $2$-filling ray always contain a proper leaf? See Question \ref{quest: limit set}.
		\item Is a ray intersecting all proper geodesics necessarily high-filling? This is asked by Yan Mary He and Kasra Rafi.
		\item Which kind of geodesic laminations can appear as the limit set of some two-side approachable long ray? What about $2$-filling rays? See Question \ref{quest: lamination}.
		\item Is every $2$-filling ray disjoint from some approachable long ray? See Question \ref{quest: disjoint}.
\end{enumerate}

\section{Appendix: construction of the lamination from the train track with weights}
\label{appendix}

In this section, we verify Lemma \ref{lem:straighten}. To do this, we show that lifts of train paths on $T$ to $\tilde{\Omega}$ are uniform-quality quasi-geodesics. 

We consider again the middle of Figure \ref{Tembedded}. The blue curves pictured divide $\Omega$ into a sphere $V$ with three boundary components and the puncture $\infty$ as well as infinitely many spheres with three boundary components. We denote the spheres with three boundary components by $U_i, i\in \Z$. The numbering is chosen such that if we consider the bi-infinite sequence \[\ldots U_{-2}, U_{-1}, V, U_0, U_1,  U_2,\ldots\] then each surface in the sequence is joined to each of the adjacent surfaces by a boundary component.

Notice that $T$ intersects each three-holed sphere $U_i$ with $i\neq 0$ in the same subtrack. 
Each $U_i$ is endowed with an isometric hyperbolic metric with boundary components of length one. In each $U_i$ with $i\neq 0$ we isotope $T$ so that there is an isometry $U_i\to U_j$ for each $i,j\neq 0$ taking the intersection $T\cap U_i$ to $T\cap U_j$. The components $U_0$ and $V$ are also equipped with hyperbolic metrics with boundary components of length one. We glue the surfaces $U_i$ and $V$ together by isometries along the boundary components such that the tracks $U_i \cap T$ and $V \cap T$ glue together to give $T$. Although a gluing was fixed at the beginning of the paper, any choices of gluing give quasi-isometric surfaces, so whether or not train paths are uniform-quality quasi-geodesics does not depend on the choice of gluing.

Note that there are exactly 17 possible train paths through the train track $T\cap V$ (up to possibly changing the orientation of the path). Eight of these join one of the boundary components of $V$ to itself. They are homotopic, keeping the endpoints on the boundary, to the eight paths drawn in Figure \ref{trainpaths}. The other nine possible train paths join one boundary component of $V$ to another. In particular we note the following by inspection: \textit{no train path in $T\cap V$ is homotopic into $\partial V$}.

\begin{figure}[h]

\begin{tabular}{l l}
\def\svgwidth{0.4\textwidth}
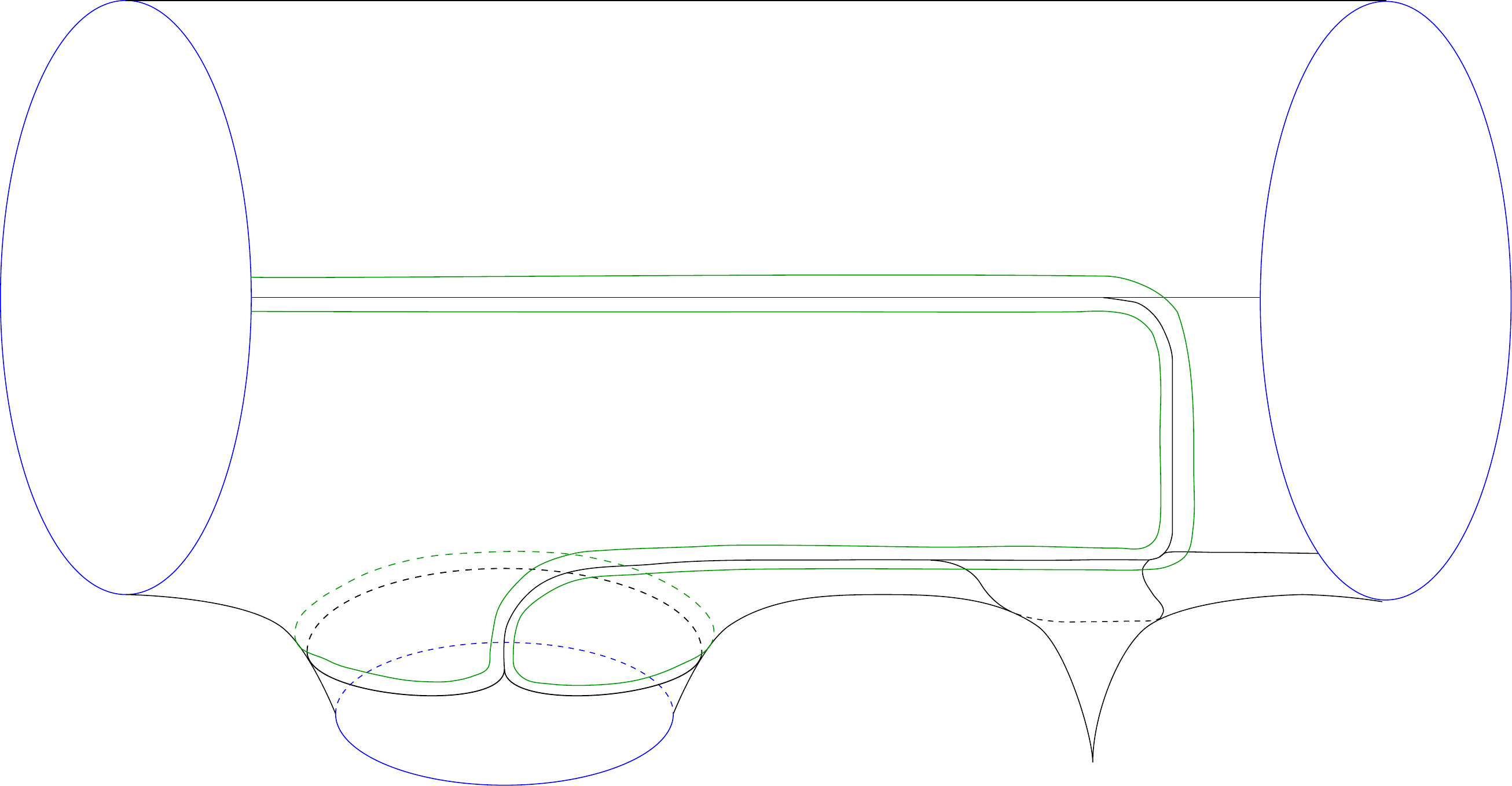 &

\def\svgwidth{0.4\textwidth}
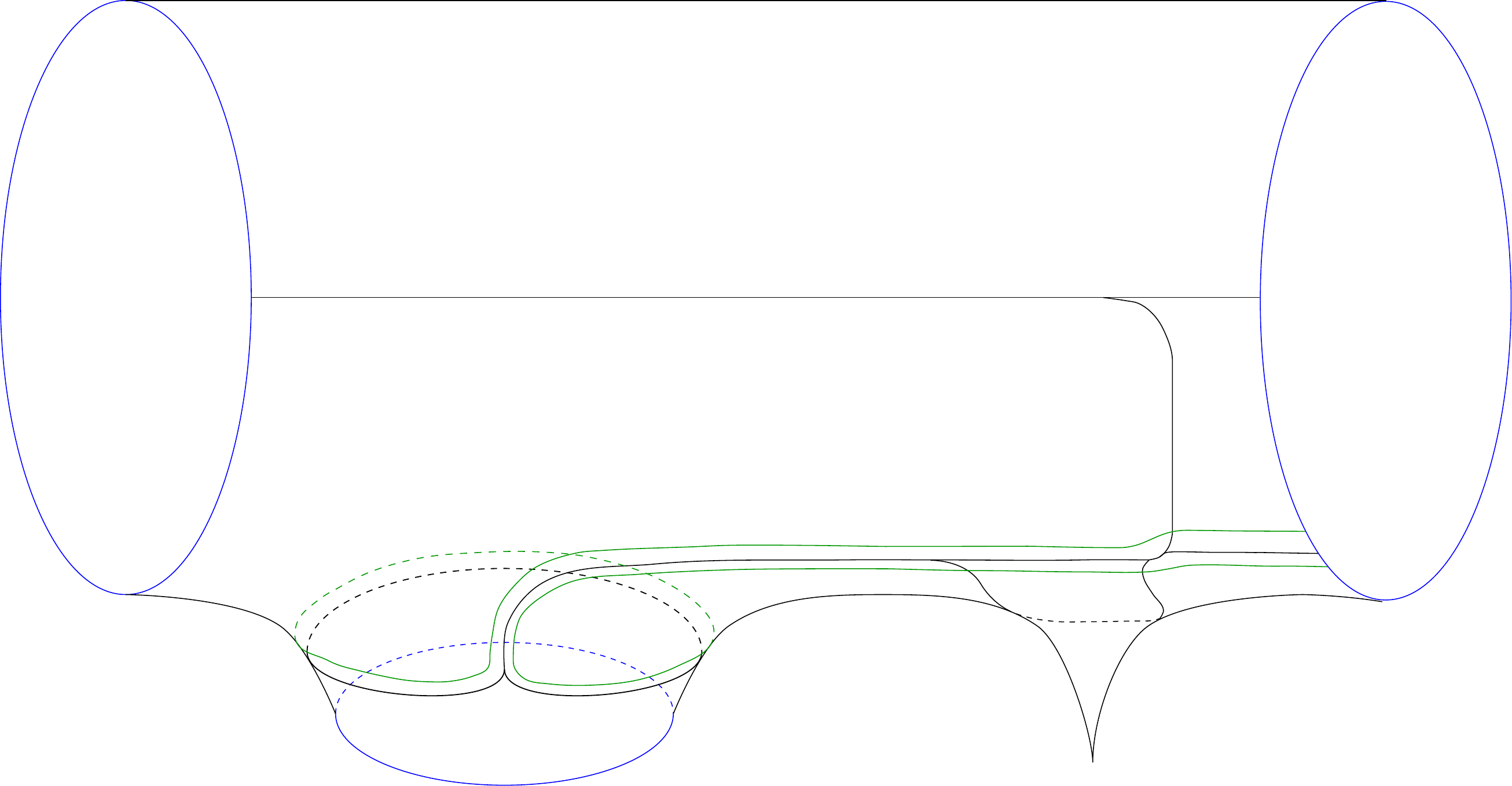 \\

\def\svgwidth{0.4\textwidth}
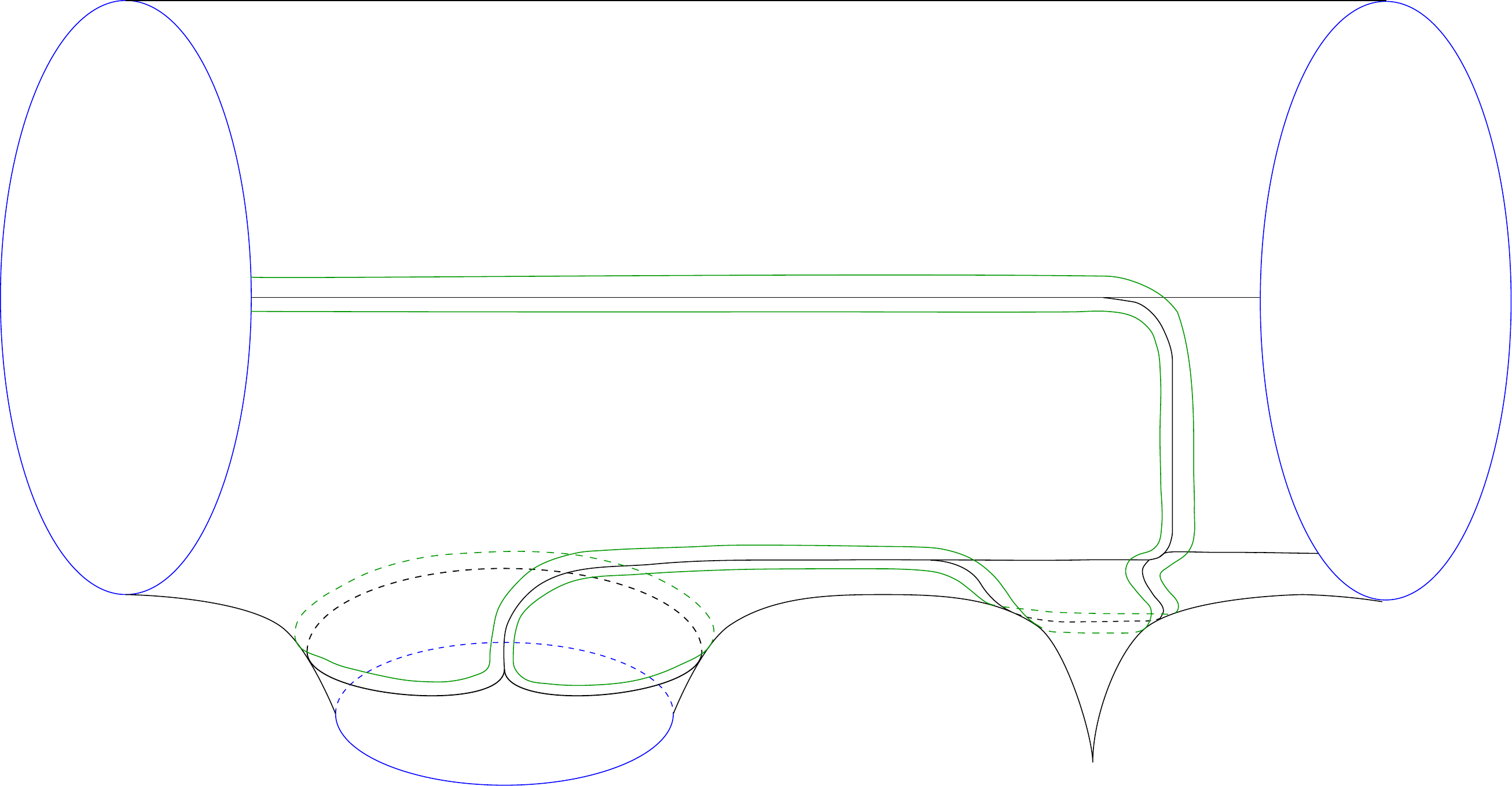 &

\def\svgwidth{0.4\textwidth}
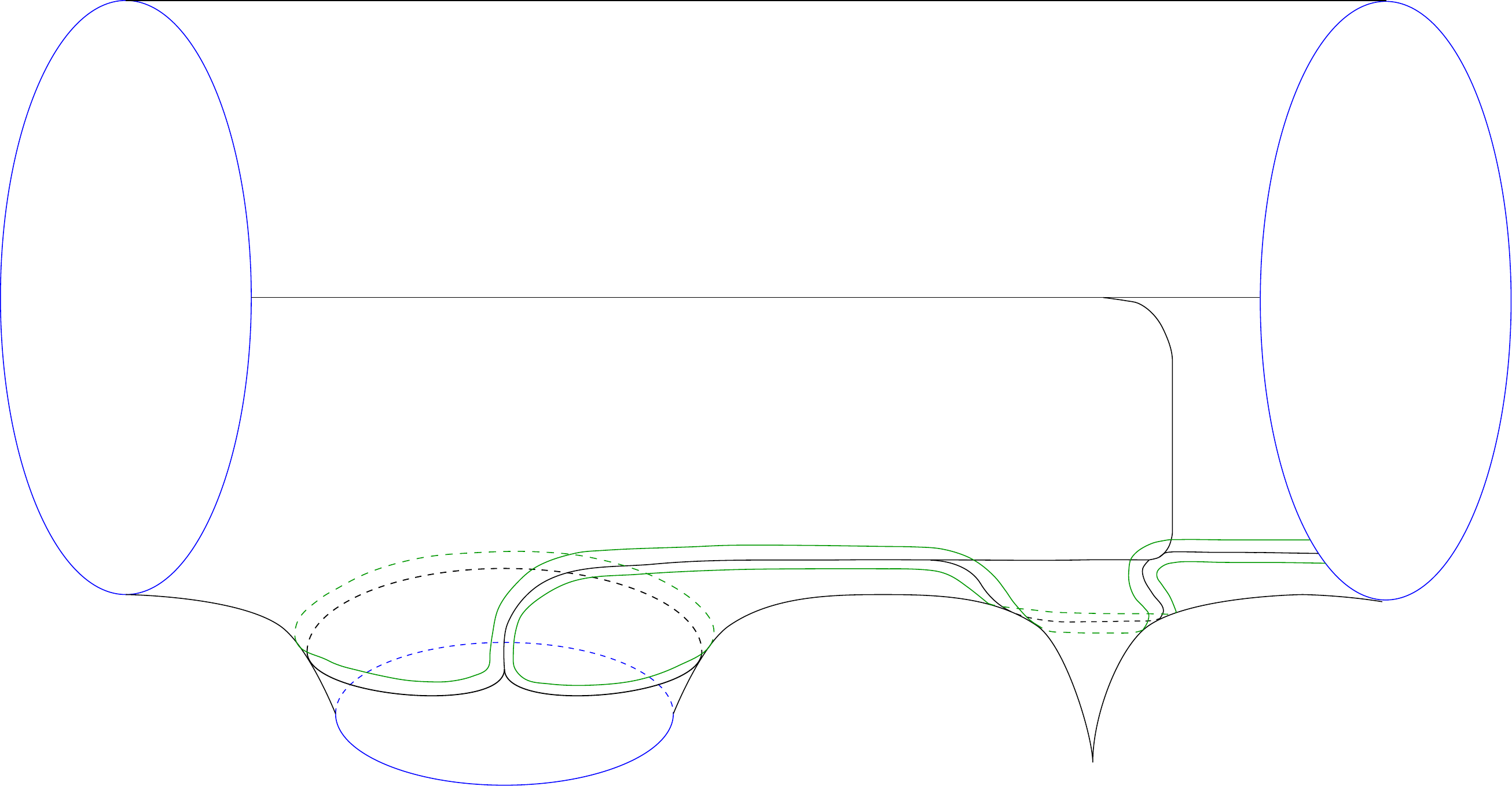 \\

\def\svgwidth{0.4\textwidth}
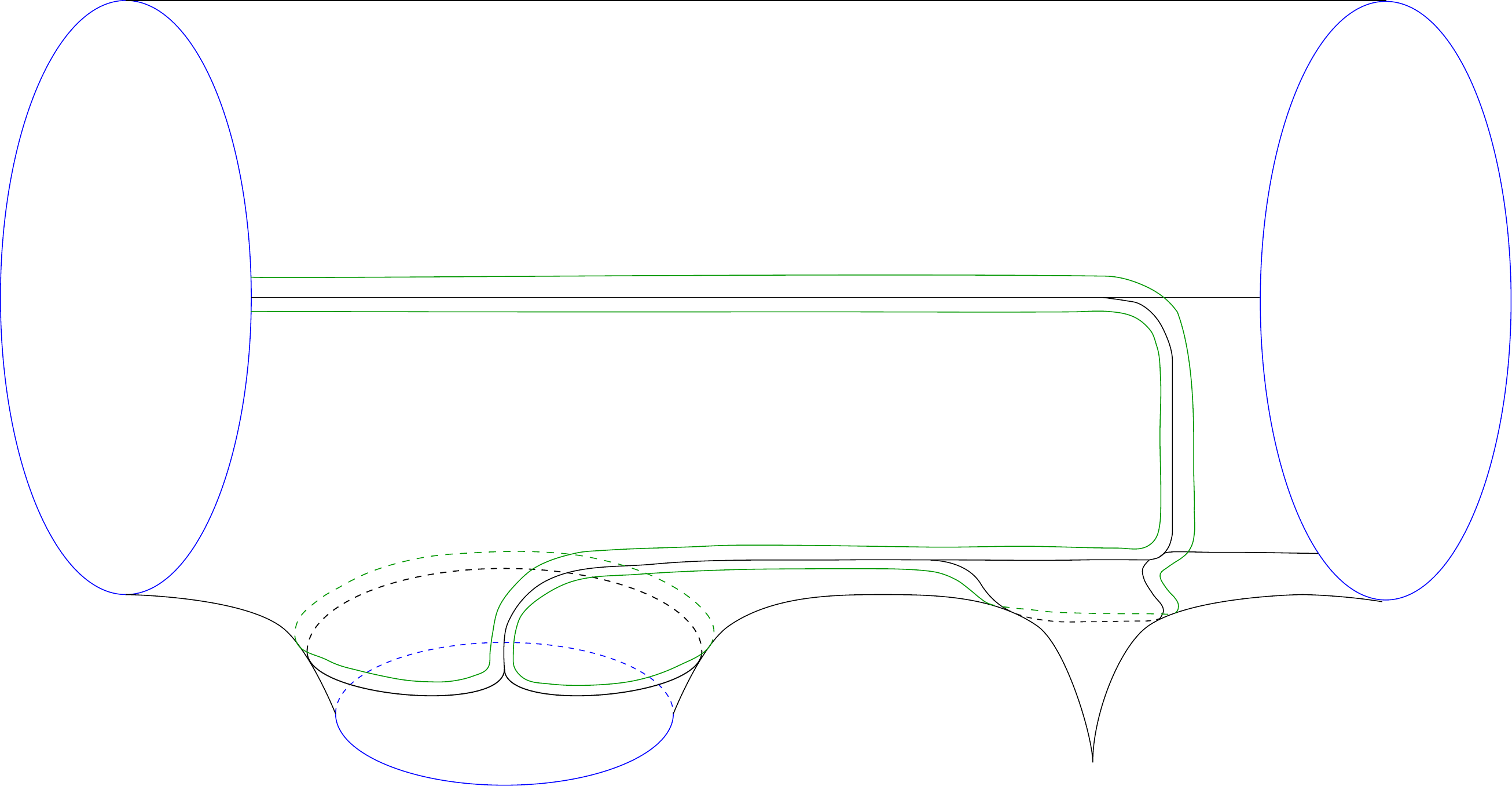 &

\def\svgwidth{0.4\textwidth}
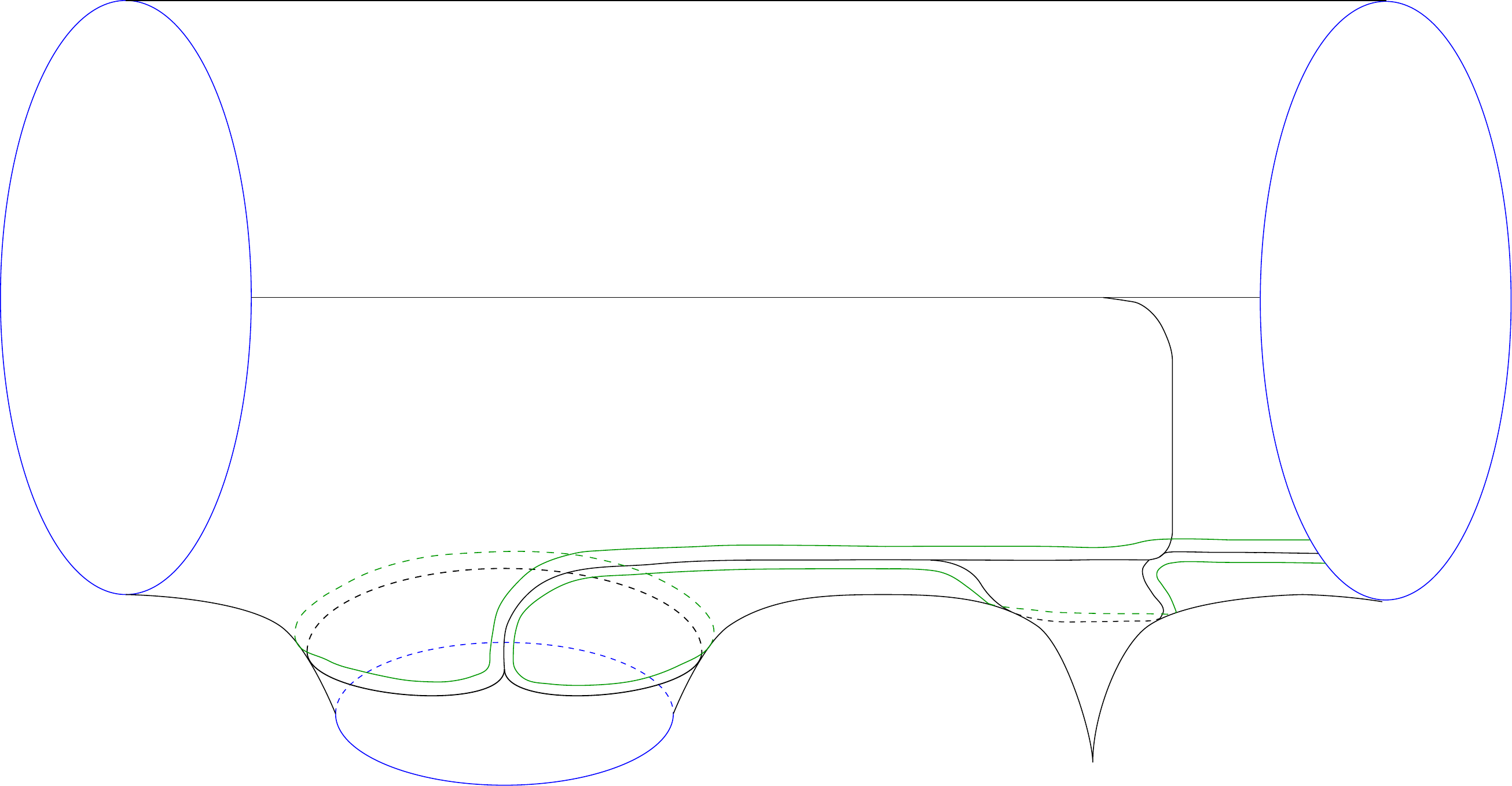 \\

\def\svgwidth{0.4\textwidth}
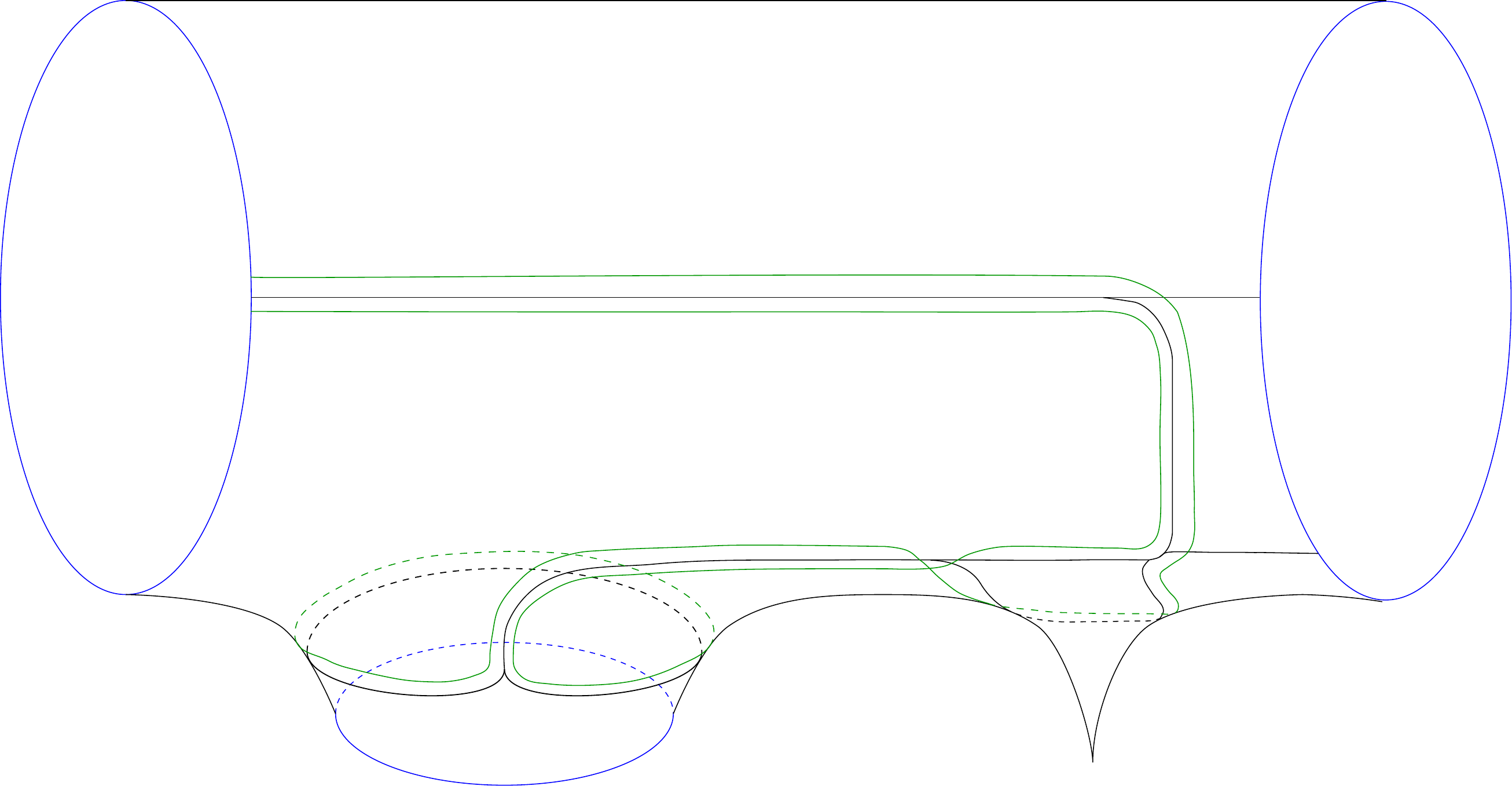 &

\def\svgwidth{0.4\textwidth}
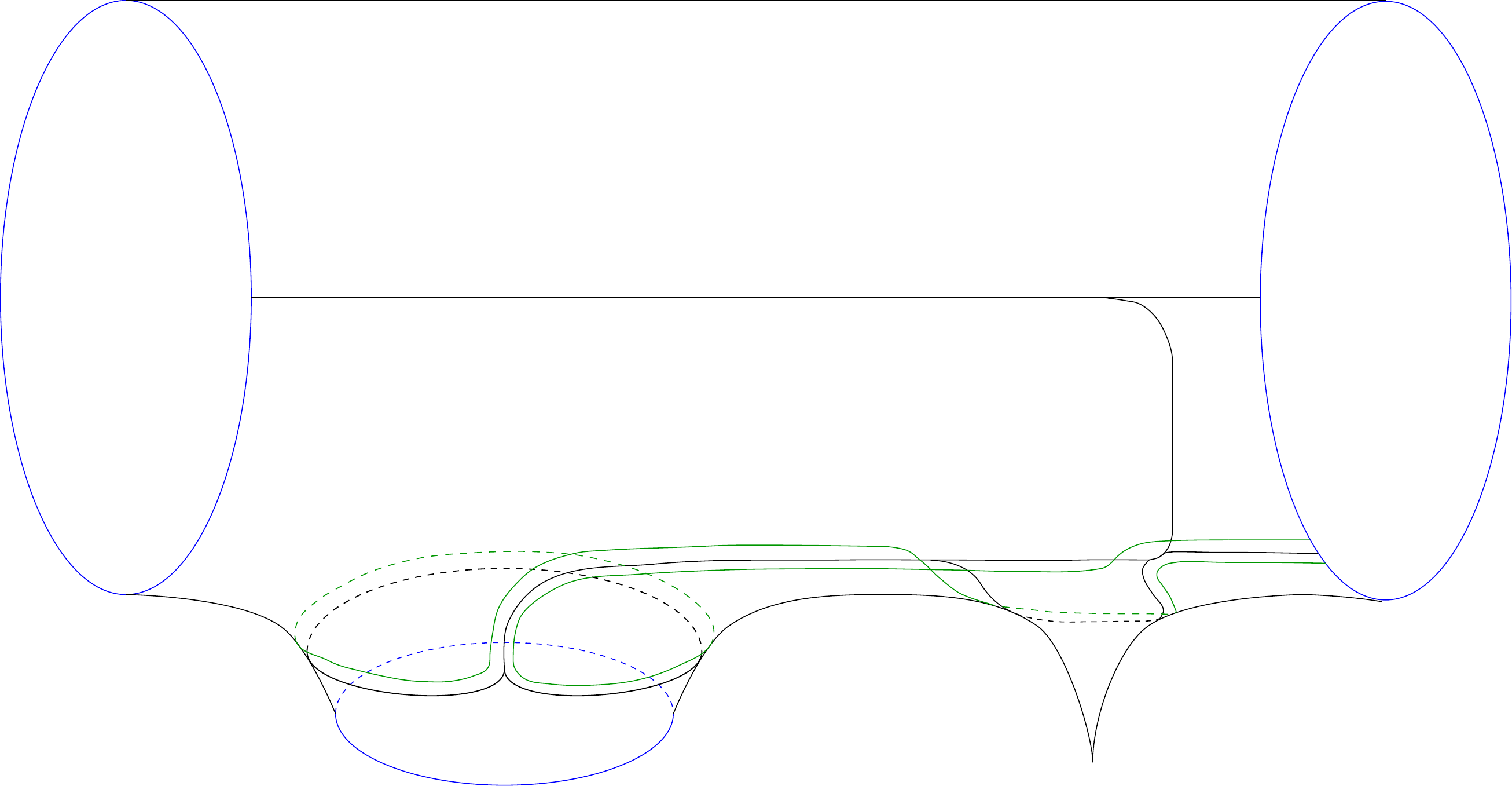
\end{tabular}

\caption{The eight possible train paths through $V$ joining a boundary component to itself.}\label{trainpaths}
\end{figure}

Similarly we analyze train paths through the train tracks with stops $T\cap U_i$. For each $i$, there are three possible train paths through $T\cap U_i$ up to changing the orientation of the path, none of which is homotopic into $\partial U_i$.

These facts imply the following. For any train path $t\in \TP(T,w)$, $t$ consists of a concatenation \[t=\ldots t_{-1} t_0 t_1 \ldots\] where each $t_i$ is a train path through $W_i\cap T$ where $W_i \in \{U_j\}_{j\in \Z}\cup \{V\}$, with endpoints on the boundary. The choice of metric on the pieces of $\Omega$, the isotopic representative chosen for $T$, and the upper bound of 17 on the number of train paths in $T\cap W_i$ imply that there is an upper bound $\kappa$ on the length of each segment $t_i$. Furthermore, no $t_i$ is homotopic into $\partial W_i$. Let \[\tilde{t}=\ldots \tilde{t}_{-1} \tilde{t}_0 \tilde{t}_1 \ldots\] be a lift of $t$ to $\tilde{\Omega}$, where each $\tilde{t}_i$ covers the segment $t_i$. Let $\mathcal{L}$ be the collection of the lifts of the curves $C_i$ on $\Omega$ to $\tilde{\Omega}$. Then there is a lower bound $\eta$ on the distance between any distinct $L_1,L_2\in \mathcal{L}$.

Consider a subpath $s$ of $\tilde{t}$ of length $D$. Then $s$ contains at least $\lfloor D/\kappa-1\rfloor\geq D/\kappa-2$ segments $\tilde{t}_i$. By the fact that no $t_i$ is homotopic into $\partial W_i$, $s$ crosses at least $D/\kappa-2$ distinct elements of $\mathcal{L}$ each of which is distance at least $\eta$ from the next. Hence the distance between the endpoints of $s$ is at least \[(D/\kappa-2)\eta=D\eta/\kappa-2\eta.\] This proves that $\tilde{t}$ is a $(\kappa/\eta,2\eta)$-quasi-geodesic.

\bibliographystyle{plain}
\bibliography{2filling}

\end{document}